\documentclass [leqno, spanish, 10pt]{amsart}

\usepackage[T1]{fontenc}
\usepackage[utf8]{inputenc}

\usepackage{amssymb, amsmath, amsfonts, amsthm, amsbsy, amscd, amsxtra, stmaryrd}
\usepackage[bbgreekl]{mathbbol} 
\usepackage[vcentermath]{youngtab} 
\usepackage[mathscr]{euscript}
\usepackage{upgreek}
\usepackage{mathtools}

\usepackage{etex} 
\usepackage{tikz}
\usetikzlibrary{arrows}
\usepackage[all]{xy}
\usepackage{subcaption}

\usepackage{url}
\usepackage{hyperref}
\hypersetup{colorlinks=false, linktocpage=true, hidelinks 
}

\usepackage{etoolbox}
\apptocmd{\sloppy}{\hbadness 10000\relax}{}{}

\usepackage{thmtools} 
\newtheorem{theorem}{Theorem}[section]
\newtheorem{corollary}[theorem]{Corollary}
\newtheorem{lemma}[theorem]{Lemma}
\newtheorem{definition}[theorem]{Definition}

\theoremstyle{definition}
\newtheorem{example}[theorem]{Example}
\newtheorem{remark}[theorem]{Remark}

\numberwithin{equation}{section}

\def\orb#1{\llbracket#1\rrbracket}
\renewcommand{\div}{\operatorname{div}}
\newcommand{\extreme}{\mathbb{M}}
\newcommand{\ass}{A}
\newcommand{\cass}{C}
\newcommand{\inc}{\imt}
\DeclareMathOperator*{\argmax}{argmax}

\newcommand{\join}{\wedge}

\newcommand{\I}{\mathscr{I}}
\newcommand{\pfaff}{\operatorname{Pfaff}}

\newcommand{\perm}{\operatorname{perm}}
\newcommand{\imm}{\operatorname{imm}}

\newcommand{\ext}{\Lambda}

\newcommand{\rictr}{\operatorname{\rho}}

\newcommand{\Id}{\operatorname{Id}}
\newcommand{\G}{\mathscr{G}}
\newcommand{\C}{\mathscr{C}}

\newcommand{\stw}{\mathbb{W}}
\newcommand{\centroid}{\mathbb{o}}

\renewcommand{\frame}{\mathscr{F}}

\newcommand{\sphere}{\mathbb{S}}

\newcommand{\setn}{\bar{\mathsf{n}}}

\newcommand{\kc}{\mathsf{c}}
\newcommand{\mkc}{\mathsf{m}}
\newcommand{\sn}{\mathbb{S}^{n}}

\newcommand{\rperp}{\lb r \ra^{\perp}}

\newcommand{\id}{\operatorname{Id}}

\newcommand{\scal}{\si}

\newcommand{\crit}{\mathsf{Crit}}
\newcommand{\critline}{\mathsf{CritLine}}
\newcommand{\zero}{\mathsf{Zero}}

\newcommand{\fie}{\mathbb{k}}

\newcommand{\ffie}{\mathbb{F}}
\newcommand{\Om}{\Omega}

\newcommand{\imt}{\iota}
\newcommand{\quat}{\mathbb{H}}

\newcommand{\cayley}{\mathbb{O}}

\newcommand{\so}{\mathfrak{so}}

\newcommand{\alg}{\mathbb{A}}

\newcommand{\balg}{\mathbb{B}}

\newcommand{\pol}{\mathsf{Pol}}
\newcommand{\harpol}{\mathsf{Har}}

\newcommand{\om}{\omega}
\newcommand{\mlt}{\circ}

\newcommand{\hess}{\operatorname{Hess}}
\newcommand{\ka}{\kappa}

\newcommand{\dum}{\,\cdot\,\,}

\newcommand{\lap}{\Delta}
\renewcommand{\j}{\mathsf{i}}
\newcommand{\la}{\lambda}
\newcommand{\ep}{\epsilon}

\newcommand{\cinf}{C^{\infty}}

\newcommand{\eno}{\operatorname{End}}
\newcommand{\si}{\sigma}
\newcommand{\pr}{\partial}

\newcommand{\re}{\text{Re\,}}
\newcommand{\im}{\text{Im\,}}
\newcommand{\B}{\mathcal{B}}
\newcommand{\lb}{\langle}
\newcommand{\ra}{\rangle}

\newcommand{\ste}{\mathbb{V}}

\newcommand{\al}{\alpha}
\newcommand{\be}{\beta}
\newcommand{\ga}{\gamma}
\newcommand{\spn}{\text{Span}\,}

\newcommand{\proj}{\mathbb{P}}

\DeclareMathOperator{\Aut}{Aut}
\newcommand{\g}{\mathfrak{g}}

\newcommand{\tensor}{\otimes}

\newcommand{\rea}{\mathbb R}
\newcommand{\com}{\mathbb C}

\newcommand{\tr}{\operatorname{\mathsf{tr}}}

\let\oldtocsection=\tocsection
\let\oldtocsubsection=\tocsubsection
\renewcommand{\tocsection}[2]{\hspace{0em}\oldtocsection{#1}{#2}}
\renewcommand{\tocsubsection}[2]{\hspace{1em}\oldtocsubsection{#1}{#2}}

\begin{document}
\title[Harmonic cubic homogeneous polynomials]{Harmonic cubic homogeneous polynomials such that the norm-squared of the Hessian is a multiple of the Euclidean quadratic form}
\author{Daniel J.~F. Fox} 
\address{Departamento de Matem\'atica Aplicada a la Ingeniería Industrial\\ Escuela T\'ecnica Superior de Ingenier\'ia y Dise\~no Industrial\\ Universidad Polit\'ecnica de Madrid\\Ronda de Valencia 3\\ 28012 Madrid Espa\~na}
\email{daniel.fox@upm.es}

\begin{abstract}
There is considered the problem of describing up to linear conformal equivalence those harmonic cubic homogeneous polynomials for which the squared-norm of the Hessian is a nonzero multiple of the quadratic form defining the Euclidean metric. Solutions are constructed in all dimensions and solutions are classified in dimension at most $4$. Techniques are given for determining when two solutions are linearly conformally inequivalent.
\end{abstract}

\maketitle

\setcounter{tocdepth}{2} 

\begin{footnotesize}
\tableofcontents
\end{footnotesize}

\section{Introduction}\label{polynomialsection}
Let $\alg$ be an $n$-dimensional vector space equipped with a pseudo-Riemannian metric $h$ parallel with respect to the flat affine connection $D$ on $\alg$ whose geodesics are affinely parameterized straight lines. The space $\pol^{k}(\alg)$ of degree $k$ homogenenous polynomials on $\alg$ is identified with the symmetric power $S^{k}\alg^{\ast}$ via polarization: with $\om \in S^{k}\alg^{\ast}$ is associated $P \in \pol^{k}(\alg)$ defined by $k!P(x) = \om(x, \dots, x)$, and given $P$, $\om$ is recovered via the formula $\om(x(1), \dots, x(k)) = \sum_{1 \leq i_{1} <  \dots <  i_{r} \leq k}(-1)^{k-r}P(x(i_{1}) + \dots + x(i_{r})) = x(1)^{i_{1}}\dots x(k)^{i_{k}}D_{i_{1}}\dots D_{i_{k}}P$ for $x(1), \dots, x(k) \in\alg$. 

This paper considers the problem of classifying, up to approriate notions of equivalence and decomposabislity, homogeneous cubic polynomials $P \in \pol^{3}(\alg)$ solving the equations
\begin{align}\label{einsteinpolynomials}
&\lap_{h} P = 0 ,& &|\hess P|_{h}^{2} = \ka|x|_{h}^{2},
\end{align}
for some constant $\ka \neq 0$. When the metric $h$ is \emph{Euclidean}, meaning it has Riemannian signature, the constant $\ka$ in \eqref{einsteinpolynomials} must be positive. The first equation of \eqref{einsteinpolynomials} simply means that the restriction of $P$ to the unit sphere $S^{n-1}$ is a spherical harmonic with eigenvalue $3(n+1)$.

Solutions to \eqref{einsteinpolynomials} are constructed in all dimensions $n \geq 2$ and they are classified up to orthogonal equivalence for $n \leq 4$ and $h$ of Riemannian signature. Diverse constructions of solutions are given for $n \geq 5$. For example, as is shown in Lemma \ref{isoparametricpolynomiallemma}, it follows essentially immediately from Lemma \ref{einsteinpolynomialslemma}  (this also follows from results in chapter $6$ of \cite{Nadirashvili-Tkachev-Vladuts}) that the cubic isoparametric polynomials of Cartan solve \eqref{einsteinpolynomials}. Sections \ref{framesection} and \ref{steinersection} give combinatorial constructions of solutions based on two-distance tight frames and partial Steiner triple systems. Additionally, techniques are developed for determining when solutions are conformally linearly equivalent.

The notational conventions used in \eqref{einsteinpolynomials} and throughout the paper are explained briefly now. Indices are abstract, and are raised and lowered using the $D$-parallel metric $h_{ij}$ and the inverse symmetric bivector $h^{ij}$. For $F \in \cinf(\alg)$, $F_{i_{1}\dots i_{k}} = D_{i_{1}}\dots D_{i_{k}}F$, and the Hessian of $F$ is indicated by either of the notational synonyms $\hess F$ or $F_{ij}$. The Laplacian $\lap_{h}$ is defined by $\lap_{h} = D^{i}D_{i} = h^{ij}D_{i}D_{j}$. The norms induced by $h_{ij}$ on $\tensor^{k}\alg^{\ast}$ and $\tensor^{k}\alg$ and their submodules are those defined by complete contraction, $|\om|^{2}_{h} = \om^{i_{1}\dots i_{k}}\om_{i_{1}\dots i_{k}}$. Wherever helpful for clarity, the dependence on the choice of metric is indicated with subscripts.

The group $GL(n, \rea)= GL(\alg)$ acts on the left on $\cinf(\alg)$ by $(g \cdot F)(x) = F(g^{-1}x)$. 
Because the action of the orthogonal subgroup $O(h) \subset GL(\alg)$ commutes with $\lap_{h}$ and preserves the norms appearing in \eqref{einsteinpolynomials}, if $P$ solves \eqref{einsteinpolynomials} with constant $\ka$, then $g\cdot P$ solves \eqref{einsteinpolynomials} with constant $\ka$ for all $g \in O(h)$. The group $GL(1, \rea)$ acts on $\alg$ by scaling, $(r\cdot P)(x) = P(r^{-1}x)$, and if $P$ solves \eqref{einsteinpolynomials} with constant $\ka$, then $r\cdot P$ solves \eqref{einsteinpolynomials} with constant $r^{-2}\ka$. Because of the $CO(h)$ invariance of the equations \eqref{einsteinpolynomials}, one wants to classify the orbit $\orb{P} = \{g\cdot P: g \in CO(h)\}$ of a solution $P$ of \eqref{einsteinpolynomials} under the action of the group of conformal linear transformations $CO(h) = GL(1, \rea)\times O(h)$. (Because all affine parallel metrics of a given signature are affinely equivalent, when $h$ is fixed it makes sense to write $O(p, n-p)$ for the orthogonal group $O(h)$ where $p$ is the maximal dimension of a subspace on which $h$ is positive definite.) 

Not all orbits of solutions of \eqref{einsteinpolynomials} are equally interesting in the sense that some are obtained automatically from others via iterative or inductive constructions. There are several senses in which a given solution generates more solutions, and solutions should be classified modulo the corresponding notions of equivalence. These other more subtle notions of equivalence are described in more detail later in the introduction.

\begin{remark}
That $P$ solve \eqref{einsteinpolynomials} means it is an element of the $\binom{n+2}{3} - n = (n+4)n(n-1)/6$-dimensional space $\harpol^{3}_{h}(\alg)$ of $h$-harmonic homogeneous cubic polynomials. The equation $|\hess P|_{h}^{2} = \ka|x|_{h}^{2}$ is a system of $(n+1)n/2$ quadratic equations in the coefficients of the associated tensor $P_{ijk}$ and solutions on the same orbit of the $(n^{2} - n + 2)/2$-dimensional group $CO(h)$ are considered equivalent, so \eqref{einsteinpolynomials} could be studied from the point of view of invariant theory and algebraic geometry. For very general considerations along these lines see \cite{Popov}. Since $\harpol^{3}_{h}(\alg)$ is an irreducible $SO(n)$-module, it admits no invariant linear form. An invariant quadratic polynomial on $\harpol^{3}_{h}(\alg)$ is unique up to a scalar multiple and can be identified with $|\hess P|^{2}_{h}$ (see the introduction to \cite{Bryant-secondorder}), so \eqref{einsteinpolynomials} is essentially the simplest equation on $\harpol^{3}_{h}(\alg)$ determined by invariant theoretic considerations. 
\end{remark}

\begin{example}
Although trivial, the one-dimensional case merits brief mention, as it is the base case of the iterative constructions described in Lemma \ref{triplepolynomiallemma}; see Example \ref{1dseedexample}. On $\rea$, any metric is a multiple of the squared absolute value, and no nontrivial cubic polynomial is harmonic. However, the cubic polynomial $\tfrac{1}{6}x^{3}$ solves the second equation of \eqref{einsteinpolynomials}.
\end{example}

When writing concrete polynomials the following notational conventions are employed. There is written $\rea^{n}$ in place of $\alg$ and $x^{i}$ are coordinates on $\rea^{n} = \alg$ such that the differentials $dx^{i}$ constitute a parallel orthonormal coframe. Polynomials are written in terms of the dual coordinates defined by $x_{i} = x^{p}h_{ip}$ rather than the coordinates $x^{i}$. In particular, this helps avoid confusion between indices and exponents. 

\begin{example}\label{2dharpolyexample}
The most general harmonic cubic polynomial on $P$ on $\rea^{2}$ has the form 
\begin{align}\label{2dharpoly}
\begin{split}
6P(x_{1}, x_{2}) &= r\cos\theta (x_{1}^{3} - 3x_{1}x_{2}^{2}) + r\sin \theta(x_{2}^{3}- 3x_{2}x_{1}^{2}) \\
& =  r (\cos\tfrac{\theta}{3}x_{1} - \sin\tfrac{\theta}{3}x_{2}) ^{3} - 3r(\cos\tfrac{\theta}{3}x_{1} - \sin\tfrac{\theta}{3}x_{2})^{2}(\sin\tfrac{\theta}{3}x_{1} + \cos\tfrac{\theta}{3}x_{2})^{2}  \\
&= r\re(e^{\j\theta/3}z)^{3}, 
\end{split}
\end{align}
for some $r \in [0, \infty)$ and $\theta \in [0, 2\pi)$, where $z = x_{1} + \j x_{2}$. The polynomial \eqref{2dharpoly} solves \eqref{einsteinpolynomials} with $\ka = 2r^{2}$. The orthogonal transformation sending $z$ to $e^{-\j\theta/3}z$ sends $P$ to $\tfrac{r}{6}(x_{1}^{3} - 3x_{1}x_{2}^{2})$, and so every nontrivial solution of \eqref{einsteinpolynomials} on $\rea^{2}$ is equivalent to $x_{1}^{3} - 3x_{1}x_{2}^{2}$ modulo the action of the group $CO(2)$ of linear conformal transformations of the plane. For later use note that the automorphism group of $x_{1}^{3} - 3x_{1}x_{2}^{2}$ is the symmetric group $S_{3}$ acting in its $2$-dimensional irreducible representation by permuting the cube roots of unity.
\end{example}

Theorem \ref{3dpolytheorem}, proved in section \ref{lowdpolysection}, shows that any two solutions of \eqref{einsteinpolynomials} with $\ka = 2c^{2} > 0$ on $\rea^{3}$ are orthogonally equivalent.
\begin{theorem}\label{3dpolytheorem}
For $P \in \pol^{3}(\rea^{3})$ not identically zero and $\kc > 0$ the following are equivalent.
\begin{enumerate}
\item\label{dth2} $P$ solves \eqref{einsteinpolynomials} for a Euclidean metric with parameter $\ka = 2\kc^{2}$.
\item\label{dth3} $P$ is a product of orthogonal homogeneous linear forms. Precisely, $P$ is in the $O(3)$-orbit of 
\begin{align}
\label{basepoly} 
\begin{split}
P(x) &= \kc x_{1}x_{2}x_{3}.
\end{split}
\end{align}
\end{enumerate}
\end{theorem}

Given a few solutions of \eqref{einsteinpolynomials}, it is straightforward to construct new solutions of \eqref{einsteinpolynomials} by taking direct sums. 
The direct sum of Euclidean vector spaces $(\alg_{1}, h_{1})$ and $(\alg_{2}, h_{2})$ is the vector space direct sum $\alg_{1}\oplus \alg_{2}$ equipped with the direct sum metric $h_{1}\oplus h_{2}$ defined by $(h_{1}\oplus h_{2})(a_{1} + a_{2}, b_{1} + b_{2}) = h_{1}(a_{1}, b_{1})h_{2}(a_{2}, b_{2})$ for $a_{i}, b_{i} \in \alg_{i}$, $i = 1, 2$. The direct sum, $P_{1}\oplus P_{2} \in \pol^{k}(\alg_{1}\oplus \alg_{2})$ of $P_{i} \in \pol^{k}(\alg_{i})$, $i = 1, 2$, is defined by $(P_{1} \oplus P_{2})(a_{1} + a_{2}) = P_{1}(a_{1}) + P_{2}(a_{2})$ for $a_{i} \in \alg_{i}$, $i = 1, 2$. It follows that $(\hess (P_{1}\oplus P_{2})) = (\hess P_{1})\oplus (\hess P_{2})$ as elements of $S^{2}(\alg_{1}\oplus \alg_{2})$, so that $(P_{1}\oplus P_{2})$ solves \eqref{einsteinpolynomials} with constant $\ka$ if $P_{i} \in \pol^{k}(\alg_{i})$, $i = 1, 2$, solve \eqref{einsteinpolynomials} with constant $\ka$. (Note that given two solutions of \eqref{einsteinpolynomials}, the equality of the corresponding constants can always be arranged by rescaling $h_{1}$ or $h_{2}$.)

\begin{example}\label{splitexample}
If $P \in \pol^{3}(\rea^{p})$ and $Q \in \pol^{3}(\rea^{q})$ solve \eqref{einsteinpolynomials} for the same constant $\ka$, then $P \circ \pi_{p} + Q \circ \pi_{q} \in \pol^{3}(\rea^{p+q})$, in which $\pi_{p}$ and $\pi_{q}$ are the projections from $\rea^{p+q}= \rea^{p}\oplus \rea^{q}$ onto $\rea^{p}$ and $\rea^{q}$, solves \eqref{einsteinpolynomials}. Since any harmonic $P \in \pol^{3}(\rea^{2})$ satisfies \eqref{einsteinpolynomials}, there are solutions to \eqref{einsteinpolynomials} on $\rea^{2n}$ for any $n \geq 2$. Concretely, by Example \ref{2dharpoly} and Theorem \ref{3dpolytheorem}, for $p, q \geq 0$ such that $2p + 3q > 0$,
\begin{align}
P(x) =  \tfrac{1}{6}\sum_{i = 1}^{p}(x_{2i-1}^{3} - 3x_{2i-1}x_{2i}^{2}) + \sum_{j = 1}^{q}x_{2p + 3j-1}x_{2p + 3j-2}x_{2p + 3j} 
\end{align}
solves \eqref{einsteinpolynomials} on $\rea^{2p + 3q}$ with $\ka = 2$. 
See also Example \ref{lanminusoneexample} where, as a special case of Lemma \ref{preparationlemma}, an $(n+2)$-dimensional solution is obtained from an $n$-dimensional solution in a similar way.
\end{example}

A polynomial $P\in \pol^{k}(\alg)$ is \textit{linearly decomposable} if there is a nontrivial linear direct sum $\alg = \alg_{1}\oplus \alg_{2}$ such that $P = (P\circ\pi_{1}\oplus P\circ\pi_{2})$, where $\pi_{i}$ is the projection onto $\alg_{i}$ along its complement. Equivalently, if $x_{i} = \pi_{i}(x)$ is the component of $x$ in $\alg_{i}$, $P(x_{1} + x_{2}) = P(x_{1}) + P(x_{2})$. In this case, there is a unique polynomial $P_{i} \in \pol^{k}(\alg_{i})$ such that $P_{i}\circ \pi_{i} = P\circ \pi_{i}$. A polynomial $P \in \pol^{k}(\alg)$ is \textit{indecomposable} if it is not decomposable. 

If $\alg$ is equipped with a metric, $h$, then $P\in \pol^{k}(\alg)$ is \textit{$h$-orthogonally decomposable} if $P$ decomposes along a nontrivial direct sum $\alg = \alg_{1} \oplus \alg_{2}$ into $h$-orthogonal subspaces $\alg_{1}, \alg_{2} \subset \alg$. In this case $P$ \textit{decomposes $h$-orthogonally along the $h$-orthogonal subspaces  $\alg_{i}$ or along the $h$-orthogonal direct sum $\alg = \alg_{1} \oplus \alg_{2}$}. A polynomial $P \in \pol^{k}(\alg)$ is \textit{$h$-orthogonally indecomposable} if it is not $h$-orthogonally decomposable. 

\begin{example}\label{2ddecomposableexample}
The polynomial $P(x, y) = xy$ on $\rea^{2}$ is Euclidean orthogonally decomposable, along the diagonal $\alg_{1} = \{(x, x) \in \rea^{2}$ and the antidiagonal $\alg_{2} = \{(x, -x) \in \rea^{2}\}$, for 
\begin{align}
P(x, y) = xy = \tfrac{1}{4}(x + y)^{2} - \tfrac{1}{4}(x - y)^{2} = P \circ \pi_{1}(x, y) + P\circ \pi_{2}(x, y).
\end{align}
The harmonic polynomial $P(u, v) = u^{3} - 3uv^{2}$ is not Euclidean orthogonally decomposable. Were it, the decomposition would be by orthogonal one-dimensional subspaces, and there would be a rotation $g \in O(2)$ such that $(g \cdot P)(x, y)$ had the form $ax^{3} + by^{3}$ for some nonzero constants $a$ and $b$. However $P$ is harmonic, so $g \cdot P$ would be as well, which can be only if $a = 0 = b$.

Linear decomposability depends on the base field. For example, although this will not be proved here, $u^{3} - 3uv^{2}$ is not linearly decomposable over $\rea$, although it is linearly decomposable over $\com$, for $u^{3} - 3uv^{2} = \tfrac{1}{2}( u + \j v)^{3} + \tfrac{1}{2}(u - \j v)^{3}$.
\end{example}
\begin{example}\label{poly3example}
It may not be obvious if an explicitly given polynomial is orthogonally decomposable.
For example, the imaginary part of $\tfrac{1}{2}(x_{3} + \j x_{1})^{2}(x_{2} + \j x_{4}) - \tfrac{1}{6}(x_{2} + \j x_{4})^{3}$ is
\begin{align}
\label{poly3} 
\begin{split}
P(x) &= \tfrac{1}{6}x_{4}^{3} - \tfrac{1}{2}x_{4}(x_{1}^{2} + x_{2}^{2} - x_{3}^{2}) + x_{1}x_{2}x_{3}
\end{split}
\end{align}
and solves \eqref{einsteinpolynomials} with $\ka = 4$. It is equivalent modulo $O(4)$ to $\tfrac{2\sqrt{2}}{12}\left(x_{1}^{3} - 3x_{1}x_{2}^{2} + x_{3}^{3} - 3x_{3}x_{4}^{2} \right)$, for
\begin{align}
\label{poly3decompositiona} 
\begin{split}
P(x) &= \tfrac{1}{6}x_{4}^{3} - \tfrac{1}{2}x_{4}(x_{1}^{2} + x_{2}^{2} - x_{3}^{2}) + x_{1}x_{2}x_{3}\\
& = \tfrac{1}{6}\left(  (x_{4}^{3} - 3x_{4}x_{1}^{2}) + (x_{4}^{3} - 3x_{4}x_{2}^{2}) - (x_{4}^{3} - 3x_{4}x_{3}^{2}) + 6x_{1}x_{2}x_{3}\right)\\
& = \tfrac{2\sqrt{2}}{12}\left(u_{1}^{3} - 3u_{1}u_{2}^{2} + v_{1}^{3} - 3v_{1}v_{2}^{2} \right),
\end{split}
\end{align}
where
\begin{align}\label{poly3decompositionb}
\begin{split}
\sqrt{2}u_{1} & = -\tfrac{\sqrt{3}}{2}x_{1} + \tfrac{\sqrt{3}}{2}x_{2} - \tfrac{1}{2}x_{3} - \tfrac{1}{2}x_{4},\qquad
\sqrt{2}u_{2}  = -\tfrac{1}{2}x_{1} + \tfrac{1}{2}x_{2} + \tfrac{\sqrt{3}}{2}x_{3} + \tfrac{\sqrt{3}}{2}x_{4},\\
\sqrt{2}v_{1} & = -\tfrac{\sqrt{3}}{2}x_{1} - \tfrac{\sqrt{3}}{2}x_{2} + \tfrac{1}{2}x_{3} - \tfrac{1}{2}x_{4},\qquad
\sqrt{2}v_{2}  = -\tfrac{1}{2}x_{1} - \tfrac{1}{2}x_{2}-  \tfrac{\sqrt{3}}{2}x_{3} + \tfrac{\sqrt{3}}{2}x_{4}.
\end{split}
\end{align}
The decomposition \eqref{poly3decompositiona}-\eqref{poly3decompositionb} shows that \eqref{poly3} is orthogonally decomposable as it exhibits this polynomial as a sum of two variable polynomials defined on orthogonal subspaces. 
\end{example}
Solutions of \eqref{einsteinpolynomials} that are orthogonally decomposable as in Examples \ref{splitexample} and \ref{poly3example} are less interesting than orthogonally indecomposable solutions as they are obtained by combining solutions in lower dimensions. 
It is of primary interest to classify orbits of solutions of \eqref{einsteinpolynomials} that are orthogonally indecomposable. 

For $n \leq 4$, the equations \eqref{einsteinpolynomials} can be solved completely via direct elementary computations, and this is done in section \ref{directsolutionsection}. By Example \eqref{2dharpolyexample} and Theorem \ref{3dpolytheorem}, in dimensions $2$ and $3$ and Riemannian signature there is a unique solution up to conformal equivalence. By Example \ref{2ddecomposableexample} and the $n = 3$ case of Theorem \ref{ealgpolynomialtheorem} these solutions are orthogonally indecomposable. 

 The $4$-dimensional case requires a more involved analysis. The first dimension in which there is more than one orbit of solutions is $n = 4$. in this case there are exactly two solutions, up to conformal equivalence. One of the orbits is orthogonally indecomposable, and the other comprises the direct sums of $2$-dimensional solutions.
Theorem \ref{4dpolytheorem} is proved in section \ref{lowdpolysection}.
\begin{theorem}\label{4dpolytheorem}
If $P \in \pol^{3}(\rea^{4})$ solves \eqref{einsteinpolynomials} for a Euclidean metric with constant $\ka > 0$, then $P$ is $O(4)$-equivalent to exactly one of the polynomials 
\begin{align}
\label{minusonethird0}
\begin{split}
 \tfrac{\sqrt{3\ka}}{12}&\left(-\tfrac{1}{3}(3x_{4}x_{1}^{2} - x_{4}^{3}) -\tfrac{1}{3}(3x_{4}x_{2}^{2} - x_{4}^{3}) -\tfrac{1}{3}(3x_{4}x_{3}^{2} - x_{4}^{3}) + 2\sqrt{5}x_{1}x_{2}x_{3}\right)\\
& =  \tfrac{\sqrt{3\ka}}{12}\left(x_{4}^{3} - x_{4}(x_{1}^{2} + x_{2}^{2} + x_{3}^{2}) + 2\sqrt{5}x_{1}x_{2}x_{3}\right). 
\end{split}\\
\label{lazero}
\begin{split}
&  \tfrac{\sqrt{2\ka}}{12}\left(x_{4}^{3} - 3 x_{4}x_{3}^{2} +3x_{1}^{2}x_{2} - x_{2}^{3} \right),
\end{split}
\end{align}
which are not orthogonally equivalent.
\end{theorem}

\begin{corollary}\label{4dindecomposablecorollary}
An orthogonally indecomposable $P \in \pol^{3}(\rea^{4})$ solves \eqref{einsteinpolynomials} for a Euclidean metric with constant $\ka > 0$ if and only if $P$ is $O(4)$-equivalent to 
\eqref{minusonethird0}. The automorphism group of $P$ is the symmetric group $S_{5}$.
\end{corollary}
As \eqref{lazero} is evidently orthogonally decomposable, Corollary \ref{4dindecomposablecorollary} follows from Theorem \ref{minusonethird0} once it is shown that \eqref{minusonethird0} is orthogonally indecomposable. Although a direct proof of the orthogonal indecomposability is given at the end of Section \ref{directsolutionsection}, it also follows from Theorem \ref{ealgpolynomialtheorem}.

Lemma \ref{recursivepolynomiallemma} shows that a solution of \eqref{einsteinpolynomials} in dimension $n$ determines a solution in dimension $n+1$. This is used to establish that orthogonally indecomposable solutions of \eqref{einsteinpolynomials} exist in all dimensions. A solution $P \in \pol^{3}(\alg)$ of \eqref{einsteinpolynomials} is \emph{conformally associative} if 
\begin{align}
2P_{l[i}\,^{p}P_{j]kp} = \tfrac{2\ka}{n-1}h_{k[i}h_{j]l}.
\end{align}
This notion is motivated and explained in more detail in Section \ref{associativitysection}.
\begin{theorem}\label{ealgpolynomialtheorem}
On a Euclidean vector space $(\alg, h)$ of dimension $n \geq 2$, there is a unique $CO(h)$ orbit $\orb{P_{n}}\subset\pol^{3}(\alg)$ comprising conformally associative solutions to \eqref{einsteinpolynomials} and represented by the polynomial
\begin{align}\label{simplicialpoly}
P_{n}(x) & = \tfrac{\sqrt{n(n+1)}}{6}\sum_{j = 2}^{n}\tfrac{1}{\sqrt{j(j+1)}}\sum_{ i = 1}^{j-1}(x_{j}^{3} - 3x_{j}x_{i}^{2}),
\end{align}
that solves \eqref{einsteinpolynomials} with constant $\ka = n(n-1)$.
Any $P \in \orb{P_{n}}$ is orthogonally indecomposable and its automorphism group is the symmetric group $S_{n+1}$ acting as reflections through the hyperplanes orthogonal to lines on which $P_{n}$ vanishes and permuting the critical points of the restriction of $P_{n}$ to the unit sphere.
\end{theorem}
The proof of Theorem \ref{ealgpolynomialtheorem} is given following the proof of Lemma \ref{recursivepolynomiallemma}. See in particular Corollary \ref{ndimexistencecorollary}. In dimensions $2$, $3$, and $4$, the orbits described in Theorem \ref{ealgpolynomialtheorem} are generated by constant multiples of the polynomials $\tfrac{1}{6}(x_{1}^{3} -3x_{1}x_{2}^{2})$, \eqref{basepoly}, and \eqref{minusonethird0}.

The polynomial $P_{n}$ of \eqref{simplicialpoly} (or any polynomial in the orbit $\orb{P_{n}}$) is called the \emph{simplicial polynomial} of dimension $n$. 

\begin{remark}
The polynomials $P_{n}$ of \eqref{simplicialpoly} appear in the proof of Corollary $22$ of \cite{Andrews-gauss} where they are given (modulo slight changes in normalizations) in the equivalent recursive form \eqref{pnrecursion}. They are used to construct examples of highly symmetric nonspherical homothetically contracting solutions of certain Gauss curvature flows. In \cite{Andrews-gauss} it is stated without proof that these polynomials have the symmetries of a regular simplex and this suggests that this construction was known classically, but I do not know a reference. Here this is proved together with many other interesting properties of the polynomials $P_{n}$ as part of Corollary \ref{ndimexistencecorollary}. 

An alternative construction of the polynomials $P_{n}$, that justifies calling them \emph{simplicial} and from which their $S_{n+1}$ invariance is obvious is given in Example \ref{simplicialframepolynomialexample}, using the machinery of equiangular tight frames discussed in Section \ref{framesection}. The uniqueness part of Theorem \ref{ealgpolynomialtheorem} provides the most effective way to check that different constructions of $P_{n}$ really do yield conformally equivalent polynomials, as in practice conformal associativity can be checked.
\end{remark}

\begin{corollary}[Corollary of Theorems \ref{3dpolytheorem} and \ref{4dpolytheorem}]\label{autocorollary}
If a solution $P \in \pol^{3}(\rea^{n})$ of \eqref{einsteinpolynomials} for a Euclidean metric and constant $\ka > 0$ has automorphism group of positive dimension, then $n \geq 5$.
\end{corollary}

In section \ref{isoparametricsection} it is shown that the cubic isoparametric polynomials of Cartan solve \eqref{einsteinpolynomials}. These polynomials are defined for $n \in \{5, 9, 14, 26\}$ and their automorphism groups are Lie groups, so this shows that the result of Corollary \ref{autocorollary} is the best possible.  

A metric and a covariant cubic symmetric tensor determine the structure tensor of a commutative multiplication $\mlt$ on $\alg$. Its multiplication operators $L:\alg \to \eno(\alg)$ defined by $L(x)y = x\mlt y$ are determined by the Hessian of $P$, $L(x)_{i}\,^{j} = h^{jp}P(x)_{ip}$. For $P$ solving \eqref{einsteinpolynomials}, the $h$-harmonicity of $P$ is equivalent to $\tr L(x) = 0$ for all $x \in \alg$, while the condition on $\hess P$ and the complete symmetry of $P$ are equivalent to the nondegeneracy and invariance of the Killing type trace-form $\tau(x, y) = \tr L(x)L(y)$, where invariance means $\tau(x \mlt y, z) = \tau(x, y \mlt z)$. Since $CO(h)$-equivalence classes of solutions of \eqref{einsteinpolynomials} correspond bijectively to isomorphism classes of such algebras, the search for solutions of \eqref{einsteinpolynomials} can be profitably converted into a problem in pure algebra. The utility of this algebraic point of view has been emphasized in the study of some different but closely related problems by V.~G. Tkachev and his collaborators, for example \cite{Nadirashvili-Tkachev-Vladuts, Tkachev-generalization, Tkachev-correction, Tkachev-jordan, Tkachev-summary, Tkachev-universality}. For example, this correspondence was used by N. Nadirashvili, V.~G. Tkachev, and S. Vl\u{a}du\c{t} in \cite{Nadirashvili-Tkachev-Vladuts} to construct homogeneous and singular solutions of the equations for minimal cones and other geometrically motivated fully nonlinear elliptic partial differential equations of Hessian type. 

\begin{example}\label{parahurwitzexample}
The para-Hurwitz algebra is the complex numbers $\com$, viewed as a real Euclidean vector space, with the multiplication $z_{1} \mlt z_{2} = \bar{z}_{1}\bar{z}_{2}$. Writing $z = x_{1} + \j x_{2}$, the matrix of the multiplication operator $L_{\mlt}(z)$ with respect to the standard Euclidean orthonormal basis is 
\begin{align}
L_{\mlt}(z) = \begin{pmatrix}
x_{1} & - x_{2}\\-x_{2} & -x_{1}
\end{pmatrix}
\end{align}
so that $\tau_{\mlt}(z, z) = \tr L_{\mlt}(z)^{2} = 2|z|^{2}$ and the cubic polynomial $P(z)$ determined by $\mlt$ and the Euclidean metric is given by 
\begin{align}
6P(z) = \lb z\mlt z, z\ra = \tfrac{1}{2}(z^{3} + \bar{z}^{3}) = \re z^{3} = x_{1}^{3} - 3x_{1}x_{2}^{2}, 
\end{align}
which recovers the solution of \eqref{einsteinpolynomials} given in Example \ref{2dharpolyexample}.
\end{example}

\begin{example}
As will be explained in detail elsewhere, the commutative nonassociative algebra associated with the polynomial $P_{n}$ of \eqref{simplicialpoly} of Theorem \ref{ealgpolynomialtheorem} is isomorphic to that studied by K. Harada in \cite{Harada} where it is proved that its automorphism group is $S_{n+1}$ (this was proved independently by R. Griess; see the appendix to \cite{Dong-Griess}).  
\end{example}

As an example of the utility of the algebraic point of view, the observation that the cubic polynomial $P$ of $(\alg, \mlt, h)$ is orthogonally indecomposable if and only if $(\alg, \mlt)$ is simple as an algebra provides a criterion for checking orthogonal indecomposability of $P$ that is useful in practice. Similarly, critical points of $P$ on spheres correspond with idempotent and square-zero elements of the associated algebra; since such elements are permuted by automorphisms of the algebra, much information about the automorphisms of $P$ can be obtained from their study. 

Although the algebraic perspective provides a lot of useful information about \eqref{einsteinpolynomials}, in order to keep the scope of the present article manageable, its applications to \eqref{einsteinpolynomials} will be treated elsewhere as part of a collaboration with V.~G. Tkachev. The only place in the paper where the algebraic correspondence is used is in the proof of claim \eqref{simp8} of Corollary \ref{ndimexistencecorollary}, where it is used to save space by using results from \cite{Harada}.

\begin{example}\label{decomposableexample}
The polynomial \eqref{d2poly2} 
\begin{align}
\label{d2poly2}
\begin{split}
P(x) & = x_{1}x_{2}x_{3} + x_{1}x_{4}x_{5} + x_{2}x_{4}x_{6} + x_{3}x_{5}x_{6} \\
&=  \tfrac{1}{2}(x_{1} + x_{6})(x_{2} + x_{5})(x_{3} + x_{4}) + \tfrac{1}{2}(x_{1} - x_{6})(x_{2} - x_{5})(x_{3} - x_{4}).
\end{split}
\end{align}
solves \eqref{einsteinpolynomials} on $\rea^{6}$ with the Euclidean metric and $\ka = 4$. 

It is apparent from the decomposition exhibited in \eqref{d2poly2} that $P$ orthogonally decomposable; precisely it is equivalent to a direct sum of polynomials as in Theorem \ref{3dpolytheorem}. As this example (and Example \ref{poly3example}) suggests, determining whether a solution of \eqref{einsteinpolynomials} is decomposable is not always straightforward; in general some technique is needed. Although the details are omitted here, he decomposition \eqref{d2poly2} was obtained by analyzing idempotent elements in the commutative algebra determined by $P$.
\end{example}

In section \ref{tensorproductsection} it is shown that tensor products of solutions of \eqref{einsteinpolynomials} yield new solutions. Lemma \ref{parahurwitzificationlemma} shows that the tensor product of a solution $P$ with the polynomial of \eqref{2dharpoly}, called the \emph{parahuwitzification} of $P$, can be interpreted as the real part of the extension of $P$ over $\com$. Lemma \ref{triplepolynomiallemma} shows that the tensor product of a solution $P$ with the polynomial of \eqref{basepoly}, called the \emph{triple} of $P$, equals that obtained by polarizing $P$. 
For example:
\begin{enumerate}
\item The polynomial 
\begin{align}
\label{tripleparahurwitz}
\begin{split}
P(x) & = x_{1}x_{3}x_{5} - x_{1}x_{4}x_{6} - x_{2}x_{3}x_{6} - x_{2}x_{4}x_{5} ,
\end{split}
\end{align}
is obtained either as the triple of the polynomial $\tfrac{1}{6}(x_{2}^{3}- 3x_{1}^{2}x_{2})$ or as the parahurwitzification of the polynomial $x_{1}x_{2}x_{3}$. It follows that it is orthogonally indecomposable, and so is not orthogonally equivalent to the polynomial \eqref{d2poly2}, from which it differs only by sign changes. By Lemma \ref{pmnlemma}, \eqref{tripleparahurwitz} is not conformally equivalent to the simplicial polynomial $P_{6}$.

\item The polynomial \eqref{permpoly} 
\begin{align}
\label{permpoly}
\begin{split}
P(x)  & = \perm \begin{pmatrix} x_{1} & x_{2} & x_{3} \\ x_{4} & x_{5} & x_{6}\\ x_{7} & x_{8} & x_{9}\end{pmatrix}\\
& = x_{1}x_{5}x_{9} + x_{2}x_{6}x_{7} + x_{3}x_{4}x_{8} + x_{1}x_{6}x_{8} + x_{2}x_{4}x_{9} + x_{3}x_{5}x_{7},
\end{split}
\end{align}
is obtained as the triple of \eqref{basepoly}. By Lemma \ref{pmnlemma}, \eqref{permpoly} is not conformally equivalent to the simplicial polynomial $P_{9}$.
\end{enumerate}

A special case of Lemma \ref{anticommutativetensortheorem} shows that the cubic polynomial of the tensor product of two compact simple real Lie algebras solves \eqref{einsteinpolynomials}. For example the tensor product $\so(3) \tensor \so(3)$ yields the determinant
\begin{align}
\label{nahmpolyso3}
\begin{split}
P(x)  & = \det \begin{pmatrix} x_{1} & x_{2} & x_{3} \\ x_{4} & x_{5} & x_{6}\\ x_{7} & x_{8} & x_{9}\end{pmatrix}\\
& = x_{1}x_{5}x_{9} + x_{2}x_{6}x_{7} + x_{3}x_{4}x_{8} - x_{1}x_{6}x_{8} - x_{2}x_{4}x_{9} - x_{3}x_{5}x_{7},
\end{split}
\end{align}
which solves \eqref{einsteinpolynomials}. Lemma \ref{anticommutativetensortheorem} applies more generally to certain anticommutative algebras, for example the imaginary octonions with the commutator bracket; see Example \ref{imagoctonionexample}.

\begin{remark}
The  immanant determined by the character of $S_{3}$ associated with the partition $(2, 1)$ (see \cite{Ye-immanants}) is
\begin{align}
\label{immpoly}
\begin{split}
P(x)  & = \imm_{(2, 1)}\begin{pmatrix} 2^{-1/3}x_{1} & x_{2} & x_{3} \\ x_{4} & 2^{-1/3}x_{5} & x_{6}\\ x_{7} & x_{8} & 2^{-1/3}x_{9}\end{pmatrix}\\
& = x_{1}x_{5}x_{9} - x_{2}x_{6}x_{7} - x_{3}x_{4}x_{8}.
\end{split}
\end{align}
It solves \eqref{einsteinpolynomials} with $\ka = 2$ (the factors of $2^{-1/3}$ are so that $|\hess P|^{2}$ is a multiple of the standard Euclidean inner product in the $x_{i}$ variables). The other immanants associated with characters of $S_{3}$ are the determinant and the permanent, so that \eqref{permpoly} and \eqref{nahmpolyso3} solve \eqref{einsteinpolynomials} suggests that \eqref{immpoly} will as well. However, as \eqref{immpoly} is orthogonally decomposable, obtained from three copies of \eqref{basepoly}, it is less interesting than the examples \eqref{permpoly} and \eqref{nahmpolyso3}. 
\end{remark}

Sections \ref{steinersection} and \ref{framesection} describe combinatorial constructions of solutions of \eqref{einsteinpolynomials} based on Steiner triple systems and equiangular tight frames. For example, as is indicated in Example \ref{dualaffineplaneexample}, the solution \eqref{d2poly2} is associated with the partial Steiner triple system determined by the affine plane of order $2$ in the manner explained in Section \ref{steinersection}, while as in indicated in Example \ref{d2poly2frameexample}, it is associated with two different two-distance tight frames in $\rea^{6}$ in the manner described in Section \ref{framesection}. 

These constructions are of interest principally for two reasons. First, Theorem \ref{ststheorem} shows that the solution of \eqref{einsteinpolynomials} associated with a Steiner triple system on a set of cardinality $n$ equal to $1$ or $9$ modulo $12$ is not conformally linearly equivalent to the simplicial polynomial $P_{n}$. This shows that not conformally associative solutions to \eqref{einsteinpolynomials} exist in arbitrarily large dimensions. Second, the construction of Section \ref{framesection} associates with a two-distance tight frame a solution of \eqref{einsteinpolynomials}. For example, in Example \ref{simplicialframepolynomialexample} this is used to construct the simplicial polynomials $P_{n}$ from the data of their critical lines (in the sense of section \ref{criticalsection}). Although the problem of reconstructing a solution of \eqref{einsteinpolynomials} from data such as its zero lines and critical lines is by no means solved, these results are suggestive for what might be possible. 
A virtue of such combinatorial constructions is that the resulting solution of \eqref{einsteinpolynomials} automatically inherits any symmetries of the underlying combinatorial data.

The constructions of solutions described in this paper by no means exhaust those known to the author and some of the most interesting examples are omitted. However, the description of most of the omitted examples requires or is facilitated by working in terms of the commutative nonassociative algebra associated with a solution of \eqref{einsteinpolynomials} and an adequate development of the necessary background requires extensive development. Consequently these examples will be detailed elsewhere in that language. Some of the more interesting examples include:
\begin{itemize}
\item The cubic polynomials of \emph{Griess algebras} of certain vertex operator algebras of type OZ. That these solve \eqref{einsteinpolynomials} can be deduced using formulas in \cite{Matsuo}. These include the cubic polynomial of the $196883$-dimensional commutative, nonassociative, nonunital Griess algebra the automorphism group of which is the monster finite simple group.
\item The cubic polynomials of the algebras of Weyl curvature tensors and Kähler Weyl curvature tensors studied by the author in \cite{Fox-curvtensor}. 
\item The cubic polynomials of the \emph{Hsiang algebras} studied by Tkachev in \cite{Tkachev-hsiang} (see also \cite{Nadirashvili-Tkachev-Vladuts, Tkachev-universality, Tkachev-summary}).
\end{itemize}

\begin{remark}
The author's original motivation for studying \eqref{einsteinpolynomials} in \cite{Fox-ahs} (see also \cite{Fox-crm, Fox-2dahs}) was that, by Lemma \ref{einsteinpolynomialslemma}, a solution of \eqref{einsteinpolynomials} yields a torsion-free affine connection on $\alg$ with respect to which the metric $h$ is a Codazzi tensor, and the Ricci curvature of which is a multiple of $h$ (so $\nabla$ is in some sense \emph{Einstein-like}). These motivations, which will otherwise be ignored here, are treated in a far more general context in \cite{Fox-ahs}.

\begin{lemma}\label{einsteinpolynomialslemma}
Let $D$ be the flat affine connection on the $n$-dimensional vector space $\alg$ whose geodesics are the affinely parameterized straight lines and let $h$ be a $D$-parallel pseudo-Riemannian metric on $\alg$.
The systems of equations \eqref{einsteinpolynomials0}-\eqref{einsteinpolynomials2} for a homogeneous cubic polynomial $P(x) \in \pol^{3}(\alg)$ are equivalent in the sense that $P$ solves one of these systems for a constant $\ka \neq 0$ if and only if it solves all the others for the same constant $\ka \neq 0$. 
\begin{align}\label{einsteinpolynomials0}
&\lap_{h} P = 0,& &\text{and}& &|\hess P|^{2}_{h} = \ka|x|^{2}_{h},
\\\label{einsteinpolynomialsa}
&\lap_{h} P = 0,& &\text{and}& &\lap_{h}|DP|_{h}^{2} = 2\ka |x|^{2}_{h},
\\\label{einsteinpolynomialsb}
&\lap_{h} P = 0,& &\text{and}& &\lap_{h}^{2} P^{2} = 4\ka |x|^{2}_{h},
\\\label{einsteinpolynomials2}
&P_{ip}\,^{p} = 0,& &\text{and}& &P_{ip}\,^{q}P_{jq}\,^{p} = \ka h_{ij}.
\end{align}
If $P$ solves any of these equivalent systems of equations then the affine connection $\nabla = D + P_{ij}\,^{k}$ has the following properties:
\begin{enumerate}
\item $\nabla$ is torsion-free.
\item $\nabla_{[i}h_{j]k} = 0$.
\item The curvature of $\nabla$ is $R_{ijk}\,^{l} = 2P_{p[i}\,^{l}P_{j]k}\,^{p}$ and satisfies $R_{ij(kl)} = 0$.
\item The Ricci curvature $R_{ij} = R_{pij}\,^{p}$ satisfies $R_{ij} = -\ka h_{ij}$.
\end{enumerate} 
\end{lemma}

\begin{proof}[Proof of Lemma \ref{einsteinpolynomialslemma}]
 It is straightforward to check that if $\lap_{h}P = 0$ then $\lap_{h}^{2}P^{2} = 2\lap_{h}|DP|_{h}^{2} = 4|\hess P|^{2}_{h}$. This shows the equivalence of \eqref{einsteinpolynomials0}, \eqref{einsteinpolynomialsa}, and \eqref{einsteinpolynomialsb}. Contracting the second equation of \eqref{einsteinpolynomials2} with $x^{i}x^{j}$ gives the second equation of \eqref{einsteinpolynomials0}, while if $P$ solves \eqref{einsteinpolynomials0} then $0 = D_{i}D_{j}(|\hess P|^{2}_{h} - \ka |x|^{2}_{h}) = 2(P_{ip}\,^{q} P_{jq}\,^{p} - \ka h_{ij})$, showing that $P$ solves \eqref{einsteinpolynomials2}.

Suppose $P$ solves  \eqref{einsteinpolynomials0}-\eqref{einsteinpolynomials2}. By definition of $\nabla$, $\nabla_{i}h_{jk} = -P_{ijk} - P_{ikj} = -2P_{ijk}$, so $\nabla_{[i}h_{j]k} = 0$. The curvature of $\nabla$ is defined by defined by $2\nabla_{[i}\nabla_{j]}X^{k} = R_{ijp}\,^{k}X^{p}$. From the definition of $\nabla$ it follows that $R_{ijk}\,^{l} = 2P_{p[i}\,^{l}P_{j]k}\,^{p}$, from which there follows $R_{ij(kl)} = 0$. The Ricci curvature of $\nabla$ is defined by $R_{ij} = R_{pij}\,^{p} = R_{ip}\,^{p}\,_{j}$ and so $R_{jk} = P_{qp}\,^{q}P_{jk}\,^{p} -P_{jp}\,^{q}P_{kq}\,_{p}= -P_{jp}\,^{q}P_{kq}\,_{p} = - \ka h_{ij}$.
\end{proof}
\end{remark}

\begin{remark}
Suppose $h$ is Euclidean. The components with respect to an $h$-orthonormal basis of the differential $Q = dP:\alg \to \alg^{\ast}$ of a solution of \eqref{einsteinpolynomials} are $h$-harmonic homogeneous quadratic polynomials. This leads to at least two alternative viewpoints on \eqref{einsteinpolynomials}.

First, the differential $DQ = DdP:\alg \to S^{2}\alg^{\ast}$ can be identified with a linear map of $\alg$ into the subspace $S^{2}_{0}\alg^{\ast} \subset S^{2}\alg^{\ast}$ of $h$-trace free forms. That $P$ solve \eqref{einsteinpolynomials} means that this map sends the $h$-unit sphere in $\alg$ into the sphere in $S^{2}_{0}\alg^{\ast}$ having radius $\sqrt{\ka}$ with respect to the norm induced by complete contraction with $h$. Thus a solution of \eqref{einsteinpolynomials} yields a polynomial map between spheres whose components are the restrictions of linear forms. Such a map is called a linear eigenmap of spheres; see e.g. \cite{doCarmo-Wallach, Smith-harmonicmappings, Toth-eigenmaps, Wood} for background about polynomial maps between spheres. In the context of polynomial eigenmaps of spheres, those given by linear forms are regarded as trivial, but here $Q$ satisfies the additional auxiliary conditions imposed by requiring that it be a differential.

Second, the graph of a polynomial mapping $Q:\alg \to \alg^{\ast}$ is a submanifold $\Sigma$ of the para-Kähler space $\alg \oplus \alg^{\ast}$, where the symplectic form is $\Om(u + \mu, v + \nu) = \nu(u) - \mu(v)$ and the para-Kähler metric is $G(u + \mu, v + \nu) = \mu(v) + \nu(u)$. This graph is Lagrangian if and only if $Q$ is the differential of a polynomial, $P$, on $\alg$. In this case the restriction of $G$ to the submanifold $\Sigma$ equals $2D dP$, where $\nabla$ is the Levi-Civita connection of $G$. The cubic form $DD dP$ equals that obtained as $\Om(\Pi(\dum, \dum), \dum)$ by pairing the second fundamental form $\Pi$ of $\Sigma$ with the symplectic structure $\Om(u + \nu, v + \nu) = \mu(v) - \nu(u)$. In particular, that $\lap_{h}P = 0$ is equivalent to $\Sigma$ having vanishing mean curvature. 
\end{remark}

\begin{remark}
Lemma \ref{notpdelemma} shows that, although \eqref{einsteinpolynomials} is a system of partial differential equations, its content is essentially algebraic and combinatorial. More precisely, any smooth solution of \eqref{einsteinpolynomials} must be a cubic homogeneous polynomial plus a linear form.
\begin{lemma}\label{notpdelemma}
Let $(\alg, h)$ be a Euclidean vector space. If $F \in C^{2}(\alg)$ solves \eqref{einsteinpolynomials} for some constant $\ka \neq 0$, then $F$ is equal to the sum of a a harmonic cubic homogeneous polynomial $P \in \pol^{3}(\alg)$ solving \eqref{einsteinpolynomials} with constant $\ka$ and a linear form $\ell \in \alg^{\ast}$.
\end{lemma}

\begin{proof}
Because $F$ is harmonic, it is real analytic. Differentiating $F_{pq}F^{pq} = \ka |x|^{2}$ twice yields $F_{ip}\,^{q}F_{jq}\,^{p} + F_{ijp}\,^{q}F_{q}\,^{p} = \ka h_{ij}$. Differentiating this yields $3F_{p(ij}\,^{q}F_{k)q}\,^{p} + F_{ijkp}\,^{q}F_{q}\,^{p} = 0$. Tracing this in $jk$ and using that all derivatives of $F$ are harmonic yields $F_{iabc}F^{abc} = 0$. Differentiating this yields $F_{iabc}F_{j}\,^{abc} + F_{ijabc}F^{abc} = 0$ and tracing this in $ij$ shows that $F_{ijkl}F^{ijkl} = 0$ which implies that $F_{ijkl} = 0$, so that $F$ is a cubic polynomial.

Write $F = P + Q + R$ with $P \in \pol^{3}(\alg)$, $Q \in \pol^{2}(\alg)$, and $R \in \pol^{1}(\alg)$. 
Taking $x = 0$ in $\ka |x|^{2} = |\hess P|^{2} + 2\lb \hess P, \hess Q\ra + |\hess Q|^{2}$ yields $|\hess Q|^{2} = 0$ and $\ka |x|^{2} = |\hess P|^{2}$. Since $\hess Q$ is a constant, this implies $Q = 0$. Since $R$ is linear, $\lap_{h}F = \lap_{h}P$, so $P$ solves \eqref{einsteinpolynomials}.
\end{proof}
\end{remark}

\begin{remark}
Some version of the results presented here appeared in \cite{Fox-ahs} and was briefly summarized in \cite{Fox-crm}. 
\end{remark}

\section{Critical lines and zeros}\label{criticalsection}
This section gives some lemmas useful for deciding orthogonal equivalence or orthogonal indecomposability of solutions.

Let $\crit(P) \subset \alg$ denote the set of critical points of $P$ and let $\proj(\crit(P))$ be its image in the projectivization $\proj(\alg)$. 

A \emph{critical line} of $P \in \pol^{3}(\alg)$ is a one-dimensional subspace $[v] \in \proj(\alg)$ spanned by a vector $v$ such that $h(v, \dum) \wedge dP(v) = 0$. If $h$ is Euclidean, a vector $v$ spanning a critical line of $P$ is a critical point of the restriction of $P$ to the $h$-sphere containing $v$. Since $-v$ is also a critical point of the restriction of $P$ to the same sphere, it is more convenient to speak of critical lines than of the (pairs of) critical points that generate them. 

Because the restriction of $P$ to a sphere of nontrivial radius has a critical point, the set of critical lines of $P$, $\critline(P) \subset \proj(\alg)$, is nonempty. Because $P$ is an odd function, there can always be chosen a generator of a critical line on which $P$ is positive.

Let $\zero(P) = \{[v] \in \proj(\alg): P(v) =0\}\subset \proj(\alg)$ be the set of one-dimensional subspaces of $\alg$ generated by zeros of $P$. 

Because $P$ is homogeneous, $\proj(\crit(P)) = \critline(P) \cap \zero(P)$. 

\begin{lemma}\label{decomposablecriticallemma}
Let $(\alg, h)$ be a Euclidean vector space. If $P \in \pol^{3}(\alg)$ is orthogonally decomposable then there are $h$-orthogonal $[v_{1}], [v_{2}] \in \critline(P)$ such that $(\hess P)(x)(v_{1}, v_{2}) = 0$ for all $x \in \alg$.
\end{lemma}

\begin{proof}
By assumption there is a splitting $\alg = \alg_{1}\oplus\alg_{2}$ where $\alg_{1}$ and $\alg_{2}$ are nontrivial proper $h$-orthogonal subspaces and there are polynomials $Q_{i} \in \pol^{3}(\alg_{i})$ such that $P = Q_{1}\circ \pi_{1} + Q_{2} \circ \pi_{2}$ where $\pi_{i}$ is the orthogonal projection onto $\alg_{i}$. let $\inc_{i}:\alg_{i} \to \alg$ be the inclusion. 
For $i = 1, 2$, if $v_{i} \in \alg_{i}$ is such that $v_{i}$ spans a critical line of $Q_{i}$, then $h(\inc_{i}(v_{i}), \dum)\wedge dP(v_{i}) = h(v_{i}, \dum)\wedge dQ_{i}(v_{i}) = 0$, so $\inc_{1}(v_{1})$ and $\inc_{2}(v_{2})$ generate orthogonal critical lines of $P$. For $x \in \alg$ let $x_{i} = \pi_{i}(x)$. Then $(\hess P)(x)(v_{1}, v_{2}) = \pi_{1}^{\ast}(\hess Q_{1})(x_{1}))(v_{1}, v_{2}) + \pi_{2}^{\ast}(\hess Q_{2})(x_{2}))(v_{1}, v_{2}) = 0$.
\end{proof}

\begin{example}
By Lemma \ref{decomposablecriticallemma} to check that a solution $P$ of \eqref{einsteinpolynomials} is orthogonally indecomposable it suffices to show that no two of its critical lines are orthogonal.
For example, the cubic polynomial $P = \tfrac{1}{6}(x_{1}^{3} - 3x_{1}x_{2}^{2})$ of Example \ref{parahurwitzexample} is orthogonally indecomposable by Lemma \ref{decomposablecriticallemma}. Its critical lines are given by the solutions of the equations $x_{1}^{2} - x_{2}^{2} = 2\la x_{1}$ and $-x_{1}x_{2}= \la x_{2}$ which are generated by $\ga_{0} = (1, 0)$, $\ga_{1} = (-1/2, \sqrt{3}/2)$ and $(-1/2, -\sqrt{3}/2)$, which are not pairwise orthogonal.
\end{example}

\begin{example}
The polynomial $x_{1}x_{2}x_{3}$ is orthogonally indecomposable but among its critical lines those generated by the coordinate axes $e_{1}$, $e_{2}$, and $e_{3}$, \emph{are} pairwise orthogonal. However, for $x = x_{1}e_{1} + x_{2}e_{2} + x_{3}e_{3}$, $(\hess P)(x)(e_{1}, e_{2}) = x_{3}$ is not zero for all $x$, so the full condition of Lemma \ref{decomposablecriticallemma} is not satisfied.
\end{example}

Define the \emph{weight} of a critical line $\ell \in \critline(P)$ to be $|v|^{-6}_{h}P(v)^{2}$ for any vector $v$ spanning $\ell$. 

Since the automorphism group of $P$, $\Aut(P)$, permutes the critical lines of $P$ and preserves their weights, $\Aut(P)$ is a finite group if $\critline(P)$ is finite.

\begin{lemma}\label{zeroautomorphismlemma}
Let $(\alg, h)$ be a Euclidean vector space and let $P \in \pol^{3}(\alg)$. The reflection $\si_{r}$ through the hyperplane $\rperp$ $h$-orthogonal to $0 \neq r \in \alg$ is an automorphism of $P$ if and only if $[r] \in \zero(P)$ and $\lb dP(x), r\ra = (\hess P)(r)(x, x) =0$ for all $x \in \rperp = \{y \in \alg: h(y, r) = 0\}$.
\end{lemma}

\begin{proof}
 Suppose $0 \neq r \in \alg$ and $\si_{r} \in \Aut(P)$. Then $P(r) = P(\si_{r}(r)) = P(-r) = -P(r)$, so $P(r) = 0$ and $[r] \in \zero(P)$. Let $x \in \rperp$. Then 
\begin{align}
\lb dP(x), r\ra = \tfrac{d}{dt}\big|_{t = 0}P(x + tr) = \tfrac{d}{dt}\big|_{t = 0}P(\si_{r}(x - tr)) = \tfrac{d}{dt}\big|_{t = 0}P(x - tr) = -\lb dP(x), r\ra,
\end{align}
so $\lb dP(x), r\ra  = 0$. 

Now suppose $[r] \in \zero(P)$ and $\lb dP(x), r\ra  = 0$ for all $x \in \rperp$. Any $w \in \alg$ has the form $w = \al r + z$ with $z \in \rperp$, and $\si_{r}(w) = -\al r + z$. Because $P$ is cubic homogeneous, $P(x + y) = P(x) + P(y) + \lb dP(x), y\ra + \lb dP(y), x\ra$ for all $x, y \in \alg$. Hence
\begin{align}
\begin{split}
P(w) &= P(\al r + z) = \al^{3}P(r) + \al^{2}\lb dP(r), z\ra + \al \lb dP(z), r\ra + P(z)  =  \al^{2}\lb dP(r), z\ra  + P(z) \\
&= -\al^{3}P(r) + \al^{2}\lb dP(r), z\ra - \al \lb dP(z), r\ra + P(z) = P(-\al r + z) = P(\si_{r}(w)),
\end{split}
\end{align}
which shows that $P(\si_{r}(w)) = P(w)$ for all $w \in \alg$.
\end{proof}

\section{Tensor products}\label{tensorproductsection}
The \emph{tensor product} $\al \tensor \be \in S^{k}(\alg \tensor \balg)^{\ast}$ of symmetric $k$-forms $\al \in S^{k}\alg^{\ast}$ and $\be \in S^{k}\balg^{\ast}$ is the unique element of $S^{k}(\alg \tensor \balg)^{\ast}$ satisfying
\begin{align}\label{formproduct}
(\al \tensor \be)(a_{1}\tensor b_{1}, \dots, a_{k}\tensor b_{k}) = \al(a_{1}, \dots, a_{k})\be(b_{1}, \dots, b_{k})
\end{align}
for all $a_{1}, \dots, a_{k} \in \alg$ and $b_{1}, \dots, b_{k} \in \balg$. Let $P^{\al} \in \pol^{k}(\alg)$, $P^{\be} \in \pol^{k}(\balg)$ and $P^{\al\tensor \be} \in \pol^{k}(\alg \tensor \balg)$ be the associated degree $k$ homogeneous polynomials. By definition, for $a \in \alg$ and $b \in \balg$,
\begin{align}\label{paltbe}
P^{\al \tensor \be}(a \tensor b) & = \tfrac{1}{k!}(\al \tensor \be)(a \tensor b, \dots, a \tensor b) = \tfrac{1}{k!}\al(a, \dots, a)\be(b, \dots, b) = k!P^{\al}(a)P^{\be}(b),
\end{align}
where the dots indicated $k$fold repetitions. The identity \eqref{paltbe} motivates defining the \emph{tensor product} $P\tensor Q \in \pol^{k}(\alg \tensor \balg)$ of $P \in \pol^{k}(\alg)$ and $Q \in \pol^{k}(\balg)$ to be the element of $\pol^{k}(\alg \tensor \balg)$ corresponding to the tensor product of the symmetric $k$ forms corresponding to $P$ and $Q$. By definition, 
\begin{align}\label{tensorpol}
(P \tensor Q)(a\tensor b) = k!P(a)Q(b)
\end{align}
for $a \in \alg$ and $b \in \balg$.

The tensor product $(\alg \tensor \balg, g \tensor h)$ of the Euclidean vector spaces $(\alg, g)$ and $(\balg, h)$ is $\alg \tensor \balg$ equipped with the tensor product bilinear form $g \tensor h$, defined as in \eqref{formproduct}.

The relation of the coordinate expressions of the tensor product of two polynomials to the coordinate expressions of the tensored polynomials is not visually obvious, but can be understood as a Kronecker product of tensorial arrays. If $\{e_{i}: 1 \leq i \leq \dim \alg\}$ and $\{f_{\al}: 1 \leq \al \leq \dim \balg\}$ are orthonormal bases of $(\alg, g)$ and $(\balg, h)$, then $\{k_{i\al} = e_{i}\tensor f_{\al}: 1 \leq i \leq \dim \alg, 1 \leq \al \leq \dim \balg\}$ is an orthonormal basis of $(\alg\tensor \balg, g \tensor h)$ and $(P \tensor Q)(k_{ia})  = k!P(e_{i})Q(f_{\al})$. 

Write $P_{i_{1}\dots i_{k}} = P(e_{i_{1}}, \dots, e_{i_{k}})$ and $Q_{\al_{1}, \dots, \al_{k}} = Q(f_{\al_{1}}, \dots, f_{\al_{k}})$. By the definition \eqref{formproduct},
\begin{align}\label{kronecker}
\begin{split}
 (P \tensor Q)_{(i_{1}\al_{i})\dots (i_{k}\al_{k})} &= (P\tensor Q)(e_{i_{1}}\tensor f_{\al_{1}}, \dots, e_{i_{k}}\tensor f_{\al_{k}}) \\
&=  P(e_{i_{1}}, \dots, e_{i_{k}})Q(f_{\al_{1}}, \dots, f_{\al_{k}}) = P_{i_{1}\dots i_{k}} Q_{\al_{1}, \dots, \al_{k}} 
\end{split}
\end{align}
shows that the array representing $P \tensor Q$ with respect to the orthonormal basis $\{k_{i\al}\}$ of $(\alg \tensor \balg, g \tensor h)$ is the Kronecker product of the arrays representing $P$ and $Q$ with respect to the orthonormal bases $\{e_{i}\}$ and $\{f_{\al}\}$. 

\begin{lemma}\label{tensorproductlemma}
Let $(\alg \tensor \balg, g \tensor h)$ be the tensor product of the Euclidean vector spaces $(\alg, g)$ and $(\balg, h)$.
\begin{enumerate}
\item\label{tpharmonic} If at least one of $P \in \pol^{3}(\alg)$ and $Q \in \pol^{3}(\balg)$ is harmonic, then $P\tensor Q$ is harmonic on $(\alg \tensor \balg, g \tensor h)$.
\item\label{tphess} If $P \in \pol^{3}(\alg)$ and $Q \in \pol^{3}(\balg)$ solve the second equation of \eqref{einsteinpolynomials} with constants $0 \neq \ka_{P}, \ka_{Q} \in \rea$, then $P\tensor Q \in \pol^{3}(\alg \tensor \balg)$ solves the second equation of \eqref{einsteinpolynomials} on $(\alg \tensor \balg, g \tensor h)$ with constant $\ka_{P}\ka_{Q}$.
\item\label{tpeinstein} If at least one of $P \in \pol^{3}(\alg)$ and $Q \in \pol^{3}(\balg)$ is harmonic and $P$ and $Q$ solves \eqref{einsteinpolynomials} with constants $0 \neq \ka_{P}, \ka_{Q} \in \rea$, then $P\tensor Q \in \pol^{3}(\alg \tensor \balg)$ solves \eqref{einsteinpolynomials} on $(\alg \tensor \balg, g \tensor h)$ with constant $\ka_{P}\ka_{Q}$.
\end{enumerate}
\end{lemma}

\begin{proof}
Let $\{e_{i}: 1 \leq i \leq \dim \alg\}$ and $\{f_{\al}: 1 \leq \al \leq \dim \balg\}$ be orthonormal bases of $(\alg, g)$ and $(\balg, h)$ and let $\{k_{i\al} = e_{i}\tensor f_{\al}: 1 \leq i \leq \dim \alg, 1 \leq \al \leq \dim \balg\}$ be the corresponding orthonormal basis of $(\alg\tensor \balg, g \tensor h)$ .
Write $P_{ijk} = P(e_{i}, e_{j}, e_{k})$, $Q_{\al\be\ga} = Q(f_{\al}, f_{\be}, f_{\ga})$ and $(P \tensor Q)_{(i\al)(j\be)(k\ga)} = (P\tensor Q)(k_{i\al}, k_{j\be}, k_{k\ga})= P_{ijk}Q_{\al\be\ga}$. If at least one of $P$ and $Q$ is harmonic, then
\begin{align}
\begin{split}
(\lap_{g\tensor h}(P\tensor Q))(k_{i\al}) & = \sum_{p, q, \mu, \nu}(P\tensor Q)(k_{i\al}, k_{p\mu}, k_{q\nu}) \\&= \sum_{p, q, \mu, \nu}P_{ipq}Q_{\al\mu \nu} = (\lap_{g}P)(e_{i})(\lap_{h}Q)(f_{\al}) = 0.
\end{split}
\end{align}
This shows \eqref{tpharmonic}. The equation \eqref{einsteinpolynomials} for $P$ is equivalent to $\sum_{p, q}P_{ipq}P_{jpq} = \ka_{p}g_{ij}$ where $g_{ij} = g(e_{i}, e_{j})$. Using the similar relation for $Q$ there results
\begin{align}\begin{split}
\sum_{p, q, \mu, \nu}(P \tensor Q)_{(i\al)(p\mu)(q\nu)}(P \tensor Q)_{(j\be)(p\mu)(q\nu)} &= \sum_{p, q, \mu, \nu}P_{ipq}Q_{\al\mu\nu}P_{jpq}Q_{\be\mu\nu}\\& = \ka_{P}\ka_{Q}g_{ij}h_{\al\be} = \ka_{P}\ka_{Q}(g\tensor h)_{(i\al)(j\be)},
\end{split}
\end{align}
which is equivalent to $P\tensor Q$ solving \eqref{einsteinpolynomials} on $(\alg \tensor \balg, g \tensor h)$ with constant $\ka_{P}\ka_{Q}$. This shows \eqref{tphess}.

Claim \eqref{tpeinstein} follows from \eqref{tpharmonic} and \eqref{tphess}.
\end{proof}

\begin{example}
If $x_{1}, x_{2}$ and $y_{1}, y_{2}, y_{3}$ are coordinate systems with respect to the standard orthonormal bases $\{e_{1}, e_{2}\}$ and $\{f_{1}, f_{2}, f_{3}\}$ in Euclidean $\rea^{2}$ and $\rea^{3}$, and $P(x_{1}, x_{2}) = \tfrac{1}{6}(x_{1}^{3} - 3x_{1}x_{2}^{2})$ and $Q(y_{1}, y_{2}, y_{3}) = y_{1}y_{2}y_{3}$, then the expression for $P\tensor Q$ with respect to the coordinates $z_{1}, \dots, z_{6}$ with respect to the orthonormal basis $k_{\al + 3(i -1)} = e_{i}\tensor f_{\al}$ of $\rea^{6} \simeq \rea^{2}\tensor \rea^{3}$ is
\begin{align}\label{phtriple}
(P\tensor Q)(z) = z_{1}z_{2}z_{3} - z_{1}z_{5}z_{6} - z_{4}z_{2}z_{6} - z_{4}z_{5}z_{3}.
\end{align}
\end{example}

The tensor products of $P$ with \eqref{2dharpoly} and \eqref{basepoly} admit alternative interpretations that are described in Lemmas \ref{parahurwitzificationlemma} and \ref{triplepolynomiallemma}.

\begin{lemma}\label{parahurwitzificationlemma}
Suppose $P \in \pol^{3}(\alg)$ solves $|\hess P|_{h}^{2} = \ka|x|_{h}^{2}$ with constant $0 \neq \ka \in \rea$ on the Euclidean vector space $(\alg, h)$. Equip $\alg\tensor_{\rea}\com$ with the Hermitian metric induced by the Euclidean metric on $\alg$. Let $Q \in \pol^{3}(\alg \tensor_{\rea}\com)$ be the real part of the polynomial on $\alg\tensor_{\rea}\com$ obtained from $P$ by extension of scalars.
\begin{enumerate}
\item $Q$ solves \eqref{einsteinpolynomials} with constant $2\ka$ on $\alg\tensor_{\rea}\com$.
\item\label{ealg2tensor} $Q = P \tensor R$ where $R(x_{1}, x_{2}) = \tfrac{1}{6}(x_{1}^{3} - 3x_{1}x_{2}^{2})$ in coordinates $x_{1}, x_{2}$ with respect to a standard orthonormal basis on Euclidean $\rea^{2}$.
\end{enumerate}
\end{lemma}

\begin{proof}
Let $z = (z_{1} = x_{1} + \j y_{1}, \dots, z_{n} = x_{n} + \j y_{n}) = x + \j y$ be coordinates on $\alg\tensor_{\rea}\com$ such that $dz_{1}, \dots, dz_{n}$ is a unitary basis of parallel one-forms. By definition $2Q(x, y) = P(z) + P(\bar{z})$. Write $P_{i} = \tfrac{\pr P}{\pr x_{i}}$ and $P_{ij} = \tfrac{\pr^{2}P}{\pr x_{i}\pr x_{j}}$. Then $2dQ = \sum_{i = 1}^{n}(P_{i}(z) + P_{i}(\bar{z}))dx_{i} + \j \sum_{i = 1}^{n}(P_{i}(z) - P_{i}(\bar{z}))dy_{i}$ and, because $P_{ij}(x)$ is a linear form in $x$, 
\begin{align}
\begin{split}
2&(DdQ)(z)  = \sum_{i , j= 1}^{n}(P_{ij}(z) + P_{ij}(\bar{z}))dx_{i}\tensor dx_{j} - \sum_{i , j= 1}^{n}(P_{ij}(z) + P_{ij}(\bar{z}))dy_{i}\tensor dy_{j} \\
&\quad + \j \sum_{i, j = 1}^{n}(P_{ij}(z) - P_{ij}(\bar{z}))(dx_{i}\tensor dy_{j} + dy_{j}\tensor dx_{i})\\
& = 2\sum_{i , j= 1}^{n}P_{ij}(x)dx_{i}\tensor dx_{j} - 2\sum_{i , j= 1}^{n}P_{ij}(x) dy_{i}\tensor dy_{j} 
- 2\sum_{i, j = 1}^{n}P_{ij}(y)(dx_{i}\tensor dy_{j} + dy_{j}\tensor dx_{i}).
\end{split}
\end{align}
In block matrix form this means
\begin{align}\label{parahurwitzificationhessian}
(\hess Q)(x, y) = \begin{pmatrix*}[r] (\hess P)(x) & -(\hess P)(y)\\ -(\hess P)(y)& -(\hess P)(x)\end{pmatrix*},
\end{align}
from which it follows that $Q$ is harmonic and $|\hess Q(x, y)|^{2} = 2|\hess P(x)|^{2} + 2|\hess P(y)|^{2} = 2\ka |(x, y)|^{2}$.

Claim \eqref{ealg2tensor} is most easily proved by checking that the matrix with respect to a suitable orthonormal basis of the Hessian of the $Q$ of \eqref{ealg2tensor} has the form \eqref{parahurwitzificationhessian}.
\end{proof}

\begin{remark}
Note that in Lemma \ref{parahurwitzificationlemma}, $P$ need not be harmonic.
\end{remark}

Because $R(x_{1}, x_{2}) = \tfrac{1}{6}(x_{1}^{3} - 3x_{1}x_{2}^{2})$ is the cubic polynomial of the para-Hurwitz algebra, claim \eqref{ealg2tensor} of Lemma \ref{parahurwitzificationlemma} motivates calling the polynomial $Q$ constructed from $P$ in that lemma the \emph{parahurwitzification} of $P$.

\begin{example}
Let $P= x_{1}x_{2}x_{3}$. The parahurwitzification $Q$ is the real part of $z_{1}z_{2}z_{3}$ where $z_{i} = x_{i} + \j y_{i}$, $i = 1, 2, 3$ and there results
\begin{align}
Q(x_{1}, y_{1}, x_{2}, y_{2}, x_{3}, y_{3}) = x_{1}x_{2}x_{3} - x_{1}y_{2}y_{3} - y_{1}x_{2}y_{3} - y_{1}y_{2}x_{3}.
\end{align}
which is the polynomial \eqref{phtriple} after relabeling of variables.
\end{example}

\begin{example}\label{parahurwitzebasepolyexample}
Let $P= \tfrac{1}{6}(x_{1}^{3} - 3x_{1}x_{2}^{2})$. The parahurwitzification $Q$ is the real part of $\tfrac{1}{6}(z_{1}^{3} - 3z_{1}z_{2}^{2})$ where $z_{i} = x_{i} + \j y_{i}$, $i = 1, 2$ and there results
\begin{align}
\begin{split}
Q(x_{1}, y_{1}, x_{2}, y_{2}) &= \tfrac{1}{6}\left(x_{1}^{3} - 3x_{1}y_{1}^{2} - 3x_{1}(x_{2}^{2} - y_{2}^{2}) + 6y_{1}x_{2}y_{2} \right)\\
& = \tfrac{1}{6}x_{1}^{3} - \tfrac{1}{2}x_{1}(x_{2}^{2} + y_{1}^{2} - y_{2}^{2}) + y_{1}x_{2}y_{2},
\end{split}
\end{align}
which is the polynomial \eqref{poly3} after relabeling of variables.
\end{example}

\begin{lemma}\label{triplepolynomiallemma}
Let $(\alg, h)$ be a Euclidean vector space. 
Suppose $P \in \pol^{3}(\alg)$ solves $|\hess P|_{h}^{2} = \ka|x|_{h}^{2}$ with constant $0 \neq \ka \in \rea$. 
\begin{enumerate}
\item The polynomial $Q(x, y, z) \in \pol^{3}(\alg \oplus \alg \oplus \alg)$ defined for $(x, y, z) \in \alg \oplus \alg \oplus \alg$ by
\begin{align}
Q(x, y, z) = P(x + y + z) - P(x + y) - P(y + z) - P(z + x) + P(x) + P(y) + P(z)
\end{align}
solves \eqref{einsteinpolynomials} with constant $2\ka$ for the metric $|(x, y, z)|^{2}_{h} = |x|^{2}_{h} + |y|^{2}_{h} + |z|^{2}_{h}$ on $\alg \oplus \alg \oplus \alg$.
\item \label{ealg3tensor} $Q = P \tensor R$ where $R(x_{1}, x_{2}, x_{3}) = x_{1}x_{2}x_{3}$ in coordinates $x_{1}, x_{2}, x_{3}$ with respect to a standard orthonormal basis on Euclidean $\rea^{3}$.
\end{enumerate}
\end{lemma}

\begin{proof}
Since the components of $\hess P$ are linear functions, $(\hess P)(x + y +z) - (\hess P)(x + y) = (\hess P)(z)$. Using this and similar identities yields
\begin{align}\label{hesstriple}
(\hess Q)(x, y, z) = \begin{pmatrix} 0 & (\hess P)(z) & (\hess P)(y) \\ (\hess P)(z) & 0 & (\hess P)(x)\\  (\hess P)(y)&  (\hess P)(x) & 0\end{pmatrix}.
\end{align}
Tracing this shows that $Q$ is harmonic, while taking its squared-norm yields 
\begin{align}
|\hess Q|^{2}_{h}(x, y, z) = 2|\hess P|^{2}_{h}(x) +  2|\hess P|^{2}_{h}(y) +  2|\hess P|^{2}_{h}(z) = 2\ka |(x, y, z)|^{2}_{h},
\end{align}
so $Q$ solves \eqref{einsteinpolynomials} with constant $2\ka$.

Claim \eqref{ealg3tensor} is most easily proved by checking that the matrix with respect to a suitable orthonormal basis of the Hessian of the $Q$ of \eqref{ealg3tensor} has the form \eqref{hesstriple}.
\end{proof}

\begin{remark}
Note that in Lemma \ref{triplepolynomiallemma}, $P$ need not be harmonic.
\end{remark}

Claim \eqref{ealg3tensor} of Lemma \ref{triplepolynomiallemma} motivates calling the polynomial $Q$ constructed from $P$ in that lemma the \emph{triple} of $P$.

\begin{example}\label{1dseedexample}
The triple of the one-variable polynomial $P = \tfrac{1}{6}x^{3}$ is $Q(x_{1}, x_{2}, x_{3}) = x_{1}x_{2}x_{3}$. 
\end{example}

\begin{example}\label{2dseedexample}
The triple of $P = \tfrac{1}{6}(x_{1}^{3} - 3x_{1}x_{2}^{2})$ is \eqref{phtriple} also obtained as the parahurwitzification of $x_{1}x_{2}x_{3}$ in Example \ref{parahurwitzebasepolyexample}.
\end{example}

\begin{example}
The tensor product of $x_{1}x_{2}x_{3}$ with itself yields $\sum_{\si \in S_{3}}x_{\si(1)}y_{\si(2)}z_{\si(3)}$, which is the permanent of the $3 \times 3$ matrix whose columns are the vectors $x$, $y$, and $z$,  as in \eqref{permpoly}.
\end{example}

\begin{example}\label{prehomexample}
There are examples of polynomials solving only the second equation of \eqref{einsteinpolynomials}.
The determinant of a $3 \times 3$ symmetric matrix,
\begin{align}
\label{prehomogpoly}
\begin{split}
P(x)  & = \det \begin{pmatrix} x_{11} & x_{12}/\sqrt{2} & x_{13}/\sqrt{2} \\ x_{12}/\sqrt{2} & x_{22} & x_{23}/\sqrt{2}\\ x_{13}/\sqrt{2} & x_{23}/\sqrt{2} & x_{33}\end{pmatrix}\\
 &= x_{11}x_{22}x_{33} - \tfrac{1}{2}\left(x_{11}x_{23}^{2} + x_{22}x_{13}^{2} + x_{33}x_{12}^{2}\right) + \tfrac{1}{\sqrt{2}}x_{12}x_{23}x_{13},
\end{split}
\end{align}
solves $|\hess P|^{2}_{h} = 3|x|^{2}_{h}$ where $h(x, x) = \tr x^{2}$. (The factors of $1/\sqrt{2}$ are used for convenience; with them $h(x, x)$ is the standard Euclidean metric in the given coordinates.) However, $P$ is not harmonic, for $\lap_{h}P = -x_{11} - x_{22} - x_{33} = - \tr x$. (While $P$ is not harmonic, it is biharmonic, that is $\lap_{h}^{2}P = 0$.)
The polynomial $P$ arises as the relative invariant of a real form of a reduced irreducible prehomogeneous vector space (see $(2)$ of table I in section $7$ of \cite{Sato-Kimura}). 

The parahurwitzification and triple of $P$ yield $12$ and $18$-variable solutions of \eqref{einsteinpolynomials}.
\end{example}

\begin{remark}
It is well known that polarization maps $G$-invariant polynomials on $\alg$ to $G$-invariant polynomials on the direct sum $\alg^{k}$. See \cite{Losik-Michor-Popov} for details.
\end{remark}

The cubic polynomials of the tensor products of certain anticommutative algebras yield solutions of \eqref{einsteinpolynomials}.

A symmetric bilinear form $h$ on an algebra $(\alg, \mlt)$ is \emph{invariant} if $h(x\mlt y, z) = h(x, y\mlt z)$. For example, the Killing form of a Lie algebra is invariant.

\begin{theorem}\label{anticommutativetensortheorem}
Let $(\g_{1}, [\dum, \dum]_{1})$ and $(\g_{2}, [\dum, \dum]_{2})$ be anticommutative algebras with invariant positive definite inner products $B_{1}$ and $B_{2}$. Let $h = B_{1}\tensor B_{2}$ be the tensor product bilinear form on the tensor product algebra $(\alg, \mlt) = (\g_{1}\tensor \g_{2}, \mlt_{1}\tensor \mlt_{2})$ (which is commutative). Then $P \in \pol^{3}(\g_{1}\tensor\g_{2})$ defined for $x \in \g_{1}\tensor \g_{2}$ by $6P(x) = h(x\mlt x, x)$ solves \eqref{einsteinpolynomials} for a nonzero constant.
\end{theorem}

\begin{proof}
For decomposable elements $a_{1}\tensor a_{2}, b_{1}\tensor b_{2} \in \g_{1}\tensor \g_{2}$, the multiplication $\mlt$ is defined by $(a_{1}\tensor a_{2})\mlt(b_{1}\tensor b_{2}) = [a_{1}, b_{1}]_{1}\tensor [a_{2}, b_{2}]_{2}$ and is evidently commutative. For decomposable $a_{1}\tensor a_{2}, b_{1}\tensor b_{2}, c_{1}\tensor c_{2} \in \g_{1}\tensor \g_{2}$, 
\begin{align}
\begin{split}
h((a_{1}\tensor a_{2})&\mlt(b_{1}\tensor b_{2}), c_{1}\tensor c_{2}) = h( [a_{1}, b_{1}]\tensor [a_{2}, b_{2}], c_{1}\tensor c_{2}) \\
& = B_{1}([a_{1}, b_{1}], c_{1})B_{2}([a_{2}, b_{2}], c_{2}) = B_{1}(a_{1}, [b_{1}, c_{1}])B_{2}(a_{2}, [b_{2}, c_{2}]) \\
&= h(a_{1}\tensor a_{2}, (b_{1}\tensor b_{2})\mlt(c_{1}\tensor c_{2})).
\end{split}
\end{align}
Hence $h(x\mlt, y)$ is completely symmetric for all $x, y, z \in \g_{1}\tensor \g_{2}$, so it makes sense to define the associated cubic polynomial $P(x)$ by $6P(x) = h(x\mlt x, x)$.

The rest of the proof is formally the same as the proof of Lemma \ref{tensorproductlemma}, so is omitted.
\end{proof}

\begin{corollary}\label{lietensorcorollary}
Let $\g_{1}$ and $\g_{2}$ be compact semisimple real Lie algebras with Killing forms $B_{1}$ and $B_{2}$. Let $h = B_{1}\tensor B_{2}$ be the tensor product bilinear form on the tensor product algebra $(\g_{1}\tensor \g_{2}, \mlt)$ (which is commutative). Then $P \in \pol^{3}(\g_{1}\tensor\g_{2})$ defined for $x \in \g_{1}\tensor \g_{2}$ by $6P(x) = h(x\mlt x, x)$ solves \eqref{einsteinpolynomials} for a nonzero constant.
\end{corollary}

\begin{proof}
Because $\g_{1}$ and $\g_{2}$ are compact and semisimple, $-B_{1}$ and $-B_{2}$ are positive definite and invariant, so $h = B_{1}\tensor B_{2}$ is a positive definite symmetric bilinear form and the claim follows from Theorem \ref{anticommutativetensortheorem}.
\end{proof}

\begin{example}\label{so3example}
View $\so(3)$ as the imaginary quaternions, $\im \quat$, equipped with the cross product algebra structure given by the commutator.
Up to a constant factor, the cubic polynomial of the tensor product $\in \quat\tensor \im \quat = \so(3)\tensor \so(3)$ is the determinant of a $3 \times 3$ matrix as in \eqref{nahmpolyso3}.
\end{example}

The extra generality of Theorem \ref{anticommutativetensortheorem} relative to Corollary \ref{lietensorcorollary} is not superfluous. Theorem \ref{anticommutativetensortheorem} applies to the tensor product of a compact semisimple real Lie algebra with the $7$-dimensional cross product algebra given by the imaginary octonions equipped with the commutator bracket.

\begin{example}\label{imagoctonionexample}
Up to a constant factor the cubic polynomial $P(z)$ determined by the tensor product $\im \quat \tensor \im \cayley$ where the imaginary octonions $\im \cayley$ are equipped with the cross-product algebra structure given by the commutator, is the $21$-variable polynomial $P(z) = \re(z_{1}(z_{2}z_{3}))$ for $z = (z_{1}, z_{2}, z_{3}) \in \im \cayley\oplus\im \cayley\oplus \im \cayley$.
The automorphism group of the octonions is the $14$-dimensional compact real form of the simple Lie group of type $G_{2}$. It preserves the cross-product on $\im \cayley$, so acts as automorphisms of $P$.
\end{example}

\begin{remark}
The commutative algebras associated to the cubic polynomials of Examples \ref{so3example} and \ref{imagoctonionexample} are among the Hsiang algebras defined and studied by Tkachev in \cite{Tkachev-hsiang}; see section $7$ of \cite{Tkachev-summary} for a summary situating these examples in the general theory.
\end{remark}

\section{An invariant of solutions}
Let $(\alg, h)$ be a Euclidean vector space. For $P \in \pol^{3}(\alg)$ define
\begin{align}\label{extremepdefined}
\extreme(P) = \argmax_{x \in \sphere_{h}(1)}P(x)
\end{align}
to be the set points at which $P$ attains its maximum on $h$-unit sphere $\sphere_{h}(1) = \{x:|x|^{2}_{h} = 1\}$. When helpful, the dependence on $h$ is indicated with a subscript, as in $\extreme_{h}(P)$.

\begin{lemma}\label{mkclemma}
Let $(\alg, h)$ be a Euclidean vector space. For $P \in \pol^{3}(\alg)$ not identically zero, the number
\begin{align}\label{mkcdefined}
\mkc(P) = \sup_{e \in \extreme(P)}\tfrac{|\hess P(e)|^{2}_{h}}{36P(e)^{2}} 
\end{align}
is constant on the $CO(h)$ orbit of $P$. If $P$ solves \eqref{einsteinpolynomials} with constant $\ka \neq 0$, then
\begin{align}\label{einsteinmkc}
\mkc(P) = \frac{\ka}{\left(\max_{x \in \sphere_{h}(1)}6P(x)\right)^{2}}.
\end{align}
\end{lemma}

Note that part of the conclusion of Lemma \ref{mkclemma} is that $\mkc(P)$ is unchanged if $h$ is replaced by $e^{t}h$. The numerical factor $36$ is a normalization.
\begin{proof}
Because $P$ is nontrivial and $\sphere_{h}$ is compact, $\extreme(P)$ is nonempty and if $e, f \in\extreme(P)$, then $P(e) = P(f)$. Since $|\hess P(e)|^{2}_{h} \leq \max_{x \in \sphere_{h}(1)}|\hess P(x)|^{2}_{h}$, the supremum in \eqref{mkcdefined} exists. It is clear from the definition of $\mkc(P)$ that $\mkc(g\cdot P) = \mkc(P)$ for $g \in O(h)$. If $r > 0$ and $\tilde{h} = r^{2}h$, then $e \in \extreme_{\tilde{h}}(P)$ if and only if $r^{-1}e\in \extreme_{h}(P)$, so
\begin{align}
\begin{split}
&\sup_{e \in \extreme_{\tilde{h}}(P)}\tfrac{|\hess P(e)|^{2}_{\tilde{h}}}{P(e)^{2}} = \sup_{f \in  \extreme_{h}(P)}\tfrac{|\hess P(rf)|^{2}_{r^{2}h}}{P(rf)^{2}}
= \sup_{f \in  \extreme_{h}(P)}\tfrac{r^{4}|\hess P(f)|^{2}_{h}}{r^{6}P(f)^{2}} = \sup_{e \in  \extreme_{h}(P)}\tfrac{|\hess P(e)|^{2}_{h}}{P(e)^{2}} ,
\end{split}
\end{align}
showing that $\mkc(P)$ is well defined and constant on the $CO(h)$ orbit of $P$.

Suppose $P$ solves \eqref{einsteinpolynomials} with constant $\ka \neq 0$. Then $|\hess P(e)|^{2}_{h} = \ka |e|^{2}_{h} = \ka$ for all $e \in  \extreme(P)$, so the supremum in \eqref{mkcdefined} is attained at any $e\in \extreme(P)$ and equals \eqref{einsteinmkc}.
\end{proof}

By Lemma \ref{mkclemma} two solutions, $P$ and $Q$, of \eqref{einsteinpolynomials} for which $\mkc(P) \neq \mkc(Q)$ are not $CO(h)$-equivalent.

\begin{lemma}\label{mkcboundlemma}
Let $(\alg, h)$ be an $n$-dimensional Euclidean vector space. If $P \in \pol^{3}(\alg)$ is harmonic, then
\begin{align}\label{mkclowerbound}
\tfrac{n}{n-1} \leq \mkc(P),
\end{align}
with equality if and only if $P(e)_{ij} = \tfrac{6P(e)}{n-1}(ne_{i}e_{j} - h_{ij})$ for all $e \in \extreme(P)$.
\end{lemma}

\begin{proof}
Let $e \in \extreme(P)$. Then there is $\theta \in \rea$ such that $P(e)_{i} = 2\theta e_{i}$. Contracting with $e$ shows $\theta = \tfrac{3}{2}P(e)$ so that $P(e)_{i} = 3P(e)e_{i}$. The endomorphism $L_{i}\,^{j} = \tfrac{1}{6P(e)}P(e)_{i}\,^{j}$ is $h$-self-adjoint. There holds $L_{i}\,^{j}e_{j} = \tfrac{1}{6P(e)}P(e)_{ij}e^{j} = \tfrac{P(e)_{i}}{3P(e)} = e_{i}$. In particular, if $v \in \spn\{e\}^{\perp} = \{u \in \alg: h(e, u) = 0\}$, then $v^{i}L_{i}\,^{j}e_{j} = e_{i}v^{i} = 0$, so $L_{i}\,^{j}$ preserves $\spn\{e\}^{\perp}$. Let $\la_{1}, \dots, \la_{n-1}$ be the eigenvalues of $L_{i}\,^{j}$ on $\spn\{e\}^{\perp}$. Because $P$ is harmonic $L_{p}\,^{p} = 0$, so, by the preceding, $\sum_{i = 1}^{n-1}\la_{i} = -1$. By the Cauchy-Schwarz inequality, $(n-1)\sum_{i = 1}^{n-1}\la_{i}^{2} \geq (\sum_{i = 1}^{n-1}\la_{i})^{2} = 1$, with equality if and only if all the $\la_{i}$ are equal to $-\tfrac{1}{n-1}$. Consequently, 
\begin{align}\label{mckb1}
\tfrac{|\hess P(e)|^{2}}{36P(e)^{2}} = L_{p}\,^{q}L_{q}\,^{p} = 1 + \sum_{i = 1}^{n-1}\la_{i}^{2} \geq \tfrac{n}{n-1},
\end{align}
with equality if and only if $\la_{i} = -1/(n-1)$ for $1 \leq i \leq n-1$, and this holds if and only if $P(e)_{ij} = 6P(e)L_{ij} = \tfrac{6P(e)}{n-1}(ne_{i}e_{j} - h_{ij})$.
\end{proof}

\begin{remark}
It would be interesting to characterize those harmonic $P$ for which there is equality in \eqref{mkclowerbound}. Corollary \ref{ndimexistencecorollary} shows that equality holds in the bound \eqref{mkclowerbound} for the simplicial polynomial $P_{n}$ defined in \eqref{simplicialpoly}. On the other hand, together Lemma \ref{stsextremelemma} and Theorem \ref{ststheorem} show that there is a solution of \eqref{einsteinpolynomials} that is not equivalent to $P_{n}$ but for which equality holds in \eqref{mkclowerbound}.
\end{remark}

\begin{corollary}
Let $(\alg, h)$ be an $n$-dimensional Euclidean vector space. If $P \in \pol^{3}(\alg)$ solves \eqref{einsteinpolynomials} with constant $\ka$ then
\begin{align}
P(x) \leq \sqrt{\tfrac{\ka(n-1)}{n}}|x|^{3}
\end{align}
for all $x \in \alg$.
\end{corollary}

\begin{proof}
This follows from \eqref{mkclowerbound}, \eqref{einsteinmkc} of Lemma \ref{mkclemma}, and the homogeneity of $P$.
\end{proof}

\begin{remark}
Lemma \ref{mkcmultlemma} gives a criterion that can sometimes be used to show inequivalence of solutions of \eqref{einsteinpolynomials}. 
\end{remark}

\begin{lemma}\label{mkcmultlemma}
Let $(\alg, g)$ and $(\balg, h)$ be Euclidean vector spaces. Suppose $P \in \pol^{3}(\alg)$ and $Q \in \pol^{3}(\balg)$ solve \eqref{einsteinpolynomials}. Then 
\begin{align}\label{mkcpq}
\mkc(P \tensor Q) \leq \mkc(P)\mkc(Q).
\end{align}
\end{lemma}
\begin{proof}
By Lemma \ref{tensorproductlemma}, $P \tensor Q$ solves \eqref{einsteinpolynomials} on $(\alg\tensor \balg, g \tensor h)$ with constant $\ka_{P}\ka_{Q}$ where $P$ and $Q$ solve \eqref{einsteinpolynomials} with constants $\ka_{P}$ and $\ka_{Q}$. By \eqref{einsteinmkc} of Lemma \ref{mkclemma},
\begin{align}\label{mkcpq1}
\begin{split}
\mkc(P \tensor Q) = \tfrac{\ka_{P}\ka_{Q}}{\left(\max_{z \in \sphere_{g\tensor h}(1)}6(P\tensor Q)(z)\right)^{2}}.
\end{split}
\end{align}
By \eqref{tensorpol},
\begin{align}\label{maxpq}
\begin{split}
\max_{z \in \sphere_{g\tensor h}(1)}&6(P\tensor Q)(z) \geq \max_{x\tensor y: x \in \sphere_{g}(1), y \in \sphere_{h}(1)}6(P\tensor Q)(x \tensor y) \\
&=  \max_{x\tensor y: x \in \sphere_{g}(1), y \in \sphere_{h}(1)}(6P(x))(6Q(y)) = \left(\max_{x \in \sphere_{g}(1)}6P(x) \right)\left(\max_{y \in \sphere_{h}(1)}6Q(y) \right).
\end{split}
\end{align}
Substituting \eqref{maxpq} in \eqref{mkcpq1} and using \eqref{einsteinmkc} of Lemma \ref{mkclemma} yields
\begin{align}\label{mkcpq2}
\begin{split}
\mkc(P \tensor Q) \leq  \tfrac{\ka_{P}\ka_{Q}}{ \left(\max_{x \in \sphere_{g}(1)}6P(x) \right)^{2}\left(\max_{y \in \sphere_{h}(1)}6Q(y) \right)^{2}} = \mkc(P)\mkc(Q),
\end{split}
\end{align}
which shows \eqref{mkcpq}.
\end{proof}

\section{Associativity and conformal associativity equations}\label{associativitysection}
Let $(\alg, h)$ be an $n$-dimensional Euclidean vector space. The \emph{nonassociativity tensor} associated with $F \in \cinf(\alg)$ is defined by
\begin{align}
\ass(F)_{ijkl} = 2F_{l[i}\,^{p}F_{j]kp},
\end{align}
where $F_{i_{1}\dots i_{k}} = D_{i_{1}}\dots D_{i_{k}}F$ and indices are raised and lowered using $h_{ij}$ and $h^{ij}$. 

Let $L(x)_{i}\,^{j} \in \eno(\alg)$ be the $h$-self-adjoint endomorphism defined by $L(x)_{i}\,^{p}h_{pj} = F(x)_{ij}$. Then $x^{i}y^{j}\ass(F)_{ijkl} = [L(x), L(y)]_{kl}$, where $[L(x), L(y)]$ denotes the commutator of endomorphisms. This observation motivates calling $\ass(F)$ the nonassociativity tensor, for $\ass(F)$ vanishes if and only if the commutative multiplication $\mlt$ defined by $x \mlt y= L(x)y$ is associative. The equations $\ass(F)_{ijkl} = 0$ are called the \emph{associativity} or \emph{WDVV} equations; see \cite{Chen-Kontsevich-Schwarz} for background and references.

A straightforward computation shows that the curvature $R_{ijk}\,^{l}$ of the torsion-free affine connection $\nabla = D + tF_{ij}\,^{k}$ satisfies $R_{ijkl} = t^{2}\ass(F)_{ijkl}$ where $R_{ijkl} = R_{ijk}\,^{p}h_{pl}$.

\begin{lemma}\label{asseqlemma}
Let $(\alg, h)$ be an $n$-dimensional Euclidean vector space. For $P \in \pol^{3}(\alg)$ there holds $\ass(P)_{ijkl} = 0$ if and only if $P$ is orthogonally equivalent to a polynomial of the form $\tfrac{1}{6}\sum_{i = 1}^{n}\la_{i}x_{i}^{3}$ for constants $\la_{1}, \dots, \la_{n} \in \rea$. If, moreover, $|\hess P|^{2} = \ka |x|^{2}$ for some $0 < \ka \in \rea$, then $\la_{1} = \dots = \la_{n} = \sqrt{\ka}$.
\end{lemma}

\begin{proof}
That $\ass(P)_{ijkl} = 0$ means that the  family of $h$-self-adjoint endomorphisms $\{L(x):x \in \alg\}$ is commuting, so is simultaneously orthogonally diagonalizable. This means that there is an $h$-orthonormal basis $\{e_{1}, \dots, e_{n}\}$ with respect to which the matrix of $L(x)$ is diagonal of the form
\begin{align}
\begin{pmatrix} \ell_{1}(x) & &\\
& \ddots &\\
&&\ell_{n}(x)
\end{pmatrix},
\end{align}
where $\ell_{1}, \dots, \ell_{n} \in \alg^{\ast}$ are linear forms. Let $x^{1}, \dots x^{n}$ be coordinates with respect to $\{e_{1}, \dots, e_{n}\}$. If $i \neq j$, then $\tfrac{\pr}{\pr x^{i}}\ell_{j}(x) = \tfrac{\pr}{\pr x^{i}}L(x)_{j}\,^{j} = \tfrac{\pr}{\pr x^{j}}L(x)_{i}\,^{j} = 0$, so $\ell_{i}(x)  = \la_{i}x_{i}$ for some $\la_{i} \in \rea$. That $P$ has the form claimed now follows from the homogeneity of $P$.

If, moreover, $|\hess P|^{2} = \ka |x|^{2}$ for some $0 < \ka \in \rea$,then $\ka = |\hess P(e_{i})|^{2} = \la_{i}^{2}$ for $1 \leq i \leq n$. This shows the last claim.
\end{proof}

Let $\rictr(\ass(F))_{ij} = \ass(F)_{pij}\,^{p}$ and $\scal(\ass(F)) = \rictr(\ass(F))_{p}\,^{p}$. When $n \geq 3$, the \emph{conformal nonassociativity tensor} is the completely trace-free tensor defined by 
\begin{align}
\cass(F)_{ijkl} = \ass(F)_{ijkl} + \tfrac{2}{n-2}\left(h_{k[i}\rictr(\ass(F))_{j]l} - h_{l[i}\rictr(\ass(F))_{j]k}\right) - \tfrac{2}{(n-1)(n-2)}\scal(\ass(F))h_{k[i}h_{j]l}.
\end{align}
(It is defined in the same way as the Weyl tensor of a Riemannian metric.) 
A $P \in \pol^{3}(\alg)$ is \emph{associative} or \emph{conformally associative} if $\ass(P)_{ijkl} = 0$ or $\cass(P)_{ijkl} = 0$.

\begin{example}\label{epcassexample}
If $F \in \cinf(\alg)$ solves \eqref{einsteinpolynomials} with constant $\ka$, then $\rictr(\ass(F))_{ij} = -\ka h_{ij}$ and $\scal(\ass(F)) = -\ka n$, so 
\begin{align}
\ass(F)_{ijkl} = 2F_{l[i}\,^{p}F_{j]kp} - \tfrac{2\ka}{n-1}h_{k[i}h_{j]l}.
\end{align}
Equivalently
\begin{align}\label{casscommutator}
x^{i}y^{j}\cass(F)_{ijkl} = [L(x), L(y)]_{kl} - \tfrac{2\ka}{n-1}x_{[k}y_{l]}.
\end{align}
\end{example}

The \emph{affine extension} of $P \in \pol^{3}(\alg)$ is the polynomial $\hat{P} \in \pol^{3}(\alg \oplus \rea)$ defined by
\begin{align}
\hat{P}(x, r) = \tfrac{1}{6}r^{3} + \tfrac{1}{2}r|x|^{2} + P(x).
\end{align}

\begin{lemma}\label{confextautlemma}
Let $(\alg, h)$ be a Euclidean vector space. The affine extensions of $P, Q \in \pol^{3}(\alg)$ are orthogonally equivalent with respect to the metric $\hat{h}$ on $\hat{\alg} = \alg \oplus \rea$ defined by $\hat{h}((x, a), (y, b)) = h(x, y) + ab$ for $(x, a), (y, b) \in \hat{\alg}$ if and only if $P$ and $Q$ are $h$-orthogonally equivalent.
\end{lemma}

\begin{proof}
It is clear from the definition that the affine extensions of orthogonally equivalent polynomials are orthogonally equivalent.

Suppose $\Phi\in \eno(\hat{\alg})$ is $\hat{h}$-orthogonal and $\hat{P}(x, r) = \hat{Q}(\Phi(x, r))$. Write $\Phi(x, r) = (\phi(x) + ru, h(x, v) + cr)$ for $\phi \in \eno(\alg)$, $u, v \in \alg$, and $c \in \rea$. That $\Phi$ be orthogonal implies $|\phi(x)|^{2} + h(x, v)^{2} + 2r\left(h(u, \Phi(x)) + ch(x, v)\right) + r^{2}(|u|^{2} + c^{2}) = |x|^{2} + r^{2}$. This yields the equations
\begin{align}\label{hateq}
& |\phi(x)|^{2} + h(x, v)^{2} = |x|^{2},& & h(u, \phi(x)) + ch(x, v) = 0, & & |u|^{2} + c^{2} = 1,
\end{align}
for all $x \in \alg$. Taking $x = v$ in the first equation yields $|\phi(v)|^{2} = 0$, so that $\phi(v) = 0$. In the second equation this yields $c|v|^2 = 0$, so that either $c = 0$ or $v = 0$. If $c = 0$, then 
\begin{align}
\begin{split}
\tfrac{1}{6}r^{3} + \tfrac{1}{2}r|v|^{2} + P(v) & = \hat{P}(v, r) = \hat{Q}(\phi(v) + ru, |v|^{2}) \\&= \hat{Q}(ru, |v|^{2}) = \tfrac{1}{6}|v|^{6} + \tfrac{1}{2}r^{2}|v|^{2}|u|^{2} + r^{3}Q(u),
\end{split}
\end{align}
for all $r \in \rea$. This forces $|v|^{2} = 0$, so $v = 0$, which contradicts the invertibility of $\Phi$.

If $c \neq 0$, then $v = 0$. In this case the first equation of \eqref{hateq} yields $|\phi(x)|^{2} = |x|^{2}$, so that $\phi$ is $h$-orthogonal, and the second equation of \eqref{hateq} yields $0 = h(u, \phi(x)) = h(\phi(u), x)$ for all $x \in \alg$, so that $\phi(u) = 0$ and hence $u = 0$. In the third equation of \eqref{hateq} this yields $c = \pm 1$. Hence $\hat{P}(x, r) = \hat{Q}(\Phi(x, r)) = \hat{Q}(\phi(x), c r)$. Taking $x = 0$ yields $c^{3} = 1$, so that $c = 1$ and $\Phi(x, r) = (\phi(x), r)$. Taking $r = 0$ yields $P(x) = \hat{P}(x, 0) = \hat{Q}(\Phi(x, 0))= \hat{Q}(\phi(x), 0) = Q(\phi(x))$ so $P$ and $Q$ are orthogonally equivalent.  
\end{proof}

\begin{lemma}\label{asshatlemma}
Let $(\alg, h)$ be a Euclidean vector space of dimension $n \geq 3$. Equip $\hat{\alg} = \alg \oplus \rea$ with the Euclidean metric $\hat{h}$ defined by $\hat{h}((x, a), (y, b)) = h(x, y) + ab$ for $(x, a), (y, b) \in \hat{\alg}$. For $P \in \pol^{3}(\alg)$ and $(x, a), (y, b), (z, c), (w, d) \in\hat{\alg}$,
\begin{align}\label{asshatp}
\ass(\hat{P})((x, a), (y, b), (z, c), (w, d)) = (\ass(P)(x, y, z, w) + h(y, z)h(x, w) - h(x, z)h(y, w), 0).
\end{align}
\end{lemma}
\begin{proof}
With respect to an $\hat{h}$-orthonormal basis of $\hat{\alg}$ adapted to the splitting $\hat{\alg} = \alg \oplus \rea$, the matrix of $\hat{L}(x, a) \in \eno(\hat{\alg})$ defined by $(\hess \hat{P})(x, a) = \hat{h}(\hat{L}(x, a)(\dum), \dum)$ has the form
\begin{align}\label{hatl}
\hat{L}(x, a) = \begin{pmatrix} L(x) + a \id & x \\ h(x, \dum) & a \end{pmatrix}.
\end{align}
Because $L(x) y = L(y)x$, there follows
\begin{align}
[\hat{L}(x, a), \hat{L}(y, b)]  = \begin{pmatrix} [L(x), L(y)] + x \tensor h(y, \dum)  - y\tensor h(x, \dum) & 0 \\ 0 & 0 \end{pmatrix},
\end{align}
which is equivalent to \eqref{asshatp}.
\end{proof}

\begin{corollary}\label{confassextcorollary}
Let $(\alg, h)$ be a Euclidean vector space of dimension $n \geq 3$. If $P \in \pol^{3}(\alg)$ solves \eqref{einsteinpolynomials} with constant $0 \neq \ka \in \rea$ then $P$ is conformally associative if and only if the affine extension of $\sqrt{\tfrac{\ka}{n-1}}P$ is associative.
\end{corollary}
\begin{proof}
By Example \ref{epcassexample}, because $P$ solves \eqref{einsteinpolynomials}, $\rictr(\ass(P))_{ij} = -P_{ip}\,^{q}P_{jq}\,^{p} = -\ka h_{ij}$, $\scal(\ass(P)) = -n\ka$, and $\cass(P)_{ijkl}= \ass(P)_{ijkl} - \tfrac{2\ka}{n-1}h_{k[i}h_{j]l}$. The claim follows upon applying these observations to the affine extension of $\sqrt{\tfrac{\ka}{n-1}}P$ and using \eqref{asshatp}.
\end{proof}

\begin{theorem}\label{confassequivalencetheorem}
On a Euclidean vector space $(\alg, h)$ of dimension $n \geq 3$, any two conformally associative solutions of \eqref{einsteinpolynomials} with the same constant $0 < \ka \in \rea$ are orthogonally equivalent.
\end{theorem}

\begin{proof}
Let $P \in \pol^{3}(\alg)$ and let $\hat{P} \in \pol^{3}(\hat{\alg})$ be its affine extension. By \eqref{hatl}, $|\hess \hat{P}|^{2}_{\hat{h}}(x, r) = |\hess P|^{2}_{h}(x) + 2r\lap_{h}P(x) + (n+1)r^{2} + 2|x|^{2}_{h}$.  If $P$ solves \eqref{einsteinpolynomials}, then $|\hess \hat{P}|^{2}_{\hat{h}}(x, r) = (\ka + 2)|x|^{2} + (n+1)r^{2}$. Rescaling $P$ it can be supposed without any loss of generality that $\ka = (n-1)$, so that $|\hess \hat{P}|^{2}_{\hat{h}}(x, r) = (n+1)|(x, r)|^{2}_{\hat{h}}$.

Now suppose additionally that $P$ is conformally associative. By Corollary \ref{confassextcorollary}, the conformal extension $\hat{P}$ of $P$ is associative. By Lemma \ref{asseqlemma}, $\hat{P}$ is orthogonally equivalent to the polynomial $\tfrac{\sqrt{n+1}}{6}\sum_{i = 1}^{n+1}x_{i}^{3}$. 

If conformally associative $P, Q \in \pol^{3}(\alg)$ solve \eqref{einsteinpolynomials} with the same constant, it may be supposed without loss of generality that this constant is $n-1$. The preceding shows that $\hat{P}$ and $\hat{Q}$ are both orthogonally equivalent to $\tfrac{\sqrt{n+1}}{6}\sum_{i = 1}^{n+1}x_{i}^{3}$, so are orthogonally equivalent. By Lemma \ref{confextautlemma} this means $P$ and $Q$ are orthogonally equivalent.
\end{proof}

\section{Direct solution of \texorpdfstring{\eqref{einsteinpolynomials}}{ep} via normal forms}\label{directsolutionsection}
The naive approach to solving \eqref{einsteinpolynomials} is the brute force approach. This entails expressing a putative solution $P$ in terms of some basis of the space of cubic harmonic polynomials and rewriting the equations \eqref{einsteinpolynomials} in terms of the coefficients with respect to this basis. This section describes this approach to the extent that it works. It is most viable in low dimensions. The principal defect of this approach is that analyzing the properties of the solutions obtained, for example something so basic as deciding whether two solutions are orthogonally equivalent, is not straightforward.

In this section the $n$-dimensional real vector space $\alg$ is referred to as $\rea^{n}$ to indicate that there is fixed an $h$-orthonormal basis $e_{1}, \dots, e_{n}$ with respect to which $x_{i}$ are coordinates. 

\begin{lemma}\label{polorthonormalbasislemma}
Let $(\alg, h)$ be an $n$-dimensional Euclidean vector space. Let $x_{1}, \dots, x_{n}$ be coordinates such that $dx_{1}, \dots, dx_{n}$ is an $h$-orthonormal parallel coframe. With respect to the tensor norm determined by complete contraction with $h_{ij}$, the collection
\begin{align}\label{polobasis}
\{\sqrt{6}x_{i}x_{j}x_{k}: 1 \leq i < j < k \leq n\} \cup \{\tfrac{1}{2}(3x_{i}^{2}x_{j} - x_{j}^{3}): 1 \leq i \neq j \leq n\}
\end{align}
is a unit norm basis of $\pol^{3}(\alg)$.
\end{lemma}

\begin{proof}
The tensor corresponding to $x_{i}x_{j}x_{k}$ via polarization has six nonzero components, corresponding to the six permutations of $ijk$, each equal to $1/6$, so its squared tensor norm is $1/6$.
The tensor corresponding to $3x_{i}^{2}x_{j} - x_{j}^{3}$ has one component equal to $-1$, corresponding to the three times repeated index $j$, and three components equal to $1$, corresponding to the three permutations of $iij$, so has squared tensor norm $4$. In a similar fashion it can be checked that the vectors in \eqref{polobasis} are pairwise orthogonal except for those of the form $\tfrac{1}{2}(3x_{i}^{2}x_{j} - x_{j}^{3})$ and $\tfrac{1}{2}(3x_{k}^{2}x_{j} - x_{j}^{3})$ which are nonetheless linearly independent. 
\end{proof}

\begin{lemma}\label{harpreplemma}
Let $h_{ij}$ be a Euclidean metric on $\rea^{n}$.
Let $P \in \pol^{3}(\rea^{n})$ be harmonic and write
\begin{align}
P(x_{1}, \dots, x_{n}) = \sum_{1\leq i < j < k\leq n}\al_{ijk}x_{i}x_{j}x_{k} + \tfrac{1}{6}\sum_{1 \leq i \neq j \leq n}\be_{ij}(3x_{i}^{2}x_{j} - x_{j}^{3}).
\end{align}
For $ijk$ distinct but not ordered, define $\al_{ijk}$ to be equal to the coefficient corresponding to the ordering of $ijk$ from least to greatest. Then $P$ solves \eqref{einsteinpolynomials} with coefficient $\ka$ if and only if there hold the equations
\begin{align}
\begin{split}
0 & = - \sum_{k \neq i, k \neq j}\left(\be_{ki}\be_{ij} + \be_{kj}\be_{ji}\right) + \sum_{k \neq i, k\neq j}\be_{ki}\be_{kj}  \\
&\quad + 2\sum_{k \neq i, k \neq j}\be_{ik}\al_{ikj}+ 2\sum_{k \neq i, k \neq j}\be_{jk}\al_{jki} + 2\sum_{k, l \notin\{i, j\}, k< l}\al_{ikl}\al_{jkl},\\
&\ka = 2\sum_{k \neq i}(\be_{ki}^{2} + \be_{ik}^{2}) + 2\sum_{k \neq i, l \neq i, k < l}\be_{ki}\be_{li}+ 2\sum_{k\neq i, l \neq i, k < l}\al_{ikl}^{2}.
\end{split}
\end{align}
\end{lemma}
\begin{proof}
If $i \neq j$, 
\begin{align}
\tfrac{\pr^{2}P}{\pr x_{i}\pr x_{j}} = \sum_{k \neq i, k \neq j}\al_{ijk}x_{k} + \be_{ij}x_{i} + \be_{ji}x_{j}, 
\end{align}
and 
\begin{align}
\tfrac{\pr^{2}P}{\pr x_{i}\pr x_{i}} = -\sum_{k \neq i}\be_{ki}x_{i} + \sum_{k \neq i}\be_{ik}x_{k}. 
\end{align}
Summing the squares of these elements and simplifying yields that if $i \neq j$, the coefficient of $x_{i}x_{j}$ in $|\hess P|^{2} $ is 
\begin{align}
\begin{split}
-2&\sum_{k \neq i}\be_{ki}\be_{ij} - 2\sum_{k \neq j}\be_{kj}\be_{ji} + 2\sum_{k \neq i, k\neq j}\be_{ki}\be_{kj} + 4\be_{ij}\be_{ji} \\
&+ 4\sum_{k \neq i, k \neq j}\be_{ik}\al_{ikj}+ 4\sum_{k \neq i, k \neq j}\be_{jk}\al_{jki} + 4\sum_{k, l \notin\{i, j\}, k< l}\al_{ikl}\al_{jkl},
\end{split}
\end{align}
and the coefficient of $x_{i}^{2}$ in $|\hess P|^{2} $ is 
\begin{align}
\begin{split}
&\left(\sum_{k \neq i}\be_{ki}\right)^{2} + \sum_{k \neq i}(\be_{ki}^{2} + \be_{ik}^{2}) + \sum_{k\neq i, l \neq i, k < l}\al_{ikl}^{2}\\
& =  2\sum_{k \neq i}(\be_{ki}^{2} + \be_{ik}^{2}) + 2\sum_{k \neq i, l \neq i, k < l}\be_{ki}\be_{li}+ 2\sum_{k\neq i, l \neq i, k < l}\al_{ikl}^{2}.
\end{split}
\end{align}
The claim follows.
\end{proof}

\begin{lemma}\label{harmonicpreparationlemma}
A nontrivial cubic homogeneous polynomial $P \in \pol^{3}(\rea^{n+1})$ harmonic with respect to a Euclidean metric $h$ is equivalent modulo $O(n+1)$ to a polynomial of the form
\begin{align}\label{harmonicprenormal}
\begin{split}
P(x_{1}, \dots, x_{n+1}) &= c\left( x_{n+1}^{3}  + 3 x_{n+1}\sum_{i = 1}^{n}\la_{i}x_{i}^{2} + Q(x_{1}, \dots, x_{n})\right) \\&= c\left(\sum_{i = 1}^{n}\la_{i}\left( 3x_{n+1}x_{i}^{2} - x_{n+1}^{3}\right) + Q(x_{1}, \dots, x_{n}) \right),
\end{split}
\end{align}
where $c > 0$, $\sum_{i = 1}^{n}\la_{i} = -1$, and $\lap Q = 0$.
\end{lemma}
\begin{proof}
First it is claimed that a nontrivial cubic homogeneous polynomial $P \in \pol^{3}(\rea^{n+1})$ is equivalent modulo $O(n+1)$ to a polynomial of the form
\begin{align}\label{pprenormal0}
P = cx_{n+1}^{3} + x_{n+1}\sum_{i = 1}^{n}a_{i}x_{i}^{2} + B(x_{1}, \dots, x_{n}), 
\end{align}
where $B \in \pol^{3}(\rea^{n})$ and $c > 0$. The restriction of a nontrivial $P \in \pol^{3}(\rea^{n+1})$ to the sphere $\sn = \{x \in \rea^{n+1}: |x|  = 1\}$ has a maximum at some $e \in \sn$, and at $e$ there holds $P_{i}(e) = 3c e_{i}$ for some $c \in \rea$. Since $3P(e) = e^{i}P_{i}(e) = 3c |e| ^{2}$, were $x \leq 0$ then $P(x) \leq 0$ for all $x \in \sn$, which cannot be because $P$ is an odd function, so $c > 0$. By an orthogonal change of variables it may be supposed that the point $e$ at which the maximum occurs is $x_{1} = 0, \dots, x_{n} = 0$, $x_{n+1} = 1$, and that at this point $P_{i}(e)$ equals $\la dx^{n+1}$. Write $P$ in the form $P = cx_{n+1}^{3} + x_{n+1}^{2}L(x_{1}, \dots, x_{n}) + x_{n+1}A(x_{1}, \dots x_{n}) + B(x_{1}, \dots, x_{n})$, in which $L$, $A$, and $B$ are homogeneous polynomials on $\rea^{n}$ of degrees $1$, $2$, and $3$, respectively, and $c$ is the positive constant from before. Then $3c dx_{n+1} = dP(e) = 3c dx_{n+1} + dL$, so it must be that the linear form $L(x)$ vanishes. By an orthogonal change of variables it may be further supposed that $A$ has the form $A = \sum_{i = 1}^{n}a_{i}x_{i}^{2}$, so that, modulo $O(n+1)$, $P$ has the form \eqref{pprenormal0}. 

If $P$ is harmonic, then 
\begin{align}\label{harmonicpreduced}
\begin{split}
0 = \lap P &= \lap  B + 2(3c + \sum_{i = 1}^{n}a_{i})x_{n+1},
\end{split}
\end{align}
so $\lap B = 0$ and $\sum_{i = 1}^{n}a_{i} = -3c$. Let $\la_{i} = -\tfrac{a_{i}}{\sum_{j = 1}^{n}a_{j}} = \tfrac{a_{i}}{3c}$ and $Q = c^{-1}B$. Then $\sum_{i = 1}^{n}\la_{i} = -1$ and $P$ has the form \eqref{harmonicprenormal} where $\lap  Q = 0$.
\end{proof}

\begin{lemma}\label{preparationlemma}
A nontrivial cubic homogeneous polynomial $P \in \pol^{3}(\rea^{n+1})$ solving \eqref{einsteinpolynomials} with the constant $\ka$ is equivalent modulo $O(n+1)$ to a polynomial of the form
\begin{align}\label{pprenormal}
\begin{split}
P(x_{1}, \dots, x_{n+1}) &= c\left(\sum_{i = 1}^{n}\la_{i}\left(3x_{n+1}x_{i}^{2} - x_{n+1}^{3}\right) + Q(x_{1}, \dots, x_{n}) \right)\\
&= c\left(x_{n+1}^{3} + 3x_{n+1}\left(\sum_{i = 1}^{n}\la_{i}x_{i}^{2}\right) + Q(x_{1}, \dots, x_{n}) \right), 
\end{split}
\end{align}
where $c > 0$, $\sum_{i = 1}^{n}\la_{i} = -1$, $36c^{2}(1 + \sum_{i = 1}^{n}\la_{i}^{2}) = \ka$,  $-\tfrac{n+1}{2} \leq \la_{i} \leq \tfrac{1}{2}$, there hold the equivalent systems of inequalities
\begin{align}\label{laineq}
&1 + \sum_{i = 1}^{n}\la_{i}^{2} \geq 2\la_{k}^{2},& & 2(1 + \la_{k}) \geq  \sum_{i \neq k, j \neq k, i \neq j}\la_{i}\la_{j} ,& &1 - \la_{k}^{2} \geq \sum_{1\leq i < j \leq n}\la_{i}\la_{j}, 
\end{align}
for all $1 \leq k \leq n$, and $Q \in \pol^{3}(\rea^{n})$ satisfies
\begin{align}\label{epreduced}
\begin{split}
\lap Q &= 0, \quad \sum_{i = 1}^{n}\la_{i}\tfrac{\pr^{2}Q}{\pr^{2} x^{i}} = 0,\\
|\hess Q| ^{2}& =36(1 + \sum_{i = 1}^{n}\la_{i}^{2})E_{n}(x) -  72\sum_{i= 1}^{n}\la_{i}^{2}x_{i}^{2} ,
\end{split}
\end{align} 
where $E_{n}(x) = E_{n}(x_{1}, \dots, x_{n})$ is the quadratic form defined by the standard Euclidean metric on $\rea^{n}$.

Either at least $2$ of the $\la_{i}$ are negative, or $n-1$ of $\la_{i}$ equal $0$ and the $n$th equals $-1$. 
In the latter case, after a permutation of the indices $\{1, \dots, n\}$ it can be assumed that $\la_{1} = \dots = \la_{n-1} = 0$ and $\la_{n} = -1$, and $Q$ has the form $Q(x_{1}, \dots, x_{n}) = R(x_{1}, \dots, x_{n-1})$ where $R(x_{1}, \dots, x_{n-1}) \in \pol^{3}(\rea^{n-1})$ solves \eqref{einsteinpolynomials} with constant $\ka = 72$.
\end{lemma}

\begin{proof}
If $P$ solves \eqref{einsteinpolynomials} with the constant $\ka$, then it is in particular harmonic, so may be assumed to have the form \eqref{harmonicprenormal} with $c > 0$, $\sum_{i = 1}^{n}\la_{i} = -1$, and $\lap Q = 0$. The second equation of \eqref{einsteinpolynomials} yields
\begin{align}
\begin{split}
\ka&(E_{n}(x) + x_{n+1}^{2})  = |\hess P |^{2} \\
&= c^{2}\left(|\hess Q|^{2}  + 72\sum_{i = 1}^{n}\la_{i}^{2}x_{i}^{2}\right) + 12c^{2}\left(\sum_{i = 1}^{n}\la_{i}\tfrac{\pr^{2}Q}{\pr^{2} x^{i}}\right)x_{n+1}  +  36c^{2}(1 + \sum_{i = 1}^{n}\la_{i}^{2})x_{n+1}^{2}.
\end{split}
\end{align}
Hence there must hold
\begin{align}\label{epreduced2}
\begin{split}
&\lap Q = 0,\qquad \sum_{i = 1}^{n}\la_{i} = -1, \qquad 36c^{2}(1 + \sum_{i = 1}^{n}\la_{i}^{2}) = \ka,\\
&\sum_{i = 1}^{n}\la_{i}\tfrac{\pr^{2}Q}{\pr^{2} x^{i}} = 0,\qquad  |\hess Q| ^{2}  + 72\sum_{i= 1}^{n}\la_{i}^{2}x_{i}^{2} = 36(1 + \sum_{i = 1}^{n}\la_{i}^{2})E_{n}(x),
\end{split}
\end{align}
which yields \eqref{epreduced}. 
Since $P$ attains its maximum on $\sn$ at $e$, for $1 \leq i \leq n$, 
\begin{align}
0 \geq (\hess P(e))(\tfrac{\pr}{\pr x_{i}}, \tfrac{\pr}{\pr x_{i}}) - 3|e|^{-2}P(e)|\tfrac{\pr}{\pr x_{i}}|^{2} = 6c\la_{i} - 3c, 
\end{align}
so $\la_{i} \leq \tfrac{1}{2}$. This yields $-1 = \sum_{j = 1}^{n}\la_{j}  \leq \la _{i} + (n-1)/2$, so that $\la_{j} \geq -(n+1)/2$.

Since $|\hess Q| ^{2} \geq 0$ at all $x$, evaluating the last expression of \eqref{epreduced} at the standard basis vectors shows that the coefficients $\la_{i}$ satisfy 
\begin{align}\label{laineq0}
1 + \sum_{j = 1}^{n}\la_{j}^{2} \geq 2\la_{i}^{2}
\end{align}
for all $1 \leq i \leq n$. Using $\sum_{i= 1}^{n}\la_{i} = -1$ yields
\begin{align}\label{laineqequivalences}
\begin{split}
1 + \sum_{i = 1}^{n}\la_{i}^{2} - 2\la_{k}^{2} & = 1 + \sum_{i \neq k}\la_{i}^{2} - \la_{k}^{2}  =1 + \sum_{i \neq k}\la_{i}^{2} - \left(1 + \sum_{i \neq k}\la_{i}\right)^{2}\\
&  = - 2\sum_{i \neq k}\la_{i} - \sum_{i \neq k, j \neq k, i \neq j}\la_{i}\la_{j} = 2(1 + \la_{k}) -  \sum_{i \neq k, j \neq k, i \neq j}\la_{i}\la_{j} \\
&= 2(1 - \la_{k}^{2} - \sum_{1\leq i < j \leq n}\la_{i}\la_{j}).
\end{split}
\end{align}
With \eqref{laineq0} this shows that there hold the inequalities \eqref{laineq} and that each of these inequalities implies the others.

Suppose $m \in \{1, \dots, n\}$ is such that $\la_{m} \leq \la_{i}$ for all $i \in \{1, \dots, n\}$. Because $\sum_{i = 1}^{n}\la_{i} = -1$, it must be that $\la_{m} < 0$. There holds
\begin{align}\label{lam}
\begin{split}
0 & \geq \la_{m}^{2} - (1 + \sum_{i \neq m}\la_{i}^{2})   =\left(1 + \sum_{i \neq m}\la_{i}\right)^{2} - (1 + \sum_{i \neq m}\la_{i}^{2}) \\
& = 2\sum_{i \neq m}\la_{i} + \sum_{i \neq m, j \neq m, i \neq j}\la_{i}\la_{j}.
\end{split}
\end{align}
If $\la_{i} \geq 0$ for all $i \neq m$, then the right-hand side of \eqref{lam} is nonnegative, so it must be $\la_{i} = 0$ for $i \neq m$, in which case $\la_{m} = -1$. Otherwise, there is some index $k\neq m$ such that $\la_{k} < 0$.
Suppose $n-1$ of the $\la_{i}$ equal $0$. After a permutation of the indices it can be assumed $\la_{1} = \dots = \la_{n-1} = 0$ and $\la_{n} = 0$. Then \eqref{epreduced} implies $\tfrac{\pr^{2}Q}{\pr x_{n}\pr x_{n}} = 0$ and $|\hess Q|^{2}  = 72E_{n-1}(x_{1}, \dots, x_{n-1})$. Consequently there are $R(x_{1}, \dots, x_{n-1}) \in \pol^{3}(\rea^{n+1})$ and $B(x_{1}, \dots, x_{n-1}) \in \pol^{2}(\rea^{n-1})$ such that $Q(x_{1}, \dots, x_{n}) = R(x_{1}, \dots, x_{n-1}) + x_{n}B(x_{1}, \dots, x_{n-1})$. Then, for $1 \leq i, j \leq n-1$,
\begin{align}
&\tfrac{\pr^{2}Q}{\pr x_{i}\pr x_{j}} = \tfrac{\pr^{2}R}{\pr x_{i}\pr x_{j}} + x_{n}\tfrac{\pr^{2}B}{\pr x_{i}\pr x_{j}}, &  &\tfrac{\pr^{2}Q}{\pr x_{i}\pr x_{n}} = \tfrac{\pr B}{\pr x_{i}},
\end{align}
so
\begin{align}
\begin{split}
72&E_{n-1}(x_{1}, \dots, x_{n-1})  = |\hess Q|^{2}  \\&= |\hess R|^{2}  + 2x_{n}\sum_{1 \leq i, j \leq n-1}\tfrac{\pr^{2}R}{\pr x_{i}\pr x_{j}}\tfrac{\pr^{2}B}{\pr x_{i}\pr x_{j}} + x_{n}^{2}|\hess B|^{2} .
\end{split}
\end{align}
Consequently $|\hess B|^{2}  = 0$, and, since $\hess B$ is a constant matrix, this implies $\hess B = 0$, so $B = 0$. Hence $Q(x_{1}, \dots, x_{n}) = R(x_{1}, \dots, x_{n-1})$ is harmonic and satisfies $72E_{n-1}(x_{1}, \dots, x_{n-1})  = |\hess R|^{2} $, so $R$ solves \eqref{einsteinpolynomials} with constant $\ka = 72$.
\end{proof}

\begin{example}\label{lanminusoneexample}
The case $\la_{1} = \dots = \la_{n-1} = 0$, $\la_{n} = -1$ of Lemma \ref{preparationlemma} yields to the following (decomposable) example. Let $R(x_{1}, \dots, x_{n-1}) \in \pol^{3}(\rea^{n-1})$ solve \eqref{einsteinpolynomials} with constant $\si > 0$. Then $P(x_{1}, \dots, x_{n+1}) = x_{n+1}^{3} - 3x_{n+1}x_{n}^{2} + \tfrac{6\sqrt{2}}{\sqrt{\si}}R(x_{1}, \dots, x_{n-1})$ solves \eqref{einsteinpolynomials} with constant $\ka = 72$. This follows upon taking $Q(x_{1}, \dots, x_{n}) = \tfrac{6\sqrt{2}}{\sqrt{\si}}R(x_{1}, \dots, x_{n-1})$ in Lemma \ref{preparationlemma}. Since there is no solution of \eqref{einsteinpolynomials} in dimension $1$, this construction works for $n+1 \geq 4$ (because it requires $n-1 \geq 2$). On the other hand, it is evident that $P$ is orthogonally decomposable, as it has the form of the solutions given in Example \ref{splitexample}. 
\end{example}

\begin{lemma}\label{qpreplemma}
Let $h_{ij}$ be a Euclidean metric on $\rea^{n}$.
Consider a harmonic cubic polynomial
\begin{align}
Q(x_{1}, \dots, x_{n}) = \sum_{1\leq i < j < k\leq n}\al_{ijk}x_{i}x_{j}x_{k} + \tfrac{1}{6}\sum_{1 \leq i \neq j \leq n}\be_{ij}(3x_{i}^{2}x_{j} - x_{j}^{3}) \in \pol^{3}(\rea^{n})
\end{align}
where $a_{ijk}$ is completely symmetric in its indices. Let $\sum_{i = 1}^{n}\la_{i} = -1$ and write $\la_{ij} = \la_{i} - \la_{j}$. Then $Q$ solves \eqref{epreduced} if and only if there hold the equations
\begin{align}
\begin{split}
0 & = \sum_{j\neq i}\la_{ji}\be_{ji}, \quad 1\leq i \leq n,\\
0 & = -\sum_{k \neq i, k \neq j}\left(\be_{ki}\be_{ij} + \be_{kj}\be_{ji}\right) + \sum_{k \neq i, k\neq j}\be_{ki}\be_{kj}  \\
&\quad+ 2\sum_{k \neq i, k \neq j}\be_{ik}\al_{ikj}+ 2\sum_{k \neq i, k \neq j}\be_{jk}\al_{jki} + 2\sum_{k, l \notin\{i, j\}, k< l}\al_{ikl}\al_{jkl},\\
18\left(1 - \la_{i}^{2} + \sum_{j \neq i}\la_{j}^{2}\right)& = \sum_{k \neq i}(\be_{ki}^{2} + \be_{ik}^{2}) + \sum_{k \neq i, l \neq i, k < l}\be_{ki}\be_{li}+ \sum_{k\neq i, l \neq i, k < l}\al_{ikl}^{2}, \quad 1 \leq i \leq n.
\end{split}
\end{align}
\end{lemma}
\begin{proof}
This follows straightforwardly from Lemmas \ref{harpreplemma} and \ref{preparationlemma}.
\end{proof}

\begin{remark}
In general it is not clear to what extent either Lemma \ref{harpreplemma} or \ref{qpreplemma} is useful for obtaining solutions in general. These equations recall those defining a simple real Lie algebra in the following sense. The Jacobi identity is a system of quadratic equations in the structure constants of the algebra, and the nondegeneracy of the Killing form yields further quadratic equations in the structure constants. These quadratic equations are not usually the basis for the study of Lie algebras, except in low dimensions (and in the presence of further auxiliary conditions, such as solvability), in which case the small number of variables makes them susceptible to direct study. Similarly, in the $3$ and $4$ dimensional cases the equations given by Lemmas \ref{harpreplemma} and \ref{qpreplemma} are amenable to direct solution, as is illustrated by the proofs of Theorems \ref{3dpolytheorem} and \ref{4dpolytheorem}. 
\end{remark}

Using Lemma \ref{preparationlemma} it can be shown that there exist nontrivial solutions to \eqref{einsteinpolynomials} for all $n \geq 2$. 

\begin{lemma}\label{epreductionlemma}
If a homogeneous cubic polynomial $P(x_{1}, \dots, x_{n+1}) \in \pol^{3}(\rea^{n+1})$ solving \eqref{einsteinpolynomials} with constant $\ka$ has the form \eqref{pprenormal} with $\la_{1} = \dots = \la_{n}$, then $\ka  = 36\tfrac{n+1}{n}c^{2}$ and $Q(x_{1}, \dots, x_{n}) \in \pol^{3}(\rea^{n})$ solves \eqref{einsteinpolynomials} with constant $36\tfrac{(n+2)(n-1)}{n^{2}}$.
\end{lemma}
\begin{proof}
If all $\la_{i}$ are equal, they equal $-n^{-1}$. In this case the condition $\sum_{i = 1}^{n}\la_{i}\tfrac{\pr^{2}Q}{\pr^{2} x^{i}} = 0$ of \eqref{epreduced} is redundant with the condition $\lap Q = 0$, and $|\hess Q| ^{2} =36(1 + \sum_{i = 1}^{n}\la_{i}^{2})E_{n}(x) -  72\sum_{i= 1}^{n}\la_{i}^{2}x_{i}^{2} = 36\tfrac{(n+2)(n-1)}{n^{2}}E_{n}(x)$, so $Q$ solves \eqref{einsteinpolynomials} with constant $72\tfrac{n-1}{n}$. 
\end{proof}
Lemma \ref{epreductionlemma} motivates Lemma \ref{recursivepolynomiallemma}.

\begin{lemma}\label{recursivepolynomiallemma}
Suppose $Q_{n}(x) \in \pol^{3}(\rea^{n})$ solves \eqref{einsteinpolynomials} with constant $\ka_{n}$. Define
\begin{align}\label{qns}
\begin{split}
Q_{n+1}&(x, x_{n+1})  =  \sqrt{\tfrac{\ka_{n+1}}{(n+1)n}}\left(\tfrac{1}{6}\sum_{i = 1}^{n}\left(x_{n+1}^{3} - 3x_{n+1}x_{i}^{2}\right) + \sqrt{\tfrac{(n+2)(n-1)}{\ka_{n}}}Q_{n}(x)\right)\\
& = \sqrt{\tfrac{\ka_{n+1}}{(n+1)n}}\left(\tfrac{n+3}{6}\left(x_{n+1}^{3} - \tfrac{3}{n+3}x_{n+1}E_{n+1}(x, x_{n+1})\right) + \sqrt{\tfrac{(n+2)(n-1)}{\ka_{n}}}Q_{n}(x)\right)
\end{split}
\end{align}
where $x \in \rea^{n}$ and $E_{n+1}(x, x_{n+1}) = E_{n}(x) + x_{n+1}^{2} = |x|^{2} + x_{n+1}^{2}$ are the quadratic forms determined by the metrics $\hat{h}(\hat{x}, \hat{y}= h(x, y) + x_{n+1}y_{n+1}$ and $h$ on $\rea^{n+1}$ and $\rea^{n}$.
\begin{enumerate}
\item\label{rcp0} $Q_{n+1}$ solves \eqref{einsteinpolynomials} with constant $\ka_{n+1}$, 
\item\label{gqns} If $g \in O(n)$ then $\hat{Q}_{n+1}$ defined by \eqref{qns} with $g\cdot Q_{n}$ in place of $Q_{n}$ equals $i(g)\cdot Q_{n+1}$ where $i:O(n) \to O(n+1)$ is the inclusion as the subgroup fixing $x_{n+1}$.
\item\label{rcp1} The critical lines of $Q_{n+1}$ not generated by $(0, 1)$ (which is a critical line of $Q_{n+1}$) come in pairs of distinct critical lines projecting orthogonally onto critical lines of $Q_{n}$. In particular if the cardinality $c_{n}$ of $\critline(Q_{n})$ is finite, then the cardinality $c_{n+1}$ of $\critline(Q_{n+1})$ is $c_{n+1} = 2c_{n} + 1$. Consequently $\Aut(Q_{n+1})$ is finite if and only if $\Aut(Q_{n})$ is finite.
\item\label{rcp2} The conformal nonassociativity tensors of $Q_{n}$ and $Q_{n+1}$ are related by
\begin{align}\label{confextnonass}
\cass(Q_{n+1})(\hat{x}, \hat{y}, \hat{z}, \hat{w}) = \tfrac{\ka_{n+1}}{\ka_{n}}\tfrac{(n+2)(n-1)}{(n+1)n}\cass(Q_{n})(x, y, z, w).
\end{align}
for $\hat{x}= (x, x_{n+1}), \hat{y}= (y, y_{n+1}), \hat{z}= (z, z_{n+1}), \hat{w}= (w, w_{n+1}) \in \rea^{n+1}$. In particular, $Q_{n+1}$ is conformally associative if and only if $Q_{n}$ is conformally associative.
\end{enumerate}
\end{lemma}

\begin{proof}
Consider $P$ and $Q$ as in \eqref{pprenormal} of Lemma \ref{preparationlemma} and define $Q_{n+1}(x, x_{n+1})$ and $Q_{n}(x)$ by $P(x, x_{n+1}) = c_{n+1}Q_{n+1}(x, x_{n+1})$ and $Q(x) = b_{n}^{-1}Q_{n}(x)$ where $x \in \rea^{n}$ and $b_{n}$ and $c_{n}$ are nonzero constants to be specified presently. Assume $Q_{n}$ solves \eqref{einsteinpolynomials} with the constant $\ka_{n}$. With these assumptions there holds
\begin{align}
Q_{n+1}(x, x_{n+1}) = c_{n+1}\left(\tfrac{1}{n}\sum_{i = 1}^{n}(x_{n+1}^{3} - 3x_{n+1}x_{i}^{2}) + b_{n}^{-1}Q_{n}(x)\right)
\end{align}
and it follows from \eqref{epreduced} that $\ka_{n+1} = 36 c_{n+1}^{2}\tfrac{n+1}{n}$ and $\ka_{n} = 36b_{n}^{2}\tfrac{(n+2)(n-1)}{n^{2}}$. Expressing $c_{n+1}$ and $b_{n}$ in terms of $\ka_{n+1}$ and $\ka_{n}$ yields claim \eqref{rcp0}.

If $g \in O(n)$, then $g \cdot Q_{n}$ solves \eqref{einsteinpolynomials} with constant $\ka_{n}$, and the equation \eqref{qns} with $g \cdot Q_{n}$ in place of $Q_{n}$ defines $\hat{Q}_{n+1}$ solving \eqref{einsteinpolynomials} with constant $\ka_{n+1}$. It is apparent from the second equality of \eqref{qns} that $\hat{Q}_{n+1} = i(g)\cdot Q_{n+1}$ where $i:O(n) \to O(n+1)$ is the inclusion as the subgroup fixing $x_{n+1}$. This shows \eqref{gqns}.

That $(x, x_{n+1})$ generate a critical line of $Q_{n+1}$ yields the equations
\begin{align}\label{qcrit}
&(\theta + x_{n+1})x  = \sqrt{\tfrac{(n+2)(n-1)}{\ka_{n}}} DQ_{n}(x),& &nx_{n+1}^{2} - 2\theta x_{n+1} - E_{n}(x) = 0,
\end{align}
for some $\theta \in \rea$. Plainly $(0, 1)$ is a solution. Suppose $x \neq 0$. The first equation of \eqref{qcrit} shows that $x$ generates a critical line of $Q_{n}$. The second equation is a quadratic equation in $x_{n+1}$ with positive discriminant, having the two distinct real roots of opposite signs, 
\begin{align}
\al^{\pm} = \tfrac{1}{n}(\theta \pm (\theta^{2} + nE_{n}(x))^{1/2}).
\end{align}
The vectors $(x, \al^{\pm})$ generate distinct critical lines of $Q_{n+1}$. The remaining claims in \eqref{rcp1} follow.

As in the proof of Lemma \ref{asshatlemma}, let $\hat{L}(\hat{x})$ be the endomorphism of $\rea^{n+1}$ satisfying $(\hess \hat{P})(\hat{x}) = \hat{h}(\hat{L}(\hat{x})(\dum), \dum)$. The matrix of $\hat{L}(\hat{x})$ with respect to an orthonormal basis of $\rea^{n+1}$ has the form
\begin{align}\label{hatlka}
\hat{L}(\hat{x}) = \sqrt{\tfrac{\ka_{n+1}}{(n+1)n}}\begin{pmatrix} \sqrt{\tfrac{(n+2)(n-1)}{\ka_{n}}}L(x) - x_{n+1} \id & -x \\ -h(x, \dum) & nx_{n+1} \end{pmatrix}.
\end{align}
Because $L(x) y = L(y)x$, from \eqref{hatlka} there follows
\begin{align}
\begin{split}
[&\hat{L}(\hat{x}), \hat{L}(\hat{y})]\hat{z}  \\
&= \tfrac{\ka_{n+1}}{(n+1)n}\left(\tfrac{(n+2)(n-1)}{\ka_{n}} [L(x), L(y)] z + h(y, z)x - h(x, z)y  + (n+1)z_{n+1}(x_{n+1}y - y_{n+1}x), \right.\\
&\qquad \qquad \left.(n+1)(x_{n+1}h(y, z) - y_{n+1}h(x, z)) \right).
\end{split}
\end{align}
There results
\begin{align}
\begin{split}
\cass&(Q_{n+1})(\hat{x}, \hat{y}, \hat{z}, \hat{w})  = \hat{h}([\hat{L}(\hat{x}), \hat{L}(\hat{y})]\hat{z}, \hat{w}) -\tfrac{\ka_{n+1}}{n}\left(\hat{h}(\hat{x}, \hat{z})\hat{h}(\hat{y}, \hat{w}) - \hat{h}(\hat{y}, \hat{z})\hat{h}(\hat{x}, \hat{w}) \right)\\
& = \tfrac{\ka_{n+1}}{\ka_{n}}\tfrac{(n+2)(n-1)}{(n+1)n}\left(\cass(Q_{n})(x, y, z, w) + \tfrac{\ka_{n}}{n-1}\left(h(x, z)h(y, w) - h(y, z)h(x, w) \right)\right)\\
&\quad + \tfrac{\ka_{n+1}}{(n+1)n}\left(h(y, z)h(x, w) - h(x, z)h(y, w) \right) \\
&\quad + \tfrac{\ka_{n+1}}{n}\left(x_{n+1}z_{n+1}h(y, w) - y_{n+1}z_{n+1}h(x, w) + x_{n+1}z_{n+1}h(y, z) - y_{n+1}w_{n+1}h(x, z)\right)\\
&\quad  -\tfrac{\ka_{n+1}}{n}\left(\hat{h}(\hat{x}, \hat{z})\hat{h}(\hat{y}, \hat{w}) - \hat{h}(\hat{y}, \hat{z})\hat{h}(\hat{x}, \hat{w}) \right)\\
 &= \tfrac{\ka_{n+1}}{\ka_{n}}\tfrac{(n+2)(n-1)}{(n+1)n}\cass(Q_{n})(x, y, z, w),
\end{split}
\end{align}
which shows \eqref{confextnonass}.
\end{proof}
Corollary \ref{ndimexistencecorollary} shows that \eqref{einsteinpolynomials} admits nontrivial orthogonally indecomposable solutions for all $n \geq 2$ and all $\ka > 0$.

The elementary Lemma \ref{calculuslemma} is needed in the proof of Corollary \ref{ndimexistencecorollary}.

\begin{lemma}\label{calculuslemma}
On $[-1, 1]$ the function $f(r) = (n+3)r^{3} - 3r + (n-1)\sqrt{\tfrac{n+2}{n}}(1 - r^{2})^{\tfrac{3}{2}}$ attains its maximum value at $r = 1$ and $r = -1/(n+1)$, and this maximum value is $n$.
\end{lemma}
\begin{proof}
There hold $f(-1) = -n$ and $f(1) = n$. A critical point in $r \in (-1, 1)$ solves $0 = f^{\prime}(r)/3 = (n+3)r^{2} - 1 - (n-1)\sqrt{\tfrac{n+2}{n}}r(1-r^{2})^{1/2}$. Squaring the equality $(n+3)r^{2} - 1 = (n-1)\sqrt{\tfrac{n+2}{n}}r(1-r^{2})^{1/2}$ and reorganizing the result shows that $0 = \left((n+1)^{2}r^{2} - 1\right)\left(2(n+1)^{2}r^{2} - n(n+1)\right)$ so that $r^{2} = 1/(n+1)$ or $r^{2} = n/(2(n+1))$. It is straightforward to check that of the $4$ solutions of these last two equations only $r = -1/(n+1)$ and $r = \sqrt{\tfrac{n}{2(n+1)}}$ are critical points of $f$. There hold $f(-1/(n+1)) = n$ and $f\left(\sqrt{\tfrac{n}{2(n+1)}}\right) = \tfrac{(n-2)\sqrt{n+1}}{\sqrt{2n}} < n$. The claim follows.
\end{proof}

Recall from \eqref{extremepdefined} the notation $\extreme(P)$ for the set of points at which $P$ assumes its maximum on the unit sphere.

\begin{corollary}\label{ndimexistencecorollary}
For $n \geq 2$ the polynomial $P_{n}(x) \in \pol^{3}(\rea^{n})$ defined in \eqref{simplicialpoly}
has the following properties:
\begin{enumerate}
\item\label{simp1} $P_{n}$ solves \eqref{einsteinpolynomials} with $\ka = n(n-1)$.
\item\label{simp2} $P_{n}$ is conformally associative.
\item\label{simp3} $\mkc(P_{n}) = \tfrac{n}{n-1}$. Equivalently, $\sup_{|x|^{2} = 1}P_{n}(x) = \tfrac{n-1}{6}$.
\item\label{simp4} The set $\extreme(P_{n})$ has cardinality $n+1$.
\item\label{simp4b} $\sum_{v \in \extreme(P_{n})}v = 0$.
\item\label{simp5} If $y \neq z \in \extreme(P_{n})$, then $\lb y, z \ra = -1/n$.
\item\label{simp5c} $\sum_{v \in \extreme(P_{n})}h(x, v)v = \tfrac{n+1}{n}x$ for all $x \in \rea^{n}$.
\item\label{simp5b} If $v \in \extreme(P_{n})$ then $P_{n}(v)_{ij} = nv_{i}v_{j} - h_{ij}$.
\item\label{simp6} $\Aut(P_{n})$ equals the symmetric group $S_{n+1}$ acting by permutations of $\extreme(P_{n})$ and is generated by the orthogonal reflections through the hyperplanes perpendicular to the differences $y  - z$ for $y, z \in \extreme(P_{n})$.
\item\label{simp7} $P_{n}$ is orthogonally indecomposable.
\item\label{simp8}
The critical points of the restriction to the unit sphere of $P_{n}$ have the form
\begin{align}\label{simplicialcritical}
v = \sqrt{\tfrac{n}{k(n+1-k)}}\sum_{i \in I}v_{i}, 
\end{align}
where $\extreme(P_{n}) = \{v_{0}, \dots, v_{n}\}$ and $I \subset \{0, \dots, n\}$ has cardinality $1 \leq k \leq n$. There holds $6P_{n}(v) = (n+1-2k)\sqrt{\tfrac{n}{k(n+1-k)}}$.
\end{enumerate}
\end{corollary}

\begin{proof}
The proof of claims \eqref{simp1}-\eqref{simp7} is by induction on $n$ with base case $n = 2$ for which all the claims are straightforward to verify via direct computations. There are written $x \in \rea^{n}$ and $\hat{x} = (x, x_{n+1}) \in \rea^{n+1}$. The formula \eqref{simplicialpoly} defines $P_{n+1}$ inductively in terms of $P_{n}$ by
\begin{align}\label{pnrecursion}
\begin{split}
P_{n+1}&(x, x_{n+1}) = \tfrac{1}{6}\sum_{i = 1}^{n}(x_{n+1}^{3} - 3x_{n+1}x_{i}^{2}) + \sqrt{\tfrac{n+2}{n}}P_{n}(x)\\
& = \tfrac{n+3}{6}\left(x_{n+1}^{3} - \tfrac{3}{n+3}x_{n+1}E_{n+1}(x, x_{n+1})\right) + \sqrt{\tfrac{n+2}{n}}P_{n}(x).
\end{split}
\end{align}
which has the form \eqref{qns} with $\ka_{n} = n(n-1)$ and $\ka_{n+1} = (n+1)n$,
so that, by Lemma \ref{recursivepolynomiallemma}, $P_{n+1}$ solves \eqref{einsteinpolynomials} with constant $\ka_{n+1}$ if $P_{n}$ solves \eqref{einsteinpolynomials} with constant $\ka_{n}$. Since $P_{2}(x_{1}, x_{2}) = \tfrac{1}{6}(x_{2}^{3} - 3x_{2}x_{1}^{2})$ solves \eqref{einsteinpolynomials} with $\ka_{2} = 2$, it follows by induction on $n$ that $P_{n}$ solves \eqref{einsteinpolynomials} with constant $\ka_{n} = n(n+1)$. Likewise, that $P_{n}$ is conformally associative follows from Lemma \ref{recursivepolynomiallemma} by induction on $n$, the base case $n = 2$ being immediate. 

Suppose $\hat{x} \in\extreme(P_{n})$. Because $|x|^{2} + x_{n+1}^{2} = 1$ there holds
\begin{align}\label{dpn}
\begin{split}
dP_{n+1}(x, x_{n+1}) &= -\sum_{i = 1}^{n}x_{n+1}x_{i}dx_{i} + \sqrt{\tfrac{n+2}{n}}dP_{n}(x) + \left( \tfrac{n}{2}x_{n+1}^{2} - \tfrac{1}{2}|x|^{2}\right)dx_{n+1}\\
& = -\sum_{i = 1}^{n}x_{n+1}x_{i}dx_{i} + \sqrt{\tfrac{n+2}{n}}dP_{n}(x) + \left( \tfrac{n+1}{2}x_{n+1}^{2} - \tfrac{1}{2}\right)dx_{n+1}.
\end{split}
\end{align}
Since $\hat{x} \in \extreme(P_{n}) \subset \critline(P_{n+1})$ it follows from \eqref{dpn} that there is $\theta \in \rea$ so that
\begin{align}
\label{mkc1}
\begin{split}
P_{n}(x)_{i} &= \sqrt{\tfrac{n}{n+2}}(x_{n+1}+ \theta)x_{i} ,\\
\end{split}\\
\label{mkc2}
\begin{split}
0 &= nx_{n+1}^{2} - 2\theta x_{n+1} - |x|^{2} = (n+1)x_{n+1}^{2} - 2\theta x_{n+1} - 1.
\end{split}
\end{align}
The homogeneity of $P_{n+1}$ implies that $\theta = 3|\hat{x}|^{-2}P_{n+1}(\hat(x)) = 3P_{n+1}(\hat{x})$.
The equation \eqref{mkc1} implies that $x$ generates a critical line of $P_{n}$. By the inductive hypothesis,
\begin{align}\label{mkc3}
P_{n}(x)  = |x|^{3}P_{n}(\tfrac{x}{|x|}) \leq  |x|^{3}\sup_{|z| = 1}P_{n}(z) \leq \tfrac{n-1}{6}|x|^{3},
\end{align}
so 
\begin{align}\label{mkc4}
\begin{split}
P_{n+1}(x, x_{n+1}) &\leq  \tfrac{1}{6}\left((n+3)x_{n+1}^{3} - 3x_{n+1}+ (n-1)\sqrt{\tfrac{n+2}{n}}|x|^{3}\right) \\
&=\tfrac{1}{6}\left((n+3)x_{n+1}^{3} - 3x_{n+1}+ (n-1)\sqrt{\tfrac{n+2}{n}}(1 - x_{n+1}^{2})^{\tfrac{3}{2}}\right) \leq \tfrac{n}{6},
\end{split}
\end{align}
where the last inequality follows from Lemma \ref{calculuslemma}.

Since $P(0, 1) = n/6$ and every element of $\extreme(P_{n+1})$ must have the same norm, equality must hold in \eqref{mkc4}, so $6P_{n+1}(\hat{x}) = n$. By \eqref{einsteinmkc} this is equivalent to $\mkc(P_{n+1}) = \tfrac{n+1}{n}$. In \eqref{mkc2} this forces $\theta = 3P_{n+1}(\hat{x}) = n/2$. If $x_{n+1} \neq 1$, then $x_{n+1} = -1/(n+1)$ and $|x|^{2} = \tfrac{n(n+2)}{(n+1)^{2}}$. Substituting theses values in \eqref{pnrecursion} yields
\begin{align}
6\sqrt{n+2}{n}P_{n}(\tfrac{x}{|x|}) = n + \tfrac{n+3}{(n+1)^{3}} - \tfrac{3}{n+1} = \tfrac{(n-1)n(n+2)^{2}}{(n+1)^{3}},
\end{align}
so that
\begin{align}
6P_{n}(\tfrac{x}{|x|}) = 6|x|^{-3}P_{n}(x) = \tfrac{(n+1)^{3}}{n^{3/2}(n+1)^{3/2}}\tfrac{(n-1)n^{3/2}(n+2)^{3/2}}{(n+1)^{3}} = n-1.
\end{align}
By the inductive hypothesis, this shows $\tfrac{x}{|x|} \in \extreme(P_{n})$.

The preceding shows that for each $\hat{x} \in \extreme(P_{n+1})$ for which $x_{n+1} \neq 1$, $\tfrac{x}{|x|} \in \extreme(P_{n})$. By the inductive hypothesis $\extreme(P_{n})$ has cardinality $n+1$, so $\extreme(P_{n+1})$ has cardinality $n+2$. This shows claim \eqref{simp4}. That $\sum_{v \in \extreme(P_{n})}v = 0$ follows by induction. Claim \eqref{simp5} is proved by straightforward induction.

In general,
\begin{align}\label{hatla}
\begin{split}
\hess P_{n+1}(\hat{x}) &= \begin{pmatrix} \sqrt{\tfrac{n+2}{n}}\hess P_{n}(x) - x_{n+1} \id & -x \\ -h(x, \dum) & nx_{n+1} \end{pmatrix}.
\end{split}
\end{align}
If $\hat{x} = (x, -1/(n+1))$ with $\tfrac{x}{|x|} \in \extreme(P_{n})$, then by the inductive hypothesis and \eqref{simp5b}, $P_{n}(\tfrac{x}{|x|})_{ij} = n\tfrac{x_{i}x_{j}}{|x|^{2}} - h_{ij}$, so 
\begin{align}\label{simp5bused}
\sqrt{\tfrac{n+2}{n}}P_{n}(x)_{ij} = (n+1)x_{i}x_{j} - \tfrac{n+2}{n+1}h_{ij}.
\end{align}
By the inductive hypothesis, in the form of claim \eqref{simp5c}, and \eqref{hatla}
\begin{align}\label{hatlb}
\begin{split}
\sqrt{\tfrac{n+2}{n}}(\hess P_{n})(x) &= \sqrt{\tfrac{n+2}{n}}|x|(\hess P_{n})(\tfrac{x}{|x|}) = \tfrac{n+2}{n+1}\left(nh(\tfrac{x}{|x|}, \dum)\tensor h(\tfrac{x}{|x|}, \dum) - \Id\right)\\
& = (n+1)h(x, \dum)\tensor h(x, \dum) -  \tfrac{n+2}{n+1}\Id.
\end{split}
\end{align}
By \eqref{hatl} and \eqref{hatlb},
\begin{align}
\begin{split}
\hess P_{n+1}(x, -\tfrac{1}{n+1}) 
&= \begin{pmatrix} \sqrt{\tfrac{n+2}{n}}\hess P_{n}(x) + \tfrac{1}{n+1} \id & -x \\ -h(x, \dum) & -\tfrac{n}{n+1} \end{pmatrix}\\
&= \begin{pmatrix}(n+1)h(x, \dum)\tensor h(x, \dum) -  \id & -x \\ -h(x, \dum) & -\tfrac{n}{n+1} \end{pmatrix} \\
&= (n+1)\hat{h}(\hat{x},\dum)\tensor \hat{h}(\hat{x}, \dum) - \Id,
\end{split}
\end{align}
which proves \eqref{simp5b}.

To prove \eqref{simp5c}, let $\hat{x} \in \extreme(P_{n+1})$ and $\hat{y} \in \rea^{n+1}$. Write $\hat{x} = (\tfrac{\sqrt{n(n+2)}}{n+1}x, -\tfrac{1}{n+1})$ so that $x \in \extreme(P_{n})$. Then, because $\sum_{x \in \extreme(P_{n})}x = 0$,
\begin{align}
\begin{split}
\sum_{\hat{x}  \in  \extreme(P_{n+1})}&\hat{h}(\hat{y}, \hat{x})\hat{x} = (0, \hat{h}(\hat{y}, (0, 1) )) + \sum_{x \in \extreme(P_{n})}(\tfrac{\sqrt{n(n+2)}}{n+1}h(y, x) - \tfrac{1}{n+1}y_{n+1})(\tfrac{\sqrt{n(n+2)}}{n+1}x, -\tfrac{1}{n+1})\\
& = (0, x_{n+1}) + \left(\tfrac{n(n+2)}{(n+1)^{2}}\sum_{x \in \extreme(P_{n})}h(y, x)x ,  \tfrac{1}{n+1}y_{n+1}\right) = \tfrac{n+2}{n+1}\hat{y},
\end{split}
\end{align}
the last equality by the inductive hypothesis. This shows \eqref{simp5c}.

Because $P_{n}$ solves \eqref{einsteinpolynomials}, the automorphism group $\Aut(P_{n})$ acts by orthogonal transformations permuting $\extreme(P_{n})$ so is isomorphic to a subgroup of $S_{n+1}$. 

It is claimed that $P_{n+1}(\hat{x} - \hat{y}) = 0$ for $\hat{x}, \hat{y} \in \extreme(P_{n+1})$. First, if $x_{n+1} = y_{n+1} = -1/(n+1)$, then
\begin{align}
P_{n+1}(\hat{x} - \hat{y}) = P_{n+1}(x - y, 0) = \sqrt{\tfrac{n+2}{n}}P_{n}(x - y) = 0
\end{align}
the last equality by the inductive hypothesis because $\tfrac{x}{|x|}, \tfrac{y}{|y|} \in \extreme(P_{n})$. From \eqref{hatl} there follows
\begin{align}\label{hatl2}
\hess P_{n+1}(\hat{x} - \hat{y}) & = \hess P_{n+1}(x - y, 0) = \begin{pmatrix} \sqrt{\tfrac{n+2}{n}}\hess P_{n}(x-y)  & y-x \\ h(y-x, \dum) & 0 \end{pmatrix}.
\end{align}
A $\hat{z}  \in \rea^{n+1}$ is orthogonal to $\hat{x} - \hat{y}$ if and only if $h(x, z) = h(y, z)$ in which case, by Lemma \ref{zeroautomorphismlemma} and the inductive hypothesis $z^{i}z^{j}P_{n}(x-y)_{ij} = 0$, which shows that $\hat{r} = \hat{x} - \hat{y}$ satisfies the condition of Lemma \ref{zeroautomorphismlemma} and so the reflection through the hyperplane orthogonal to $\hat{r}$ is an automorphism of $P_{n+1}$. 

Now suppose $\hat{y} = (0, 1)$. Then $P_{n+1}(\hat{x} - \hat{y}) = P_{n+1}(x , -\tfrac{n+2}{n+1}) = 0$ follows from \eqref{pnrecursion}, $|x|^{2} = \tfrac{n(n+2)}{(n+1)^{2}}$, and the fact that, since $\tfrac{x}{|x|} \in \extreme(P_{n})$, there holds $P_{n}(\tfrac{x}{|x|}) = \tfrac{n-1}{6}$ by the inductive hypothesis. From \eqref{hatl} there follows
\begin{align}\label{hatl3}
\hess P_{n+1}(\hat{x} - \hat{y}) & = \hess P_{n+1}(x ,  -\tfrac{n+2}{n+1}) = \begin{pmatrix} \sqrt{\tfrac{n+2}{n}}\hess P_{n}(x) + \tfrac{n+2}{n+1}\Id & -x \\ -h(x, \dum) & -\tfrac{n(n+2)}{n+1} \end{pmatrix}.
\end{align}
A $\hat{z} \in \rea^{n+1}$ is orthogonal to $\hat{x} - \hat{y}$ if and only if $h(x, z) = \tfrac{n+2}{n+1}z_{n+1}$ in which case, by \eqref{hatl3} and \eqref{simp5bused},
\begin{align}
\label{hatl3b}
\begin{split}
\hess P_{n+1}(x ,  -\tfrac{n+2}{n+1}) (\hat{z}, \hat{z})& =\sqrt{\tfrac{n+2}{n}}z^{i}z^{j}P_{n}(x)_{ij} + \tfrac{n+2}{n+1}|z|^{2} - \tfrac{(n+2)^{2}}{n+1}z_{n+1}^{2}\\
& = (n+1)h(x, z)^{2} -  \tfrac{(n+2)^{2}}{n+1}z_{n+1}^{2} = 0.
\end{split}
\end{align} 
This shows that $\hat{r} = \hat{x} - \hat{y}$ satisfies the condition of Lemma \ref{zeroautomorphismlemma} and so the reflection through the hyperplane orthogonal to $\hat{r}$ is an automorphism of $P_{n+1}$. 

It has been shown that for $\hat{x}, \hat{y} \in \extreme(P_{n+1})$ the pairwise difference $\hat{r} = \hat{x} - \hat{y}$ satisfies the hypotheses of Lemma \ref{zeroautomorphismlemma}, so the reflection through the hyperplane orthogonal to $\hat{r}$ is an automorphism of $P_{n+1}$. This reflection interchanges $\hat{x}$ and $\hat{y}$ and it is straightforward to check that it fixes $\extreme(P_{n+1}) \setminus\{\hat{x}, \hat{y}\}$, so acts as a transposition of $\extreme(P_{n+1})$. It follows that these reflections generate a group isomorphic to $S_{n+2}$ acting by automorphisms of $P_{n+1}$, and that all such automorphisms arise in this way.

Were $P_{n}$ were orthogonally decomposable, then its automorphism group would preserve the orthogonal decomposition, but the preceding shows that the action of $\Aut(P_{n}) = S_{n+1}$ on $\rea^{n}$ is its irreducible representation of dimension $n$. 

There remains to prove \eqref{simp8}. Let $\extreme(P_{n}) = \{v_{0}, \dots, v_{n}\}$ and for a cardinality $k$ subset $I  \subset \{0, \dots, n\}$ let $u_{I} = \sum_{i \in I}v_{i}$. By \eqref{simp5}, $(\hess P_{n}(v_{i}))(v_{j}, \dum) = -h(v_{i} + v_{j}, \dum)$ and $(\hess P_{n}(v_{i}))(v_{i}, \dum) = (n-1)h(v_{i}, \dum)$ for $0 \leq i \neq j \leq n$. The algebra structure given by the nonassociative commutative product $\mlt$ defined by $v_{i}\mlt v_{j} = -v_{i} - v_{j}$ if $i \neq j$ and $v_{i} \mlt v_{i} = (n-1)v_{i}$ was studied by K. Harada in \cite{Harada}. The square-zero and idempotent elements of this algebra generate the critical lines of $P_{n}$. Using \eqref{simp4b}, it follows that $(\hess P_{n}(u_{I}))(u_{I}, \dum)  = (n+1-2k)h(u_{I}, \dum)$, which shows that $u_{I}$ generates a critical line of $P_{n}$ and satisfies $6P_{n}(u_{I}) = (n+1 - 2k)|u_{I}|^{2}$, and it follows from Lemma $2$ and Corollary $3$ of \cite{Harada} that any element square-zero or idempotent with respect to $\mlt$ is a multiple of some $u_{I}$. By \eqref{simp5}, $|u_{I}|^{2} = \tfrac{k(n+1-k)}{n}$, so the rescaled element $v_{I} = \sqrt{\tfrac{n}{k(n+1-k)}}u_{I}$ has unit norm and satisfies $6P_{n}(v_{I}) = (n+1-2k)\sqrt{\tfrac{n}{k(n+1-k)}}$. This proves \eqref{simp8}.
\end{proof}

\begin{proof}[Proof of Theorem \ref{ealgpolynomialtheorem}]
By \eqref{rcp2} of Lemma \ref{recursivepolynomiallemma}, for any $0 < \ka \in \rea$ there exists a conformally associative solution of \eqref{einsteinpolynomials} in any dimension $n \geq 2$. The polynomial \eqref{simplicialpoly} is such a solution; when $n = 3$ any solution is conformally associative, and from Lemma \ref{recursivepolynomiallemma} it follows that \eqref{simplicialpoly} is conformally associative for any $n \geq 3$. By Theorem \ref{confassequivalencetheorem} any two such solutions are orthogonally equivalent. 
\end{proof}

\begin{remark}
Given $Q_{n}$ solving \eqref{einsteinpolynomials} with constant $\ka_{n}$, the equation \eqref{qns} defines $Q_{n+1}$ solving \eqref{einsteinpolynomials} with constant $\ka_{n+1}$. 
Starting with any seed polynomial solving \eqref{einsteinpolynomials} in any given dimension, Lemma \ref{recursivepolynomiallemma} constructs a sequence of polynomials of increasingly higher dimensions solving  \eqref{einsteinpolynomials}. The resulting sequence is determined up to orthogonal equivalence by the orthogonal equivalence class of the seed. More interesting examples can be constructed provided there can be found more interesting seeds for the iteration. 
\end{remark}

\begin{example}\label{pmnexample}
The tensor product of the degree $m$ and $n$ simplicial polynomials is not equivalent to the degree $mn$ simplicial polynomial.
\begin{lemma}\label{pmnlemma}
For $m, n \geq 2$ the orbits $\orb{P_{m}\tensor P_{n}}$ and $\orb{P_{mn}}$ are not $CO(mn)$ equivalent.
\end{lemma}

\begin{proof}
If $v$ is a critical point of the restriction of $P_{n}$ to the $h$ unit sphere, then $r^{-1}v$ is a critical point of $r^{2}P_{n}$ to the $r^{2}h$ unit sphere, so if $P_{n}(v) \neq 0$, then $e = \tfrac{v}{6P_{n}(v)}$ and $c(v) = |(\hess P_{n})(e)|^{2}_{h}$ depend only on the homothety class of $h$.
By Corollary \ref{ndimexistencecorollary}, $e$ has the form $\tfrac{1}{n+1-2k}\sum_{i \in I}v_{i}$ for $v_{i} \in \extreme(P_{n})$ and $1 \leq k \leq n+1$ such that $2k \neq n+1$, and the possible values of $c(v)$ are $|(\hess P_{n})(e)|^{2} = n(n-1)|e|^{2} = \tfrac{k(n+1-k)(n-1)}{(n+1 - 2k)^{2}} = f_{n}(k)$ for the same range of $k$. Because $f_{n}(n+1 - k) = f_{n}(k)$ in what follows it suffices to consider $1 \leq k < (n+1)/2$. Let $u$ and $v$ be critical points of the restrictions of $P_{m}$ and $P_{n}$ to the unit spheres such that $P_{m}(u) \neq 0$, $P_{n}(v) \neq 0$, $c(u) = \tfrac{m}{m-1}$, and $c(v) = \tfrac{n}{n-1}$. Then $u \tensor v$ is a critical point of the restriction of $P_{m}\tensor P_{n}$ to the unit sphere such that $c(u\tensor v) = \tfrac{mn}{(m-1)(n-1)}$. Were the orbits $\orb{P_{m}\tensor P_{n}}$ and $\orb{P_{mn}}$ $CO(mn)$-equivalent, there would be a critical point $w$ of the restriction of $P_{mn}$ to the unit sphere such that $P_{mn}(w) \neq 0$ and $c(w) = \tfrac{mn}{(m-1)(n-1)}$. It will be shown that this is impossible. By the preceding, the possible values of $c(w)$ are $f_{mn}(k) = \tfrac{k(mn+1 - k)(mn-1)}{(mn + 1- 2k)^{2}}$ for $1 \leq k < (mn + 1)/2$. Because in this range $f_{mn}(k)$ is increasing in $k$, to complete the proof it suffices to check that $f_{mn}(1) = \tfrac{mn}{mn-1} < \tfrac{mn}{(m-1)(n-1)}$ and $f_{mn}(2) = \tfrac{2(mn-1)^{2}}{(mn-3)^{2}} > \tfrac{mn}{(m-1)(n-1)}$ for $m, n \geq 2$. The first inequality is immediate. For the second, if $m \geq 3$ and $n \geq 4$, then $\tfrac{mn}{(m-1)(n-1)} = 1 + \tfrac{1}{m-1} + \tfrac{1}{n-1} + \tfrac{1}{(m-1)(n-1)} \leq 2$ while $f_{mn}(2) > 2$; the remaining cases $(m, n) \in \{(2, 2), (2, 3), (2, 4), (3, 3)\}$ can be checked by direct evaluation.
\end{proof}
\end{example}

\section{Classification of solutions in dimension \texorpdfstring{$3$}{3} and \texorpdfstring{$4$}{4}: proofs of Theorems \texorpdfstring{\ref{3dpolytheorem}}{3d} and \texorpdfstring{\ref{4dpolytheorem}}{4d}}\label{lowdpolysection}
In this section the normalized equations \eqref{epreduced} of Lemma \ref{preparationlemma} are used to prove Theorems \ref{3dpolytheorem} and \ref{4dpolytheorem}, that characterize solutions of \eqref{einsteinpolynomials} (in Riemannian signature) in dimensions $3$ and $4$.

\begin{proof}[Proof of Theorem \ref{3dpolytheorem}]
For $c \neq 0$, the harmonic polynomial $P =\kc x_{1}x_{2}x_{3}$ solves 
\eqref{einsteinpolynomials} with $\ka = 2\kc^{2}$. Consequently, any polynomial in the $O(3)$ orbit of $P$ is also harmonic and solves \eqref{einsteinpolynomials} with $\ka = 2\kc^{2}$. This shows \eqref{dth3} implies \eqref{dth2}. 

Let $P$ be a nontrivial solution of \eqref{einsteinpolynomials} with parameter $\ka$. By Lemma \ref{preparationlemma}, it can be supposed that a nonzero positive multiple of $P$ has the form \eqref{pprenormal} with $c = 1$, so that
\begin{align}
P(x_{1}, x_{2}, x_{3}) = -\la(x_{3}^{3} - 3x_{3}x_{1}^{2}) + (1+\la)(x_{3}^{3} -3x_{3}x_{2}^{2}) + Q(x_{1}, x_{2}),
\end{align}
where $Q$ and $\la$ satisfy \eqref{epreduced}. 
By \eqref{epreduced}, $Q$ is harmonic, so, by Example \ref{2dharpolyexample}, $Q$ has the form \eqref{2dharpoly} for some $r > 0$ and solves \eqref{einsteinpolynomials} with parameter $2r^{2}$. This observation coupled with the last equation of \eqref{epreduced} yields
\begin{align}\label{3din0}
2r^{2}(x_{1}^{2} + x_{2}^{2}) + 72(\la^{2}x_{1}^{2} + (1+\la)^{2}x_{2}^{2}) = 36(1 + \la^{2} + (1+\la)^{2})(x_{1}^{2} + x_{2}^{2}).
\end{align}
By \eqref{3din0}, $-72 \la = 2r^{2} = 72(1 + \la)$, so that $\la = -1/2$ and $2r^{2} = 36$. Hence $P$ has the form 
\begin{align}\label{3din1}
\begin{split}
x_{3}^{3} - \tfrac{3}{2}x_{3}(x_{1}^{2} + x_{2}^{2}) + \tfrac{1}{\sqrt{2}}\cos \theta \left(x_{1}^{3} - 3x_{1}x_{2}^{2}\right)+\tfrac{1}{\sqrt{2}}\sin \theta \left(x_{2}^{3} - 3x_{2}x_{1}^{2}\right).
\end{split}
\end{align}
A rotation in the $(x_{1}, x_{2})$ plane preserves the first two terms of \eqref{3din1}, and, as in the discussion following \eqref{2dharpoly}, after such a rotation it may be supposed that $\theta = \pi/2$, so that $P$ has the the form 
\begin{align}
\label{poly2}
\begin{split}
 x_{3}^{3}& - \tfrac{3}{2}x_{3}(x_{1}^{2} + x_{2}^{2}) + \tfrac{1}{\sqrt{2}}\left(x_{2}^{3} - 3x_{2}x_{1}^{2}\right) \\
& = 3\sqrt{3}\left(\sqrt{\tfrac{2}{3}}x_{2} + \tfrac{1}{\sqrt{3}}x_{3}\right)\left( \tfrac{1}{\sqrt{2}}x_{1} -\tfrac{1}{\sqrt{6}}x_{2} + \tfrac{1}{\sqrt{3}}x_{3}\right)\left( - \tfrac{1}{\sqrt{2}}x_{1} -\tfrac{1}{\sqrt{6}}x_{2}+ \tfrac{1}{\sqrt{3}}x_{3}\right).
\end{split}
\end{align}
(\eqref{poly2} solves \eqref{einsteinpolynomials} with $\ka = 54$.) That the polynomial \eqref{poly2} and $3\sqrt{3}x_{1}x_{2}x_{3}$ lie on the same $SO(3)$-orbit is apparent from \eqref{poly2}, which exhibits $P(x)$ in the form $Q(g^{-1}x)$ for $Q(x) = 3\sqrt{3}x_{1}x_{2}x_{3}$ and the element $g$ of $SO(3)$ the rows of which are the coefficients in the factorization in the last line of \eqref{poly2}. This $g$ is the rotation sending the ray on which the coordinates are all equal (on which the restriction to the sphere of $Q$ achieves its maximum) into the direction of $x_{3}$, as in the normalization leading to \eqref{pprenormal}. This shows that \eqref{dth2} implies \eqref{dth3}.
\end{proof}

Theorem \ref{4dpolytheorem} classifies the solutions of \eqref{einsteinpolynomials} on $\rea^{4}$ in Riemannian signature, showing that any solution is orthogonally equivalent to one of the polynomials in Table \ref{4dpolynormalforms} and that for a given value of $\ka$ the two different solutions in Table \ref{4dpolynormalforms} are not orthogonally equivalent.
\begin{table}[h]
\caption{Normal forms for solutions of \eqref{einsteinpolynomials} in dimension $4$}\label{4dpolynormalforms}
\begin{tabular}{|p{.7\linewidth}|c|c|c|c|}
\hline
$P(x_{1}, x_{2}, x_{3}, x_{4})$ & $\la$ & $\ka$ & $\mkc(P)$\\
\hline
&&&\\
$\begin{aligned} \tfrac{c}{6}\left(x_{4}^{3} - 3 x_{4}x_{3}^{2} +3x_{1}^{2}x_{2} - x_{2}^{3} \right)\end{aligned}$ & $0$ & $2c^{2}$ & $2$ \\
\hline
&&&\\
$\begin{aligned} \tfrac{c}{6}\left(x_{4}^{3} - x_{4}(x_{1}^{2} + x_{2}^{2} + x_{3}^{2}) + 2\sqrt{5}x_{1}x_{2}x_{3}\right)\end{aligned}$ & $-\tfrac{1}{3}$ & $\tfrac{4}{3}c^{2}$ & $\tfrac{4}{3}$\\
\hline
\end{tabular}
\end{table}
(In either case, $\ka = 2c^{2}(1 + 2\la + 3\la^{2}) = c^{2}(6(\la + \tfrac{1}{3})^{2} + \tfrac{4}{3})$.)
The solution \eqref{lazero} (the second solution of Table \ref{4dpolynormalforms}) is that given by Example \ref{lanminusoneexample}. The solution \eqref{minusonethird0} (the first solution of Table \ref{4dpolynormalforms}) is obtained via Lemma \ref{recursivepolynomiallemma} and Theorem \ref{3dpolytheorem}.

\begin{proof}[Proof of Theorem \ref{4dpolytheorem}]
The proof has two parts. First it is shown that any solution of \eqref{einsteinpolynomials} on $\rea^{4}$ in Riemannian signature is conformally equivalent to \eqref{laminusonethird} or \eqref{lazero}, that is that there are no other solutions. Second, these two solutions are shown to be inequivalent.

Suppose $P \in \pol^{3}(\rea^{4})$ solves \eqref{einsteinpolynomials}. By Lemma \ref{preparationlemma}, after a conformal linear transformation, $P$ can be supposed to have the form \eqref{pprenormal}, that is the form 
\begin{align}\label{4dpprenormal}
P(x_{1}, x_{2}, x_{3}, x_{4}) = \la_{1}(3x_{4}x_{1}^{2} - x_{4}^{3}) + \la_{2}(3x_{4}x_{2}^{2} - x_{4}^{3}) +  \la_{3}(3x_{4}x_{3}^{2} - x_{4}^{3}) + Q(x_{1}, x_{2}, x_{3}),
\end{align}
where $\la_{1} +  \la_{2} +  \la_{3}  = -1$, $\la_{i} \in [-2, 1/2]$ and $Q(x_{1}, x_{2}, x_{3}) \in \pol^{3}(\rea^{3})$ solves the equations \eqref{epreduced}. By Lemma \ref{polorthonormalbasislemma}, $Q$ can be supposed to have the form
\begin{align}\label{4dq}
Q(x_{1}, x_{2}, x_{3}) & = \al x_{1}x_{2}x_{3} +\tfrac{1}{6} \sum_{1 \leq i \neq j \leq 3}\be_{ij}(3x_{i}^{2}x_{j} - x_{j}^{3}),
\end{align}
for constants $\al, \be_{ij} \in \rea$. Let $\la_{1}, \la_{2}, \la_{3} \in \rea$ sum to $-1$. Moreover, the $\la_{i}$ must satisfy $1 + \la_{1}^{2} + \la_{2}^{2} + \la_{3}^{2} \geq 2\la_{i}^{2}$. 
By Lemma \ref{qpreplemma}, that $Q$ solves \eqref{epreduced} yields the equations
\begin{align}\label{bebeeqns}
\begin{split}
0 & = 2\al(\be_{13} + \be_{23}) + \be_{31}\be_{32} - \be_{12}\be_{31} - \be_{21}\be_{32},\\
0 & = 2\al(\be_{12} + \be_{32}) + \be_{21}\be_{23} - \be_{13}\be_{21} - \be_{31}\be_{23},\\
0 & = 2\al(\be_{21} + \be_{31}) + \be_{12}\be_{13} - \be_{23}\be_{12} - \be_{32}\be_{13},
\end{split}
\end{align}
and
\begin{align}\label{purebeeqns}
\begin{split}
18(1 - \la_{1}^{2} + \la_{2}^{2} + \la_{3}^{2}) & = \al^{2} + \be_{12}^{2} + \be_{13}^{2} + \be_{21}^{2} + \be_{31}^{2} + \be_{21}\be_{31},\\
18(1 + \la_{1}^{2} - \la_{2}^{2} + \la_{3}^{2}) & = \al^{2} + \be_{21}^{2} + \be_{23}^{2} + \be_{12}^{2} + \be_{32}^{2} + \be_{12}\be_{32},\\
18(1 + \la_{1}^{2} + \la_{2}^{2} - \la_{3}^{2}) & = \al^{2} + \be_{31}^{2} + \be_{32}^{2} + \be_{13}^{2} + \be_{23}^{2} + \be_{13}\be_{23},
\end{split}
\end{align}
together equivalent to the requirement that $|\hess Q| ^{2}  + 72\sum_{i= 1}^{n}\la_{i}^{2}x_{i}^{2} = 36(1 + \sum_{i = 1}^{n}\la_{i}^{2})E_{n}(x)$, and the equations
\begin{align}\label{labeeqns}
\begin{split}
0 & = (\la_{2} - \la_{1})\be_{21} + (\la_{3} - \la_{1})\be_{31} = \la_{21}\be_{21} + \la_{31}\be_{31},\\
 0 & = (\la_{1} - \la_{2})\be_{12} + (\la_{3} - \la_{2})\be_{32} = \la_{12}\be_{12} + \la_{32}\be_{32},\\
0 & = (\la_{1} - \la_{3})\be_{13} + (\la_{2} - \la_{3})\be_{23} = \la_{13}\be_{13} + \la_{23}\be_{23},
\end{split}
\end{align}
equivalent to the requirement that $\sum_{i = 1}^{n}\la_{i}\tfrac{\pr^{2}Q}{\pr^{2} x^{i}} = 0$. Here $\la_{ij} = \la_{i} - \la_{j} = -\la_{ji}$. Note that $\la_{12} + \la_{23} + \la_{31} = 0$.

In each of the three triads of equations \eqref{bebeeqns}, \eqref{purebeeqns}, and \eqref{labeeqns}, the second and third equation are obtained from the first by cyclically permuting the indices $123$, so that the entire set of equations \eqref{bebeeqns}-\eqref{labeeqns} is invariant under permutation of the indices $123$. Since permutation of the variables $x_{1}, x_{2}, x_{3}$ is an orthogonal transformation, this means that where convenient such a permutation can be used to normalize the equations.

Multiplying the first equation of \eqref{bebeeqns} by $\la_{23}\la_{31}$ and using \eqref{labeeqns} to simplify the result yields
\begin{align}
\begin{split}
0 & = 2\al\la_{23}(\la_{31}\be_{13} + \la_{31}\be_{23}) + \la_{23}\be_{32}\la_{31}\be_{31} - \la_{23}\la_{31}\be_{12}\be_{31} - \la_{23}\be_{31}\la_{31}\be_{31}\\
& = 2\al\la_{23}(\la_{23} + \la_{31})\be_{23} + \la_{12}\la_{31}\be_{12}\be_{31} - \la_{23}\la_{31}\be_{12}\be_{31} - \la_{12}\la_{31}\be_{12}\be_{21}\\
& = 2\al\la_{21}\la_{23}\be_{23} + (\la_{12}\la_{31} - \la_{23}\la_{31} - \la_{31}^{2})\be_{12}\be_{31} = -2\al\la_{23}\la_{12}\be_{23} + 2\la_{12}\la_{31}\be_{12}\be_{31}.
\end{split}
\end{align}
Calculating in a similar fashion the result of multiplying the second and third equations of \eqref{bebeeqns} by $\la_{31}\la_{12}$ and $\la_{12}\la_{23}$ yields the three equations
\begin{align}
\label{bebeeqns2}
&\la_{12}\la_{23}\al\be_{23} = \la_{31}\la_{12}\be_{31}\be_{12},&&\la_{23}\la_{31}\al\be_{31} = \la_{12}\la_{23}\be_{12}\be_{23},&\la_{31}\la_{12}\al\be_{12} = \la_{23}\la_{31}\be_{23}\be_{31},
\end{align}
which are equivalent to \eqref{bebeeqns2}. Write $\Delta = \la_{12}\la_{23}\la_{31}$. From \eqref{bebeeqns2} there follow
\begin{align}\label{bebeeqns3}
\begin{split}
\Delta\be_{31}\be_{12}\be_{23} &= \la_{12}\la_{23}^{2}\al\be_{23}^{2} = \la_{23}\la_{31}^{2}\al\be_{31}^{2} =\la_{31}\la_{12}^{2}\al\be_{12}^{2},\\
\Delta^{2}\al^{3}\be_{31}\be_{12}\be_{23}& = \Delta^{2}\be_{31}^{2}\be_{12}^{2}\be_{23}^{2}. 
\end{split}
\end{align}
There are three essentially different special cases to consider, depending on whether $\{\la_{1}, \la_{2}, \la_{3}\}$ comprises $1$, $2$, or $3$ distinct values. The analysis will be made assuming that the $\la_{i}$ sum to $-1$, but ignoring the condition $\la_{i} \in [-2, 1/2]$. There will result extra putative solutions that will be eliminated when this condition on the $\la_{i}$ is imposed.

Next it is claimed that \textit{there is no $P \in \pol^{3}(\rea^{4})$ having the form \eqref{epreduced} with $\Delta = (\la_{1} - \la_{2})(\la_{2} - \la_{3})(\la_{3} - \la_{1}) \neq 0$ that solves \eqref{einsteinpolynomials}.}

That $\{\la_{1}, \la_{2}, \la_{3}\}$ be distinct is the same as that $\Delta \neq 0$. Were $\al \neq 0$, the equalities $\la_{12}\la_{23}^{2}\al\be_{23}^{2} = \la_{23}\la_{31}^{2}\al\be_{31}^{2} =\la_{31}\la_{12}^{2}\al\be_{12}^{2}$ given by \eqref{bebeeqns3} would imply that $\la_{12}\al$, $\la_{23}\al$, and $\la_{31}\al$ have the same sign, but these quantities sum to $0$, so this is impossible. Consequently, it must be that $\al = 0$. Then, since the $\la_{i}$ are distinct, the equations \eqref{bebeeqns2} imply that the products $\be_{12}\be_{23}=0$, $\be_{23}\be_{31} =0$, and $\be_{31}\be_{12} =0$, which is possible if and only if at least two of $\be_{12}$, $\be_{23}$, and $\be_{31}$ vanish. It cannot be that all three vanish, for, summing the three equations \eqref{purebeeqns}, this would imply $18(3 + \la_{1}^{2} + \la_{2}^{2} + \la_{3}^{2}) = 0$, which is absurd. Since the equations being considered are unchanged under cyclic permutations of the indices, it can without loss of generality be supposed that $\be_{23} = 0 = \be_{31}$. By \eqref{labeeqns} this implies $\be_{13} = 0$ and $\be_{21} = 0$, so the only possibly nonzero coefficients are $\be_{12}$ and $\be_{32}$. In this case the equations \eqref{purebeeqns} become
\begin{align}\label{purebeeqns2}
\begin{split}
18(1 - \la_{1}^{2} + \la_{2}^{2} + \la_{3}^{2}) & =  \be_{12}^{2},\\
18(1 + \la_{1}^{2} - \la_{2}^{2} + \la_{3}^{2}) & = \be_{12}^{2} + \be_{32}^{2} + \be_{12}\be_{32},\\
18(1 + \la_{1}^{2} + \la_{2}^{2} - \la_{3}^{2}) & = \be_{32}^{2}.
\end{split}
\end{align}
Multiplying the equations of \eqref{purebeeqns2} by $\la_{23}^{2}$ and using \eqref{labeeqns} to simplify the result yields
\begin{align}\label{purebeeqns3}
\begin{split}
18(1 - \la_{1}^{2} + \la_{2}^{2} + \la_{3}^{2})\la_{23}^{2} & =  \la_{23}^{2}\be_{12}^{2},\\
18(1 + \la_{1}^{2} - \la_{2}^{2} + \la_{3}^{2})\la_{23}^{2} & = (\la_{23}^{2} + \la_{12}^{2} + \la_{12}\la_{23})\be_{12}^{2}\\&= (\la_{23}^{2} + \la_{12}\la_{13})\be_{12}^{2} = (\la_{12}^{2} + \la_{13}\la_{23})\be_{12}^{2},\\
18(1 + \la_{1}^{2} + \la_{2}^{2} - \la_{3}^{2})\la_{23}^{2} & = \la_{12}^{2}\be_{12}^{2}.
\end{split}
\end{align}
Subtracting the first equation of \eqref{purebeeqns3} from the second equation, dividing the result by $\la_{12}$, and simplifying using $\la_{13} = 1 + 2\la_{1} + \la_{2}$ and $\la_{23} = 1 + \la_{1} + 2\la_{2}$ yields
\begin{align}\label{pureba}
36(\la_{1} + \la_{2})(1 + \la_{1} + 2\la_{2})^{2} = (1 + 2\la_{1} + \la_{2})\be_{12}^{2}.
\end{align}
Similarly, subtracting the third equation of \eqref{purebeeqns3} from the second equation and dividing the result by $\la_{23}$ yields
\begin{align}\label{purebb}
36(1 + \la_{1})(1 + \la_{1} + 2\la_{2})^{2} = (1 + 2\la_{1} + \la_{2})\be_{12}^{2}.
\end{align}
Since $\be_{12} \neq 0$, $1 + 2\la_{1} + \la_{2} = \la_{13} \neq 0$, and $1 + \la_{1} + 2\la_{2} = \la_{23} \neq 0$, the vanishing of the difference of the equations \eqref{pureba} and \eqref{purebb} implies $\la_{2} = 1$. Since $1 + \la_{1} + 2\la_{2} = \la_{23} \neq 0$, the sum of the equations \eqref{pureba} and \eqref{purebb} yields
\begin{align}
18(3 + \la_{1})^{2} = 18(1 + \la_{1} + 2\la_{2})^{2} = \be_{12}^{2}.
\end{align}
Taking $\la_{2} = 1$ in the first equation of \eqref{purebeeqns2} yields $36(3 + 2\la_{1}) = \be_{12}^{2} = 18(3 + \la_{1})^{2}$, but no real $\la_{1}$ solves this equation. 

The remaining case to consider is that at least two of the $\la_{i}$ coincide. Since a permutation of $\{123\}$ induces an orthogonal transformation of $\rea^{4}$ fixing the $x_{4}$ direction, if two of the $\la_{i}$ coincide, it can be assumed without loss of generality that $\la_{1} = \la_{2} = \la$ and $\la_{3} = -1 - 2\la$.

It is claimed that if $P \in \pol^{3}(\rea^{4})$ solves \eqref{einsteinpolynomials} with constant $\ka$ and has the form \eqref{epreduced} with $\la_{1} = \la_{2} = \la$ (so $\la_{3} = -1 -2\la$), then $\la \in \{-1/3, 0\}$ and $P$ is $O(4)$-equivalent to one of the polynomials in Table \ref{4dpolynormalforms}.

Suppose $P$ has the form \eqref{4dpprenormal} with $Q$ as in \eqref{4dq}.
The equations \eqref{labeeqns} imply $\be_{31} = 0 = \be_{32}$ and $\be_{23} = -\be_{13}$. The first of the equations \eqref{bebeeqns} is vacuous, and the other two equations of \eqref{bebeeqns} simplify to \begin{align}\label{albebebe}
&\al\be_{12} = -\be_{21}\be_{23}, &&\al\be_{21}  = \be_{12}\be_{23}. 
\end{align}
The equations \eqref{purebeeqns} yield $36(1 + 2\la + 2\la^{2}) = \al^{2} + \be_{12}^{2} + \be_{21}^{2} + \be_{23}^{2}$ and $-36(\la^{2} + 2\la) = \al^{2} + \be_{23}^{2}$. Combining these yields the equivalent pair of equations
\begin{align}\label{albebebe2}
\begin{split}
\be_{12}^{2}& + \be_{21}^{2} = 36(1 + 4\la + 3\la^{2}) = 108 (\la + 1)(\la + \tfrac{1}{3}),\\
\al^{2} & + \be_{23}^{2} = -36\la(\la+2).
\end{split}
\end{align}
The relations \eqref{albebebe2} force $\la \in [-2, -1]\cup [-1/3, 0]$. From \eqref{albebebe2} and the restrictions on the coefficients so far determined, it follows that $P$ is conformally equivalent to a polynomial of the form \eqref{lanotminusone} for some $\theta \in [0, 2\pi)$, where, however, it remains to determine with of the polynomials \eqref{lanotminusone} solves \eqref{einsteinpolynomials}.
Lemma \ref{preparationlemma} also requires $\{\la, -(1 + 2\la)\} \subset [-2, 1/2]$. That $-2 \leq -1 - 2\la \leq 1/2$ implies $-3/4\leq \la \leq 1/2$, so Lemma \ref{preparationlemma} implies that each solution for $\la \in [-2, -1]$ is equivalent to a solution with $\la \in [-1/3, 0]$. (The point is that $e_{4}$ is a maximum of $P$ restricted to the unit sphere when $\la \in [-1/3, 0]$, but is not a maximum of $P$ restricted to the unit sphere when $\la \in [-2, -1]$.) Hence $\la \in [-1/3, 0]$.

Together the equations \eqref{albebebe} yield $\al\be_{12}^{2} = -\be_{12}\be_{21}\be_{23} = - \al\be_{21}^{2}$, so $\al(\be_{12}^{2} + \be_{21}^{2}) = 0$. 

If $\la \in (-1/3, 0]$, then the first equation of \eqref{albebebe2} implies $\be_{12}^{2} +\be_{21}^{2} \neq 0$, so  $\al = 0$. In \eqref{albebebe} this implies $\be_{23} = 0$. In the second equation of \eqref{albebebe2} this forces $\la \in \{-2, 0\}$, so $\la = 0$. Consequently, either $\la = 0$, in which case $\al = 0$, $\be_{23} = 0$, and $\be_{12}^{2} + \be_{21}^{2} = 36$, or $\la = -1/3$, in which case $\be_{12} = 0 = \be_{21}$, and $\al^{2} + \be_{23}^{2} = 20$. 

If $\la_{1} = \la_{2} = \la = -1/3$, then $\la_{3} = -1/3$, so that $\la_{i}$ are all equal $-1/3$. In this case, by \eqref{epreduced} and Lemma \ref{epreductionlemma}, $P$ has the form $\tfrac{c}{6}(x_{4}^{3} - x_{4}(x_{1}^{2} + x_{2}^{2} + x_{3}^{2}) + Q(x_{1}, x_{2}, x_{3}))$ where $\ka = \tfrac{4}{3}c^{2}$ and $Q$ solves \eqref{einsteinpolynomials} with constant $40$. By Theorem \ref{3dpolytheorem}, $Q$ is equivalent to $2\sqrt{5}x_{1}x_{2}x_{3}$ via an orthogonal transformation of the subspace with coordinates $x_{1}$, $x_{2}$, and $x_{3}$. Since such a transformation preserves the form of $P$, $P$ is $O(4)$-equivalent to the polynomial \eqref{minusonethird0}.

In the case $\la = 0$, there is $\theta \in [0, 2\pi)$ such that $\be_{12} = 6\cos \theta$ and $\be_{21} = 6\sin\theta$, so $P$ has the form
\begin{align}\label{lazerotheta}
\begin{split}
& \tfrac{c}{6}\left(x_{4}^{3} -3x_{3}^{2}x_{4} + \cos\theta(3x_{1}^{2}x_{2} - x_{2}^{3}) + \sin \theta (3x_{2}^{2}x_{1} - x_{1}^{3})\right),
\end{split}
\end{align}
As in Example \eqref{2dharpolyexample}, by an element of $O(4)$ that acts as a rotation in the $x_{1}x_{2}$ plane and trivially on the $x_{3}x_{4}$ plane this is equivalent to the form \eqref{lazero}.

In either case $\ka = 2c^{2}(1 + 2\la + 3\la^{2}) = c^{2}(6(\la + \tfrac{1}{3})^{2} + \tfrac{4}{3})$. By the construction determined $P$, $P$ assumes its maximum on the $h$-unit sphere at $e_{4}$ where there holds $P(e_{4}) = c/6$, so, by Lemma \ref{mkclemma},
\begin{align}
\mkc(P) = \tfrac{\ka}{36P(e_{4})^{2}} = \tfrac{\ka}{c^{2}} = 6(\la + \tfrac{1}{3})^{2} + \tfrac{4}{3}.
\end{align}
Alternatively, $\la$ is determined by $\mkc(P) =\tfrac{\ka}{c^{2}}$ by $\la = \sqrt{\tfrac{1}{6}(\mkc(P) - \tfrac{4}{3})} - \tfrac{1}{3}$.

Since the values of $\mkc(P)$ are different for the polynomials of \eqref{minusonethird0} and \eqref{lazero}, by Lemma \ref{mkclemma} these polynomials are not equivalent modulo $CO(4)$. This completes the proof of Theorem \ref{4dpolytheorem}. 
\end{proof}

\begin{remark}\label{4dpolydiscriminationremark}
This remark addresses a subtle potential confusion regarding the parameter $\la$ in the proof of Theorem \ref{4dpolytheorem}.

Consider the two parameter family of polynomials
\begin{align}
\label{lanotminusone}
\begin{split}
 & \tfrac{c}{6}\left(x_{4}^{3} + 3x_{4}(\la(x_{1}^{2} + x_{2}^{2}) - (1+2\la)x_{3}^{2}) + 3\sqrt{-2\la - \la^{2}}(x_{1}^{2} - x_{2}^{2})x_{3}\right.\\
&\quad \left.+ \sqrt{1 + 4\la + 3\la^{2}}\left(\cos\theta(3x_{1}^{2}x_{2} - x_{2}^{3}) + \sin \theta (3x_{2}^{2}x_{1} - x_{1}^{3})\right)\right), 
\end{split}
\end{align}
where $\la \in [-2, -1] \cup [-1/3, 0]$, $\theta \in [0, 2\pi)$, and $\ka = 2c^{2}(1 + 2\la + 3\la^{2}) = c^{2}(6(\la + \tfrac{1}{3})^{2} + \tfrac{4}{3})$. 

When $\la = 0$ the polynomial \eqref{lanotminusone} becomes \eqref{lazerotheta} which is $O(4)$-equivalent to \eqref{lazero} as is explained in the proof of Theorem \ref{4dpolytheorem}. 

When $\la = -1/3$ the polynomial \eqref{lanotminusone} becomes
\begin{align}\label{laminusonethird}
P = \tfrac{c}{6}\left(x_{4}^{3} - x_{4}(x_{1}^{2} + x_{2}^{2} + x_{3}^{2}) + \sqrt{5}(x_{1}^{2} - x_{2}^{2})x_{3}\right). 
\end{align}
which is $O(4)$-equivalent to \eqref{minusonethird0} by a rotation by $\pi/2$ in the $x_{1}x_{2}$ plane.

(If $\la \in (-2, -1) \cup (-1/3, 0)$, the polynomials \eqref{lanotminusone} do not solve \eqref{einsteinpolynomials} because the coefficients of $x_{1}x_{3}$ and $x_{2}x_{3}$ in $|\hess P|^{2} $ do not vanish.)

It follows from the proof of Theorem \ref{4dpolytheorem} that the polynomials \eqref{lanotminusone} for $\la \in [-2, -1]$ are equivalent to polynomials of the form \eqref{lanotminusone} for some other $\la \in [-1/3, 0]$. 

The argument proving Theorem \ref{4dpolytheorem} shows that the polynomials
\begin{align}
\label{laminustwo}
\begin{split}
&\tfrac{c}{6}\left(x_{4}^{3} + 3 x_{4}(-2x_{1}^{2} - 2x_{2}^{2} + 3x_{3}^{2}) + \sqrt{5}(3x_{1}^{2}x_{2} - x_{2}^{3}) \right)\quad  \text{if}\,\,\la = -2,
\end{split}\\
\label{laminusone}
\begin{split}
& \tfrac{c}{6}\left(x_{4}^{3} - 3 x_{4}(x_{1}^{2} + x_{2}^{2} - x_{3}^{2}) +3 (x_{1}^{2} - x_{2}^{2}) x_{3} \right) \quad  \text{if}\,\,\la = -1.
\end{split}
\end{align}
corresponding to $\la = -2$ and $\la = -1$ solve \eqref{einsteinpolynomials}. Lemma \ref{preparationlemma} implies that these solutions are orthogonally equivalent to the solutions obtained for $\la = -1/3$ and $\la = 0$, respectively. The subtle point is that $e_{4}$ is a maximum of $P$ restricted to the unit sphere when $\la \in [-1/3, 0]$, but is not a maximum of $P$ restricted to the unit sphere when $\la \in [-2, -1]$, and the normalizations made in the proof of Lemma \ref{preparationlemma} require that $e_{4}$ be a maximum of the restriction of $P$ to the unit sphere. 
\end{remark}

There remains to prove Corollary \ref{4dindecomposablecorollary}. 
\begin{proof}[Proof of Corollary \ref{4dindecomposablecorollary}]
By Lemma \ref{decomposablecriticallemma} to show that $P$ is orthogonally indecomposable it suffices to show that no two of its critical lines are orthogonal.

The orthogonal indecomposability of \eqref{minusonethird0} is shown by calculating explicitly its critical lines. There are $15$ of them. 
Let 
\begin{align}
P = x_{4}^{3} - x_{4}(x_{1}^{2} + x_{2}^{2} + x_{3}^{2}) + 2\sqrt{5}x_{1}x_{2}x_{3}.
\end{align}
A nonzero $x$ generates a critical line of $P$ if there is $\la \in \rea$ such that
\begin{align}
\label{mta}\la x_{1} &= -2x_{1}x_{4} + 2\sqrt{5}x_{2}x_{3},\\
\label{mtb}\la x_{2} &= -2x_{2}x_{4} + 2\sqrt{5}x_{3}x_{1},\\
\label{mtc}\la x_{3} &= -2x_{3}x_{4} + 2\sqrt{5}x_{1}x_{2},\\
\label{mtd}\la x_{4} & = 3x_{4}^{2} - x_{1}^{2} - x_{2}^{2} - x_{3}^{2}.
\end{align}
It must be that $x_{4} \neq 0$, for otherwise, by \eqref{mtd}, $x$ would be $0$. For definiteness suppose that $|x|^{2} = r^{2}$.

Multiplying \eqref{mta}, \eqref{mtb}, and \eqref{mtc} by $x_{1}$, $x_{2}$, and $x_{3}$, respectively, yields
\begin{align}\label{mte}
2\sqrt{5}x_{1}x_{2}x_{3} = x_{1}^{2}(\la + 2x_{4})= x_{2}^{2}(\la + 2x_{4})= x_{3}^{2}(\la + 2x_{4}).
\end{align}
By \eqref{mte}, if $x_{4} \neq -\la/2$, then either all of $x_{1}$, $x_{2}$, and $x_{3}$ equal $0$ or none of them equals $0$, while if $x_{4} = -\la/2$, then $x_{1}x_{2}x_{3} = 0$.

Suppose $x_{1} = x_{2} = x_{3} = 0$. Then $x_{4} = \pm r$, so $(0, 0, 0, 1)$ generates a critical line of weight $1$.

Suppose $x_{4} \neq -\la/2$ and none of $x_{1}$, $x_{2}$, and $x_{3}$ is $0$. By \eqref{mte} there holds $x_{1}^{2} = x_{2}^{2} = x_{3}^{2}$. In \eqref{mtd} this yields $0 = 3x_{4}^{2} - \la x_{4} - 3x_{1}^{2}$, while in \eqref{mta}, because $x_{1} \neq 0$, it yields $\pm 2\sqrt{5}x_{1} = \la + 2x_{4}$. Substituting the latter expression into the $20$ times the former and simplifying yields
\begin{align}
\begin{split}
0 &= 60x_{4}^{2} - 20\la x_{4} - 60x_{1}^{2} = 60x_{4}^{2} - 20\la x_{4} - 3(\la + 2x_{4})^{2}\\
& = 48x_{4}^{2} - 32\la x_{4} - 3\la^{2} = (12x_{4} + \la)(4x_{4} - 3\la),
\end{split}
\end{align}
so that $x_{4} = -\la/12$ and $\pm x_{1} = \pm x_{2} = \pm x_{3} = \sqrt{5}\la/12$ or $x_{4} = -3\la/4$ and $\pm x_{1} = \pm x_{2} = \pm x_{3} = \sqrt{5}\la/4$, in both cases subject to the condition, that follows from \eqref{mte}, that the product $x_{2}x_{3}$ have the same sign as $x_{1}(\la + 2x_{4})$. In the first case $\la^{2} = 8r^{2}$ and in the second case $\la^{2} = 2r^{2}/3$, so there result the $4$ critical lines of weight $1$ with generators
\begin{align}
& ( \sqrt{5}, - \sqrt{5}, - \sqrt{5}, -1),&
& (- \sqrt{5},  \sqrt{5}, -\sqrt{5}, -1),&
&(- \sqrt{5}, - \sqrt{5},  \sqrt{5}, -1),&
& ( \sqrt{5},  \sqrt{5},  \sqrt{5}, -1),&
\end{align}
and the $4$ critical lines of weight $2/27$ with generators
\begin{align}
& ( \sqrt{5}, - \sqrt{5}, - \sqrt{5}, 3),&
& (- \sqrt{5},  \sqrt{5}, -\sqrt{5}, 3),&
&(- \sqrt{5}, - \sqrt{5},  \sqrt{5}, 3),&
& ( \sqrt{5},  \sqrt{5},  \sqrt{5}, 3),&
\end{align}

Suppose $x_{4} = -\la/2$ and $x_{1}x_{2}x_{3} =0$. Suppose $x_{3}  = 0$. Then, by \eqref{mtc}, $x_{1}x_{2} = 0$. Suppose $x_{2} = 0$ and $x_{1} \neq 0$. By \eqref{mtd}, $0 = 3x_{4}^{2} - \la x_{4} - x_{1}^{2} = 5x_{4}^{2} - x_{1}^{2}$. Substituting this into $r^{2} = x_{1}^{2} + x_{4}^{2}$ yields $2$ critical lines of weight $2/27$ generated by
\begin{align}
-( \sqrt{5}, 0, 0, 1).
\end{align}
Because the equations \eqref{mta}-\eqref{mtd} are unchanged by permutations of the variables $x_{1}$, $x_{2}$, and $x_{3}$, the same reasoning applies if $x_{1} = 0$ or $x_{2} = 0$ is supposed initially, yielding $4$ critical lines of weight $2/27$ generated by 
\begin{align}
&- (0, 0,  \sqrt{5},  1),& &-(0,  \sqrt{5},  0,  1).
\end{align}
(All the listed generators have been chosen so that $P$ is positive on the generator.) Suitably rescaled generators of the weight $1$ critical lines generate an $A_{4}$ root system. The reflections through the hyperplanes orthogonal to their pairwise differences permute the critical lines and act as automorphisms of $P$, showing that this group equals the symmetric group $S_{5}$ acting as permutations of the weight $1$ critical lines of $P$.
\end{proof}

\section{Cubic isoparametric polynomials}\label{isoparametricsection}
A hypersurface in the round sphere $S^{n-1}$ is \textit{isoparametric} if its principal curvatures are constant. See \cite{Cecil} for background and references. In \cite{Cartan-cubic}, Cartan classified the isoparametric hypersurfaces in a round sphere having at most three distinct principal curvatures. In this case, the principal curvatures must all have the same multiplicity $m$, which must be one of $1$, $2$, $4$, or $8$, and the hypersurface must be a tube of constant radius over the image in $S^{3m+1}$ of the Veronese embedding of the projective plane over one of the real definite signature unital composition algebras. This yields four one-parameter families (the parameter is the radius), each of which can be realized as the level sets $P(x) = \cos 3t$ of a homogeneous cubic polynomial $P \in \pol^{3}(\rea^{n})$ solving 
\begin{align}\label{cartaniso}
&\lap P = 0, && |dP| ^{2} = 9|x|^{4}.
\end{align}
These homogeneous cubic polynomials are determined by the multiplication and invariant bilinear form on the deunitalizations of the simple real Euclidean Jordan algebras of $3\times 3$ symmetric matrices over the real definite signature unital composition algebras. (The deunitalization of a unital metrized algebra means the subspace orthogonal to the unit equipped with the multiplication obtained by projecting the given product orthogonally onto this subspace; in this concrete context it means the trace-free Hermitian matrices equipped with the product given by the trace-free part of the usual Jordan product of Hermitian matrices). The corresponding polynomials are defined in dimensions $5$, $8$, $14$, and $26$.

\begin{lemma}[\cite{Fox-ahs}, \cite{Nadirashvili-Tkachev-Vladuts}]\label{isoparametricpolynomiallemma}
If $P \in \pol^{3}(\rea^{n})$ solves \eqref{cartaniso} then it solves \eqref{einsteinpolynomials} with $\ka = 18(n+2)$.
\end{lemma}
\begin{proof}
Straightforward calculation using \eqref{cartaniso} shows
\begin{align}
2|\hess P|^{2}  = \lap |dP| ^{2} = 9\lap |x|^{4}  = 18(|x|^{2} \lap |x|^{2}  + |d|x|^{2} | ^{2}) = 36(n+2)|x|^{2} ,
\end{align}
so that $P$ solves \eqref{einsteinpolynomials} with $\ka = 18(n+2)$.
\end{proof}

\begin{remark}
In \cite{Nadirashvili-Tkachev-Vladuts, Tkachev-jordan} it is shown that the cubic isoparametric polynomials satisfy the \emph{radial eigencubic} equations
\begin{align}\label{radialeigencubic}
|Du|^{3}\div \tfrac{Du}{|Du|} = |Du|^{2}\lap u - \tfrac{1}{2}\lb Du, D|Du|^{2}\ra = \la |x|^{2}u
\end{align}
for some $\la \in \rea$, so that by \cite{Hsiang} their zero level sets define cubic minimal hypercones. That the cubic isoparametric polynomials solve \eqref{einsteinpolynomials} was observed in section $7$ of \cite{Fox-ahs} and \cite{Fox-crm}, and is also shown in Equation $(6.4.15)$ (see also Proposition $6.11.1$ and Corollary $6.11.2$) of \cite{Nadirashvili-Tkachev-Vladuts}. 
\end{remark}

Let $\fie$ be one of $\rea$, $\com$, $\quat$ (quaternions), or $\cayley$ (octonions), and let $m = \dim_{\rea}\fie$. Let $z_{1}, z_{2}, z_{3} \in \fie$ and let $\bar{z}_{i}$ denote the canonical conjugation on $\fie$ fixing the real subfield. Let $u, v \in \rea$, so $(u, v, z_{1}, z_{2}, z_{3}) \in \rea^{2} \oplus \fie \oplus \fie \oplus \fie = \rea^{3m+2}$. By \cite{Cartan-cubic, Cartan-isoparametricsurvey} any solution of \eqref{cartaniso} is equivalent modulo a rotation to one of the four polynomials
\begin{align}\label{cartanformula}
\begin{split}
Q&(u, v, z_{1}, z_{2}, z_{3})\\ &= u^{3} + \tfrac{3}{2}u\left(z_{1}\bar{z}_{1} + z_{2}\bar{z}_{2} - 2z_{3}\bar{z}_{3} - 2v^{2}\right) \\
&\qquad + \tfrac{3\sqrt{3}}{2}v\left(z_{1}\bar{z}_{1} - z_{2}\bar{z}_{2}\right) + \tfrac{3\sqrt{3}}{2}\left((z_{1}z_{2})z_{3} + \bar{z}_{3}(\bar{z}_{2}\bar{z}_{1}) \right)\\
& = u^{3} - 3uv^{2} - \tfrac{1}{2}( u^{3} - 3uz_{1}\bar{z}_{1}) - \tfrac{1}{2}( u^{3} - 3uz_{2}\bar{z}_{2}) + ( u^{3} - 3uz_{3}\bar{z}_{3}) \\
&\qquad -\tfrac{\sqrt{3}}{2}(v^{3} - 3vz_{1}\bar{z}_{1})  + \tfrac{\sqrt{3}}{2}(v^{3} - 3vz_{2}\bar{z}_{2})  + \tfrac{3\sqrt{3}}{2}\left((z_{1}z_{2})z_{3} + \bar{z}_{3}(\bar{z}_{2}\bar{z}_{1}) \right),
\end{split}
\end{align}
which is, up to changes of notation, equation $(17)$ of \cite{Cartan-cubic}. The parentheses in some terms of \eqref{cartanformula} are necessary when $m = 8$, because $\cayley$ is not associative. 

When $\fie$ is $\rea$ or $\com$, \eqref{cartanformula} can be written in terms of the determinant of a trace-free symmetric or Hermitian matrix:
\begin{align}\label{isodet}
Q(x, y, z_{1}, z_{2}, z_{3}) &=\tfrac{3\sqrt{3}}{2}\det \begin{pmatrix} -\tfrac{1}{\sqrt{3}}x + y & z_{3} & \bar{z}_{1} \\ \bar{z}_{3} & -\tfrac{1}{\sqrt{3}}x - y & z_{2} \\ z_{1} & \bar{z}_{2} & \tfrac{2}{\sqrt{3}}x\end{pmatrix}.
\end{align}
Over $\quat$ or $\cayley$ sense has to made of the determinant. When $\fie$ is $\quat$ or $\cayley$, the determinant in \eqref{isodet} can be given sense in the following way. Over a field, a trace-free $3 \times 3$ matrix $X$ satisfies $3\det X = \tr X^{3}$. When $\fie = \quat$, the matrix multiplication is associative, but over $\fie = \cayley$, this fails, so there has to be written $\tfrac{1}{2}(X(X^{2}) + (X^{2})X)$ instead of $X$. Then $Q(x, y, z_{1}, z_{2}, z_{3})$ equals $\tfrac{1}{6}\tr(X(X^{2}) + (X^{2})X)$ where
\begin{align}
X = 3^{1/2}2^{-1/3} \begin{pmatrix} -\tfrac{1}{\sqrt{3}}x + y & z_{3} & \bar{z}_{1} \\ \bar{z}_{3} & -\tfrac{1}{\sqrt{3}}x - y & z_{2} \\ z_{1} & \bar{z}_{2} & \tfrac{2}{\sqrt{3}}x\end{pmatrix}.
\end{align}
With this choice of coordinates $\tfrac{1}{3}\tr X^{2} = x^{2} + y^{2} + |z_{1}|^{2} + |z_{2}|^{2} + |z_{3}|^{2}$ is the standard Euclidean metric.

These solutions of \eqref{einsteinpolynomials} are special because their automorphism groups are large.
The groups $O(n)$, $U(n)$, and $Sp(n)$ act by conjugation on the $3\times 3$ Hermitian matrices over $\rea$, $\com$, and $\quat$ and these actions are automorphisms of the Jordan multiplication, so preserve the determinant and are automorphisms of the associated cubic polynomials. By \cite{Freudenthal}, the compact Lie group of type $F_{4}$ acts as automorphisms on the Jordan algebra of $3 \times 3$ Hermitian matrices over $\cayley$, so acts as automorphisms of the $26$-variable cubic isoparametric polynomial.

\begin{remark}
That the automorphism groups of the cubic isoparametric polynomials are Lie groups of positive dimension has the consequence that $\critline(P)$ is not a discrete subset of $\proj(\alg)$.

The following problem appears interesting: \emph{characterize the orthogonally indecomposable solutions of \eqref{einsteinpolynomials} having automorphism groups of positive dimension}. The cubic polynomials of the curvature algebras studied in \cite{Fox-curvtensor} provide examples in addition to the cubic isoparametric polynomials. A slightly more refined question is to characterize the orthogonally indecomposable solutions of \eqref{einsteinpolynomials} having a given Lie group as automorphisms. 
\end{remark}

\section{Cubic polynomial associated with a regular partial Steiner triple system}\label{steinersection}
A regular partial Steiner triple system on $\setn = \{1, \dots, n\}$ is a collection $\B$ of three element subsets of $\setn$, called \emph{blocks}, such that each pair of distinct $i$ and $j$ in $\setn$ is contained in at most one block of $\B$, and each $i \in \setn$ is contained in exactly $r$ blocks of $\B$. The number $r$ is called the \emph{replication number}. A partial Steiner triple system is a Steiner triple system if each pair of distinct $i$ and $j$ in $\setn$ is contained in exactly one block of $\B$. The sets of partial Steiner triple systems and Steiner triple systems on $\setn$ are denoted $PSTS(n)$ and $STS(n)$.
\begin{lemma}\label{pstslemma}
Let $\B \in PSTS(n)$ be a regular partial Steiner triple system with replication number $r$ and let $\ep \in \{\pm 1\}^{\B}$.
For $I = abc \in \B$ let $x_{I} = x_{a}x_{b}x_{c}$. The polynomial 
\begin{align}\label{partialsteinerpolynomial}
P_{\B, \ep}(x) = \sum_{I \in \B}\ep_{I}x_{I}
\end{align}
solves \eqref{einsteinpolynomials} with $\ka = 2r$ for the standard Euclidean metric with quadratic form $\sum_{i \in \setn}x_{i}^{2}$.
\end{lemma}
\begin{proof}
Let $e_{i} = \tfrac{\pr}{\pr x_{i}}$ be the orthonormal basis of $\rea^{n}$ with respect to which $x_{i}$ are coordinates.
If $i \neq j$ the value of $(\hess P_{\B, \ep}(x))(e_{i}, e_{j})$ is $\ep_{I}x_{k}$ for the unique $I$ such that $I = \{i, j, k\}$, and $0$ otherwise. The number of occurrences of $x_{k}$ is twice the number $r$ of blocks in which $k$ occurs, so $P_{\B, \ep}$ solves \eqref{einsteinpolynomials} with $\ka = 2r$.
\end{proof}
When $\ep_{I} = 1$ for all $I \in \B$, the subscript indicating dependence on $\ep$ is omitted. The polynomial $P_{\B}$ is the \textit{cubic polynomial associated with the regular partial Steiner system $\B$}.

In general it is not obvious whether $P_{\B, \ep}$ and $P_{\B, \bar{\ep}}$ for $\ep\neq \bar{\ep} \in \{\pm 1\}^{\B}$ are orthogonally equivalent. Example \ref{dualaffineplaneexample} shows that they need not be equivalent.

\begin{example}\label{fanoplaneexample}
A projective plane with blocks of size $3$ is unique up to isomorphism and is called the Fano projective plane. It can be represented as 
\begin{align}\label{fano}
\B = \{123, 145, 167, 246, 257, 347, 356\} \in STS(7),
\end{align}
which has replication number $r = 3$. The blocks consist of the points on a line in the Fano plane with the points labeled as in Figure \ref{fanoplane}.
\begin{figure}
\begin{tikzpicture}[baseline=(current  bounding  box.center), scale=.5,auto=left]
  \node (m1) at (0,4)  {$1$};
  \node (m2) at (-1.71, 1)  {$2$};
  \node (m3) at (-3.42,-2)  {$3$};
  \node (m4) at (1.71, 1)  {$4$};
 \node (m5) at (3.42, -2)  {$5$};
 \node (m6) at (0, -2)  {$6$};
 \node (m7) at (0, 0)  {$7$};
\draw (m1) -- (m2);
\draw (m2) -- (m3);
\draw (m1) -- (m4);
\draw (m4) -- (m5) ;
\draw (m1) -- (m7);
\draw (m7) -- (m6) ;
\draw (m3) -- (m6);
\draw (m6) -- (m5) ;
\draw (m2) -- (m7);
\draw (m7) -- (m5) ;
\draw (m4) -- (m7);
\draw (m7) -- (m3) ;
\draw (m4) arc (30:150:2) (m2);
\draw (m2) arc (150:270:2) (m6);
\draw (m6) arc (-90:30:2) (m4);
\end{tikzpicture}  
\caption{Fano projective plane}\label{fanoplane}
\end{figure}
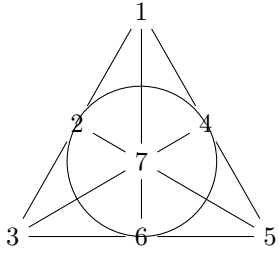
The cubic polynomial corresponding to \eqref{fano} is
\begin{align}\label{fanopoly}
\begin{split}
&x_{1}x_{2}x_{3} + x_{1}x_{4}x_{5} + x_{1}x_{6}x_{7} + x_{2}x_{4}x_{6} + x_{2}x_{5}x_{7} + x_{3}x_{5}x_{6} + x_{3}x_{4}x_{7}.
\end{split}
\end{align}

\begin{lemma}\label{fanoplanelemma}
The harmonic cubic polynomial $P_{\B}$ of \eqref{fanopoly} associated with the Steiner triple system $\B$ determined by the Fano projective plane \eqref{fano} is conformally equivalent to the simplicial polynomial $P_{7}$ defined in \eqref{simplicialpoly}.
\end{lemma}
Lemma \ref{fanoplanelemma} is a special case of the more general Lemma \ref{ffielemma} below.
The Fano projective plane is a special case of the following more general construction. Let $\proj^{k}(\ffie_{2})$ be the $k$-dimensional projective space over the field with two elements, $\ffie_{2}$. This space has $n = 2^{k+1}- 1$ elements. It can be viewed as $\ffie_{2}^{k+1} \setminus \{0\}$, so an element $x \in \proj^{k}(\fie)$ is a $(k+1)$-tuple $x = (x_{1}, \dots, x_{k+1})$ with $x_{i} \in \ffie_{2}$ and not all $x_{i}$ equal to $0$. A line in $\proj^{k}(\ffie_{2})$ is the image of a two-dimensional subspace of $\ffie_{2}^{k+1}$. Any two nonzero elements $x, y \in \ffie_{2}^{k+1}$ generate a two-dimensional subspace of $\ffie_{2}^{k+1}$ the nonzero elements of which are $\{x, y, x+y\}$, and every two-dimensional subspace of $\ffie_{2}^{k+1}$ arises in this way. Thus the lines in $\proj^{k}(\ffie_{2})$ are identified with sets of the form $\{x, y, x+y\}$, where $x, y \in \proj^{k}(\ffie_{2})$ are distinct. Standard counting arguments show that there are $\tfrac{1}{3}\binom{n}{2} = \tfrac{(2^{k+1} - 1)(2^{k} - 1)}{3}$ lines in $\proj^{k}(\ffie_{2})$. The set $\B$ of lines in $\proj^{k}(\ffie_{2})$ is a Steiner triple system. Each line contains three points and each pair of points lies on exactly one line. Each point lies on $r = (n-1)/2 = 2^{k}-1$ lines. The set $\proj^{k}(\ffie_{2})$ is in bijection with $\setn$; to $x \in \proj^{k}(\ffie_{2})$ assign the number $\sum_{i = 1}^{k+1}x_{i}2^{i-1} \in \setn$. 

\begin{lemma}\label{ffielemma}
The harmonic cubic polynomial $P_{\B}$ associated with the Steiner triple system $\B$ determined by the lines in $\proj^{k}(\ffie_{2})$ is conformally equivalent to the simplicial polynomial $P_{n}$ defined in \eqref{simplicialpoly}, where $n = 2^{k+1} - 1$.
\end{lemma}
\begin{proof}
It is shown that $P_{\B}$ is conformally associative. The claim then follows from the uniqueness claim in Theorem \ref{confassequivalencetheorem}. 

Consider the vector space $\alg$ with basis $\{e_{i}: i \in \proj^{k}(\ffie_{2})\}$. Let $c_{ijk}$ be $1$ or $0$ as $\{ijk\}$ is or is not a block of $\B$. Equivalently $c_{ijk}$ equals $1$ if $i+j+k = 0$ in $\ffie_{2}^{k+1}$, and $0$ otherwise. Since two points determine a unique line, if $i,j,k,l \in \proj^{k}(\ffie_{2})$ are pairwise distinct, then
\begin{align}
\ass(P)(e_{i}, e_{j}, e_{k}, e_{l}) = c_{jk\,j+k}c_{il\,i+l} - c_{ik\,i+k}c_{jl\,j+l}.
\end{align}
Either of the products on the right-hand side is nonzero if and only if the two blocks it contains intersect. Since $i+k = j+l$ if and only if $j+k = i+l$ (add $i+j$ to both sides of either) and $i+k = j$ if and only if $j+k = i$, either bot $c_{ik\,i+k}c_{jl\,j+l}$ and $c_{jk\,j+k}c_{il\,i+l}$ equal $1$ or both equal $0$. In either case, $A(F)(e_{i}, e_{j}, e_{k}, e_{l}) = 0 = h(e_{i}, e_{l})h(e_{j}, e_{k}) - h(e_{j}, e_{l})h(e_{i}, e_{k})$.
If $i, j, k \in \proj^{k}(\ffie_{2})$ are pairwise distinct, then
\begin{align}
\ass(P)(e_{i}, e_{j}, e_{k}, e_{i}) =  - c_{ik\,i+k}c_{ij\,i + j} = - 1 = -\left(h(e_{i}, e_{i})h(e_{j}, e_{k}) - h(e_{j}, e_{i})h(e_{i}, e_{k})\right).
\end{align}
If $i,j  \in \proj^{k}(\ffie_{2})$ are distinct, then
\begin{align}
\ass(P)(e_{i}, e_{j}, e_{i}, e_{j}) =  c_{ji\,j+i}c_{ij\,i+j} = 1 = -\left(h(e_{i}, e_{j})h(e_{j}, e_{i}) - h(e_{j}, e_{j})h(e_{i}, e_{i})\right).
\end{align}
The preceding shows that $\ass(P)(x, y, z, w) = -\left(h(x, w)h(y, z) - h(y,w )h(x, z)\right)$ for all basis vectors $x, y, z, w$, and hence for all $x, y, v, w \in \alg$. This suffices to show that $\cass(P) = 0$, so that $P$ is conformally associative. 
\end{proof}
\end{example}

\begin{example}\label{dualaffineplaneexample}
\captionsetup[subfigure]{labelformat = parens, labelfont=rm,textfont=normalfont,singlelinecheck=off,justification=raggedright}
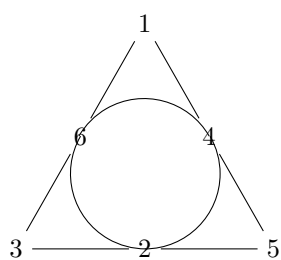
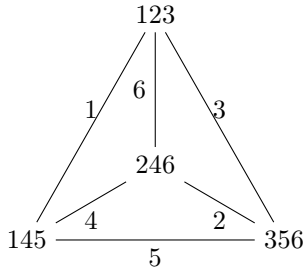
\begin{figure}[!ht]
\renewcommand{\thesubfigure}{\alph{subfigure}}
\begin{subfigure}{.45\textwidth}
\begin{tikzpicture}[baseline=(current  bounding  box.center), scale=.5,auto=left]
  \node (m1) at (0,4)  {$1$};
  \node (m2) at (-1.71, 1)  {$6$};
  \node (m3) at (-3.42,-2)  {$3$};
  \node (m4) at (1.71, 1)  {$4$};
 \node (m5) at (3.42, -2)  {$5$};
 \node (m6) at (0, -2)  {$2$};
\draw (m1) -- (m2);
\draw (m2) -- (m3);
\draw (m1) -- (m4);
\draw (m4) -- (m5) ;
\draw (m3) -- (m6);
\draw (m6) -- (m5) ;
\draw (m4) arc (30:150:2) (m2);
\draw (m2) arc (150:270:2) (m6);
\draw (m6) arc (-90:30:2) (m4);
\end{tikzpicture}  
\caption{Dual affine plane of order $2$}\label{dualaffineplane}
\end{subfigure}
\begin{subfigure}{.45\textwidth}
\begin{tikzpicture}[baseline=(current  bounding  box.center), scale=.5,auto=left]
\node (m1) at (0,4)  {123};
\node (m2) at (-3.42,-2)  {145};
\node (m3) at (3.42,-2)  {356};
\node (m4) at (0,0)  {246};
\draw (m1) -- (m2) node[midway, above] {$1$};
\draw (m2) -- (m3) node[midway, below]   {$5$};
\draw (m3) -- (m4) node[midway, below] {$2$};
\draw (m4) -- (m1) node[midway, left] {$6$};
\draw (m1) -- (m3) node[midway, above] {$3$};
\draw (m2) -- (m4) node[midway, below] {$4$};
\end{tikzpicture}  
\caption{Complete graph $K_{4}$}\label{K4graph}
\end{subfigure}
\caption{Partial Steiner triple system $\left\{123, 145, 246, 356\right\} \in PSTS(6)$}\label{psts6}
\end{figure}
The blocks of the partial Steiner triple system indicated in Figure \ref{psts6} are obtained from the dual affine plane of order $2$, as can be seen in Figure \ref{dualaffineplane}, or from 
the incidence relations of the edge set of the unique trivalent graph on four vertices, the complete graph $K_{4}$ on four vertices, as is indicated in Figure \ref{K4graph}. The corresponding cubic polynomial $P$ is that of \eqref{d2poly2}. 
The decomposition \eqref{d2poly2} together with Theorem \ref{3dpolytheorem} shows that $P$ is orthogonally decomposable, equivalent to a multiple of $P_{3}\oplus P_{3}$. By Theorem \ref{ealgpolynomialtheorem}, $P$ is not equivalent to $P_{6}$.
\end{example}

\begin{lemma}\label{stsextremelemma}
Let $P_{\B}$ be the cubic polynomial associated with a Steiner triple system $\B \in STS(n)$. The element $e = \tfrac{1}{\sqrt{n}}\sum_{i = 1}^{n}e_{i}$ is contained in $\extreme(P_{\B})$ and $\mkc(P_{\B}) =  \tfrac{n}{n-1}$.
\end{lemma}
\begin{proof}
For $i \neq j \in \setn$ let $i\join j \in \setn$ be the unique element such that $\{i, j, i\join j\} \in \B$. Write $P = P_{\B}$. By the definition of $P$, $(\hess P(e_{k})(e_{i}, e_{j})$ equals $1$ if $k = i \join j$ and $0$ otherwise. Hence 
\begin{align}
&(\hess P(e))(e_{i}, \dum) = h(e - \tfrac{1}{\sqrt{n}}e_{i}, \dum), & &(\hess P(e))(e, \dum) = \tfrac{n-1}{\sqrt{n}}h(e, \dum).
\end{align}
In particular $6P(e) = (\hess P(e))(e, e) = \tfrac{n-1}{\sqrt{n}}$. Because $\B$ is Steiner, its replication number $r$ satisfies $n -1 = 2r$.  By Lemma \ref{pstslemma}, $P$ solves \eqref{einsteinpolynomials} with constant $\ka = 2r = n-1$, so by \eqref{einsteinmkc} of Lemma \ref{mkclemma},
\begin{align}
\mkc(P) = \tfrac{\ka}{\left(\max_{x \in \sphere_{h}(1)}6P(x)\right)^{2}} \leq \tfrac{n-1}{36P(e)^{2}} = \tfrac{n}{n-1}.
\end{align}
Since by Lemma \ref{mkcboundlemma}, $\mkc(P) \geq \tfrac{n}{n-1}$, there must hold $\mkc(P) = \tfrac{n}{n-1}$, and moreover it must be the case that $e \in \extreme(P)$. 
\end{proof}

\begin{theorem}\label{ststheorem}
Let $P_{\B}$ be the cubic polynomial associated with a Steiner triple system $\B \in STS(n)$. If $n$ equals $1$ or $9$ modulo $12$ then $P_{\B}$ is not conformally linearly equivalent to the simplicial polynomial $P_{n}$, so is not conformally associative.
\end{theorem}

\begin{proof}
For $i \neq j \in \setn$, let $f = \tfrac{1}{\sqrt{3}}(e_{i}  + e_{j} + e_{i\join j})$. The $2dP(f) = (\hess P(f))(f, \dum) = \tfrac{2}{\sqrt{3}}h(f, \dum)$, so $f$ generates a critical line of $P$ and $6P(f) = (\hess P(f))(f, f) = \tfrac{2}{\sqrt{3}}$. 

By Corollary \ref{ndimexistencecorollary} the possible values of $6P_{n}(v)$ for unit norm $v$ generating a critical line of $P_{n}$ are of the form $(n+1-2k)\sqrt{\tfrac{n}{k(n+1-k)}}$ for $1\leq k \leq n$. Note that $P_{n}$ solves \eqref{einsteinpolynomials} with parameter $n(n-1)$, so $\tfrac{1}{\sqrt{n}}P_{n}$ solves \eqref{einsteinpolynomials} with parameter $n-1$ as does $P_{\B}$. Were $P_{\B}$ equivalent to $\tfrac{1}{\sqrt{n}}P_{n}$, then there would be $1 \leq k \leq n$ so that $\tfrac{2}{\sqrt{3}} = \tfrac{n+1-2k}{\sqrt{k(n+1-k)}}$. Squaring this yields $\tfrac{3(n-1)}{4} = \tfrac{k(n+1 -k)(n-1)}{(n+1 - 2k)^{2}}$, which holds if and only if $0 = (4k - 3(n+1))(4k - (n+1))$, so for $P_{\B}$ to be equivalent to $P_{n}$ a necessary condition is that $4$ divide $n+1$. The order $n$ of a Steiner triple system equals $1$ or $3$ modulo $6$. If $n$ equals $1$ or $9$ modulo $12$, then $(n+1)/2$ equals $1$ or $5$ modulo $6$, and in neither case is even, so it cannot be that $4$ divides $n+1$. (On the other hand, if $n$ equals $3$ or $7$ modulo $12$, then $n+1$ equals $4$ or $8$ modulo $12$, and $n+1$ is divisible by $4$.)
\end{proof}

\begin{example}
Fill a $3 \times 3$ grid with the numbers from $\bar{\mathsf{9}}$ and tile the plane with this grid. In each row, column, and diagonal there are exactly three distinct integers, and these are the $12$ blocks of \eqref{9124}. This is illustrated in Figure \ref{3affineplane}. There results the Steiner triple system $\B$ having format $(9, 12, 4)$ with blocks
\begin{align}\label{9124}
\B = \left\{012, 345, 678, 036, 147, 258, 057, 138, 246, 048, 156, 237 \right\}.
\end{align}
This is the \textit{affine plane of order $3$}.

By Theorem \ref{ststheorem} the cubic polynomial $P_{\B}$ associated with the Steiner triple system \eqref{9124} underlying the affine plane of order $3$ is not conformally equivalent to the simplicial polynomial $P_{9}$, so is not conformally associative.

\begin{figure}[!ht]
\begin{tikzpicture}[baseline=(current  bounding  box.center), scale=.5,auto=left]
\node[teal] (m0) at (0,0)  {$0$};
\node[teal] (m1) at (4,0)  {$1$};
\node[teal] (m2) at (8,0)  {$2$};
\node[teal] (m3) at (0,-4)  {$3$};
\node[teal] (m4) at (4, -4)  {$4$};
\node[teal] (m5) at (8, -4)  {$5$};
\node[teal] (m6) at (0, -8)  {$6$};
\node[teal] (m7) at (4, -8)  {$7$};
\node[teal] (m8) at (8, -8)  {$8$};
\node (m1l) at (-8,0)  {$1$};
\node (m4l) at (-8, -4)  {$4$};
\node (m7l) at (-8, -8)  {$7$};
\node (m2l) at (-4,0)  {$2$};
\node (m5l) at (-4, -4)  {$5$};
\node (m8l) at (-4, -8)  {$8$};
\draw[orange] (m0) -- (m1);
\draw[orange] (m1) -- (m2);
\draw[orange] (m3) -- (m4);
\draw[orange] (m4) -- (m5);
\draw[orange] (m6) -- (m7);
\draw[orange] (m7) -- (m8);
\draw[orange] (m0) -- (m3);
\draw[orange] (m1) -- (m4);
\draw[orange] (m2) -- (m5);
\draw[orange] (m6) -- (m3);
\draw[orange] (m7) -- (m4);
\draw[orange] (m8) -- (m5);
\draw[orange] (m0) -- (m4);
\draw[dashed] (m1) -- (m5);
\draw[orange] (m3) -- (m7);
\draw[orange] (m4) -- (m8);
\draw[orange] (m2) -- (m4);
\draw[orange] (m1) -- (m3);
\draw[dashed] (m5) -- (m7);
\draw[orange] (m4) -- (m6);
\draw[orange] (m0) -- (m5l);
\draw[orange] (m3) -- (m8l);
\draw[orange] (m3) -- (m2l);
\draw[orange] (m6) -- (m5l);
\draw[orange] (m5l) -- (m7l);
\draw[dashed] (m4l) -- (m2l);
\draw[orange] (m1l) -- (m5l);
\draw[dashed] (m4l) -- (m8l);
\draw[dashed] (m1l) -- (m2l);
\draw[dashed] (m2l) -- (m0);
\draw[dashed] (m4l) -- (m5l);
\draw[dashed] (m5l) -- (m3);
\draw[dashed] (m7l) -- (m8l);
\draw[dashed] (m8l) -- (m6);
\draw[dashed] (m1l) -- (m4l);
\draw[dashed] (m4l) -- (m7l);
\draw[dashed] (m2l) -- (m5l);
\draw[dashed] (m5l) -- (m8l);
\end{tikzpicture}  
\caption{Affine plane of order $3$}\label{3affineplane}
\end{figure}
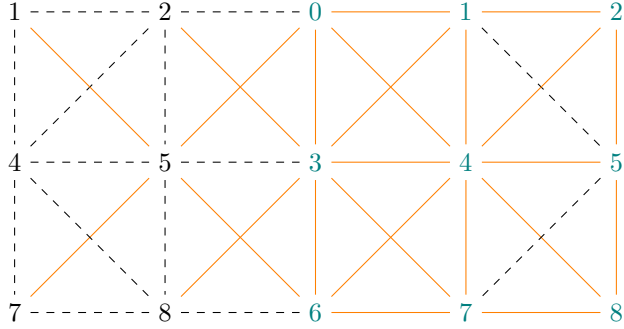
\end{example}

\begin{example}\label{pfaffexample}
For the Euclidean inner product $h(X, X) = -\tfrac{1}{2}\tr X^{2}$ on the space of $6 \times 6$ skew-symmetric matrices, the Pfaffian $P(X) = \pfaff X$ of $X \in \mathfrak{so}(6, \rea)$ is $h$-harmonic and solves $|\hess P|^{2}_{h} = 6|x|^{2}_{h}$. This can be justified as follows. Let $\alg$ be a $6$-dimensional real vector space equipped with a Riemannian metric $h$. Let $\{e_{1}, \dots, e_{6}\}$ be an ordered $h$-orthonormal basis. Let $\mu \in \ext^{6}\alg^{\ast}$ be the volume form equal to $1$ when evaluated on $\{e_{1}, \dots, e_{6}\}$. Let $e_{ij} = e_{i}\wedge e_{j} = e_{i}\tensor e_{j} - e_{j}\tensor e_{i}$. Let $h$ be the norm on $\ext^{2}\alg$ induced by declaring that $e_{ij}$ have unit $h$-norm (this is half of the norm given by complete contraction with $h$). Given $\al, \be \in \ext^{2}\alg$ define $\al \mlt \be \in \ext^{2}\alg$ to be the unique element such that
\begin{align}\label{pfaffmultdefined}
h(\al \mlt \be, \ga)\mu = \al \wedge \be \wedge \ga
\end{align}
for all $\ga \in \ext^{2}\alg$. From \eqref{pfaffmultdefined} it is immediate that $\mlt$ is commutative and $h(\al \mlt \be, \ga)$ is completely symmetric in $\al$, $\be$, and $\ga$. Let $P(\al) = \tfrac{1}{6}h(\al \mlt \al,\al)\in \pol^{3}(\ext^{2}\alg)$ be the associated cubic polynomial. The Pfaffian of $\al \in \ext^{2}\alg$ is defined by $6(\pfaff \al)\mu = \al \wedge \al \wedge \al$, and by \eqref{pfaffmultdefined}, this equals $h(\al \mlt \al, \al)\mu = 6P(\al)\mu$, so $P(\al) = \pfaff(\al)$. If $\al = \sum_{i < j}x_{ij}e_{ij}$, then
\begin{align}\label{pfaff66}
\begin{split}
\pfaff \al & = x_{12}x_{34}x_{56} + x_{12}x_{36}x_{45} - x_{12}x_{35}x_{46} - x_{13}x_{24}x_{56} + x_{13}x_{25}x_{46} \\
&\quad  - x_{13}x_{26}x_{45} + x_{14}x_{23}x_{56} - x_{14}x_{25}x_{36} + x_{14}x_{26}x_{35}- x_{15}x_{23}x_{46} \\
&\quad  + x_{15}x_{24}x_{36} - x_{15}x_{26}x_{34} + x_{16}x_{23}x_{45} - x_{16}x_{24}x_{35}+ x_{16}x_{25}x_{34}  .
\end{split}
\end{align}
If $x_{ij}$ are viewed as the cooordinates of a skew-symmetric matrix $X$ with respect to the basis $\{E_{ij} - E_{ji}: i < j\}$, where $E_{ij}$ is the elementary matrix with $1$ in row $i$ and column $j$ and $0$ elsewhere, then $h(\al, \al) = -\tfrac{1}{2}\tr X^{2} = \sum_{i < j}x_{ij}^{2}$ and $\pfaff \al = \pfaff X$. From \eqref{pfaff66} it is evident that $P$ is harmonic. That it satisfies \eqref{einsteinpolynomials} can be deduced by observing that is associated to a regular partial Steiner triple system, as follows.

This example has the form \eqref{partialsteinerpolynomial} for the regular partial Steiner triple system of format $(15, 15, 3)$, where $\bar{15}$ is identified with the two element subsets of $\{1, 2, 3, 4, 5, 6\}$, and the blocks of $\B$ are the $15$ possible partitions $\{(i_{1},j_{1}), (i_{2},j_{2}), (i_{3},j_{3})\}$ of $\{1, 2, 3, 4, 5, 6\}$ such that $i_{1} < i_{2} < i_{3}$ and $i_{p} < j_{p}$ for $p = 1, 2, 3$. Each such partition $I$ corresponds to a unique permutation of $\{1, 2, 3, 4, 5, 6\}$, and the coefficient $\ep_{I}$ is the sign of this permutation. 

This polynomial $P_{\B, \ep}$ equals \eqref{pfaff66}. It is the relative invariant of a real form of the reduced irreducible prehomogeneous vector space appearing as $(3)$ in table I of the classification of Sato-Kimura, \cite{Sato-Kimura}, namely for $GL(6, \rea)$ acting on skew-symmetric endomorphisms of the $3$-dimensional vector space $\alg$. Consequently, its automorphism group is large, containing $SL(6, \rea)$. The automorphisms that also preserve $h$ therefore contain $SO(h) = SO(6, \rea)$. This is apparent from the desciption above, for if $g \in O(h)$, then
\begin{align}
\begin{split}
h(g \cdot \al \mlt g \cdot \be, g \cdot \ga)\mu &= g\cdot \al \wedge g \cdot be \wedge g \cdot \ga = g\cdot (\al \wedge \be \wedge \ga) = h(\al\mlt \be, \ga)g \cdot \mu\\
&\quad  = \det(g)h(\al \mlt \be, \ga)\mu =\det(g) \al \wedge \be \wedge \ga\\
&\quad = \det(g)h(\al \mlt \be, \ga)\mu = \det(g)h(g\cdot (\al \mlt \be), g\cdot \ga)\mu,
\end{split}
\end{align}
shows that $g$ is an automorphism of $(\ext^{2}\alg, \mlt, g)$ if $g \in SO(h)$.
\end{example}

\begin{example}
Example \ref{pfaffexample} admits the following generalization. Let $\alg$ be a $6n$-dimensional real vector space equipped with a Riemannian metric $h$. Let $\{e_{1}, \dots, e_{6n}\}$ be an ordered $h$-orthonormal basis. Let $\mu \in \ext^{6n}\alg^{\ast}$ be the volume form equal to $1$ when evaluated on $\{e_{1}, \dots, e_{6n}\}$. For $I = \{i_{1} <  \dots < i_{n}\} \subset \overline{6n}$, let $e_{I} = e_{i_{1}}\wedge \dots \wedge e_{i_{n}}$. Let $h$ be the norm on $\ext^{2n}\alg$ induced by declaring that $e_{I}$ have unit $h$-norm (this is $(2n!)^{-1}$ times the norm given by complete contraction with $h$). Given $\al, \be \in \ext^{2n}\alg$ define $\al \mlt \be \in \ext^{2n}\alg$ to be the unique element such that
\begin{align}\label{pfaff2nmultdefined}
h(\al \mlt \be, \ga)\mu = \al \wedge \be \wedge \ga.
\end{align}
for all $\ga \in \ext^{2n}\alg$. From \eqref{pfaff2nmultdefined} it is immediate that $\mlt$ is commutative and $h(\al \mlt \be, \ga)$ is completely symmetric in $\al$, $\be$, and $\ga$. Let $P(\al) = \tfrac{1}{6}h(\al \mlt \al,\al) \in \pol^{3}(\ext^{2n}\alg)$ be the associated cubic polynomial. 

Let $\I$ be the set of cardinality $2n$ increasing subsets of $\overline{6n}$.
If $I\subset \I$, then $h(\al \mlt e_{I}, e_{I})\mu =\al \wedge e_{I} \wedge e_{I} = 0$, so $\lap_{h}P(\al) = \sum_{I}h(\al \mlt e_{I}, e_{I}) =0$, showing that $P$ is $h$-harmonic. If $I, J \subset \in \I$, differentiating $6P(\al) = h(\al \mlt \al, \al)$ yields $(\hess P(\al))(e_{I}, e_{J}) = h(\al \mlt e_{I}, e_{J})\mu = \al \wedge e_{I} \wedge e_{J}$. If $I \cap J$ is nonempty, then $e_{I} \wedge e_{J} =0$. If $I \cap J = \emptyset$, then there is a unique complementary increasing cardinality $2n$ subset $(I \cup J)^{c} = \overline{6n} \setminus (I \cup J) \in \I$. Hence, writing $\al_{K} = h(\al, e_{K})$ for $K \in \I$, 
\begin{align}
\begin{split}
|\hess P|^{2}_{h}(\al) &= \sum_{I, J \in \I}(\hess P(\al))(e_{I}, e_{J})^{2}= \sum_{I, J \in \I}h(\al \mlt e_{I}, e_{J})^{2} \\
&= \sum_{I, J \in \I: I \cap J = \emptyset}\al_{(I \cup J)^{c}}^{2} = \binom{2n}{n}\sum_{K\in I}\al_{K}^{2} =  \binom{2n}{n}|\al|_{h}^{2},
\end{split}
\end{align}
where the coefficient $\binom{2n}{n}$ appears as the number of ways of choosing disjoint $I, J \in \I$ such that $(I \cup J)^{c} = K$. This shows that $P$ and $h$ solve \eqref{einsteinpolynomials} on $\ext^{2n}\alg$ with $\ka = \binom{2n}{n}$.
\end{example}

\section{Cubic polynomial associated with a centered two-distance tight frame}\label{framesection}
Let $(\alg, h)$ be a Euclidean vector space. A collection $\frame$ of vectors in $(\alg, h)$ is a \emph{frame} if there are $A, B > 0$ such that
\begin{align}
&A|x|^{2}_{h} \leq \sum_{v \in \frame}h(x, v)^{2}  \leq B |x|^{2}_{h}, & &\text{for all}\,\, x \in \alg.
\end{align}
Since $\alg$ has finite dimension any spanning set is a frame with constants $A$ and $B$ equal to the minimum and maximum values of the sum $\sum_{v \in \frame}h(x, v)^{2}$ restricted to the $h$-unit sphere.
A frame is \emph{unit norm} if all its vectors have norm $1$. A frame $\frame$ is \emph{tight} if there is $M > 0$ such that
\begin{align}\label{tight1}
&\sum_{v \in \frame}h(x, v)^{2}  = M |x|^{2}, & &\text{for all}\,\,  x \in \alg.
\end{align}
The identity \eqref{tight1} is equivalent to $\sum_{v \in \frame}h(x, v)v  = Mx$ for all $x \in \alg$, for \eqref{tight1} can be rewritten as $Mh_{ij} = \sum_{v \in \frame}v_{i}v_{j}$, and raising the index $j$ yields $M\delta_{i}\,^{j} = \sum_{v \in \frame}v_{i}v^{j}$. Tracing $M\delta_{i}\,^{j} = \sum_{v \in \frame}v_{i}v^{j}$ shows that the frame constant of a tight frame satisfies
\begin{align}\label{frameconstant}
M= \tfrac{1}{n}\sum_{v \in \frame}|v|^{2}.
\end{align}
Such a frame is said to be \textit{$M$-tight}.
From \eqref{frameconstant} it follows that the frame constant of a cardinality $m$ unit norm tight frame is $M = \tfrac{m}{n}$.

A frame $\frame$ is \emph{centered} if its \textit{centroid} $\centroid = \centroid(\frame) = \tfrac{1}{m}\sum_{v \in \frame}v$ satisfies $\centroid = 0$. This terminology was introduced  in \cite{Fickus-Jasper-Mixon-Peterson-Watson} in the context of equiangular tight frames, but it makes sense for any finite frame.

A frame is \emph{two-distance} if there is $ \{c_{1}, c_{2}\} \subset \rea$ such that $h(u, v) \in \{c_{1}, c_{2}\}$ for all $u, v \in \frame$. A two-distance frame is \emph{equiangular} if $c_{1} = c = -c_{2}$. In this case, there are in fact two angles possible. The slightly misleading terminology \emph{equiangular} is well established; it refers to the fact that the angles between the lines generated by vectors of an equiangular frame are indeed all equal.

For a cardinality $m$ equiangular tight frame in a vector space of dimension $n \geq 2$ the equiangularity constant $c$ satisfies $c^{2} = \frac{m-n}{n(m-1)}$, for, if $u \in \frame$, then $1 + (m-1)c^{2} = \sum_{v \in \frame}h(u, v)^{2} = \tfrac{m}{n}|u|^{2} = \tfrac{m}{n}$. 

Lemma \ref{gorbitlemma} shows that the orbit of a nontrivial unit vector under the irreducible action of a finite group is a centered unit norm tight frame. 
\begin{lemma}\label{gorbitlemma}
Let $\rho:G \to \eno(\alg)$ be an irreducible $n$-dimensional representation of the finite group $G$ and let $h$ be a $\rho(G)$-invariant Euclidean metric. Then the orbit $\frame = \{\rho(g)v: g \in G\}$ of any $h$-unit norm $v \in \alg$ is a centered unit norm tight frame with frame constant $\tfrac{1}{n}\tfrac{|G|}{|G_{v}|}$ where $G_{v} = \{g \in G: \rho(g)v = v\}$ is the stabilizer of $v$. 
\end{lemma}
\begin{proof}
For $g \in G$, $G_{\rho(g)v} = gG_{v}g^{-1}$, so $|G_{\rho(g)v}| = |G_{v}|$. 
The endomorphism 
\begin{align}
\Phi(x) = \sum_{g \in G}h(\rho(g)v, x)\rho(g)v = |G_{v}|\sum_{u \in \frame}h(u, x)u 
\end{align}
of $\alg$ is $G$-equivariant, so is a multiple of the identity by the Schur Lemma. Because $h(\Phi(v), v) = \sum_{g\in G}h(\rho(g), v)^{2} \geq |v|^{2}_{h} = 1$, the multiplier is positive, so $\frame = \{\rho(g)v: g \in G\}$ is a tight frame. Since $|\frame| = |G|/|G_{v}|$, by \eqref{frameconstant}, the frame constant is $\tfrac{1}{n}{|G|}{|G_{v}|}$. The centroid $\centroid = |G|^{-1}\sum_{g \in G}\rho(g)v$ is $G$-invariant, so equals $0$ by the irreducibility of $\rho$.
\end{proof}
A result equivalent to Lemma \ref{gorbitlemma} is implicit in \cite{Vale-Waldron} and is stated as Lemma $2.3$ of \cite{Vale-Waldron-ginvariant} (the centeredness is not mentioned there). Note however that in \cite{Vale-Waldron, Vale-Waldron-ginvariant} the factor $|G_{v}|^{-1}$ is omitted. This apparent discrepancy is not an error and has the following explanation. Here the frame $\frame = \{\rho(g)v: g \in G\}$ is considered as the set of distinct images of $v$ under the action of $G$. In \cite{Vale-Waldron, Vale-Waldron-ginvariant} and many treatments of frames motivated by signal processing, a frame is regarded not as a set, but as a sequence (so as a map), so that what is considered in \cite{Vale-Waldron, Vale-Waldron-ginvariant} is not $\frame$ but the sequence $\{v_{g}: g \in G\}$ where $v_{g} = \rho(g)v$. In this sequence a given vector is repeated $|G_{v}|$ times. 

The convention used here is more natural for the contexts considered here, as is shown by taking $G = S_{n+1}$ acting on its $n$-dimensional irreducible representation, which yields the cardinality $n+1$ simplicial equiangular frame in an $n$-dimensional space as is apparent from Example \ref{simplicialframeexample}.

\begin{example}\label{simplicialframeexample}
The basic example of a centered equiangular tight frame is the \textit{simplicial} frame, defined as follows. Let $e_{i}$ be an orthonormal basis in Euclidean space $\rea^{n+1}$ and let $e = \sum_{i = 1}^{n+1}e_{i}$. The symmetric group $S_{n+1}$ acts on $\rea^{n+1}$ permuting the vectors $e_{i}$, and this action fixes $e$. The induced action of $S_{n+1}$ on the orthogonal complement $\ste = \{x \in \rea^{n+1}: \lb x, e\ra = 0\}$ is irreducible. Since $f_{i}= (n(n+1))^{-1/2}(e - (n+1)e_{i}) \in \ste$ is a multiple of the orthogonal projection of $e_{i}$ onto $\ste$ unimodular with respect to the induced metric on $\ste$, $\frame =  \{f_{i}:1 \leq i \leq n\}$ spans $\ste$, and the induced action of $S_{n+1}$ on $\ste$ is by permutations of $\frame$. Since $\lb f_{i}, f_{j}\ra = \tfrac{1-n}{n}$, $\frame$ is equiangular. Since $\frame$ can be viewed as the vertices of a simplex, it is called the \emph{simplicial frame}. Since $\frame$ is the $S_{n+1}$ orbit of $f_{1}$, Lemma \ref{gorbitlemma} implies that it is a centered unit norm tight frame with frame constant $(n+1)/n$. 
\end{example}

\begin{example}\label{icosahedralexample}
It would be useful to characterize for which $G$ and which representations $\rho$ the frame $\frame$ of Lemma \ref{gorbitlemma} is a two-distance frame. For example, for $\frame$ to be equiangular the equiangularity constant $c$ must satisfy $c^{2} = \tfrac{1}{n}\tfrac{|G|- n|G_{v}|}{|G| - |G_{v}|}$.

The alternating group $A_{5}$ can be identified with the orientation-preserving symmetries of a regular icosahedron on $\rea^{3}$. If the $12$ vertices of the icosahedron are taken to be the cyclic permutations of $(1 + \om^{2})^{1/2}(\pm 1, \pm \om, 0)$ where $\om = \tfrac{1 + \sqrt{5}}{2} = 1 + \om^{-1}$, then $A_{5}$ is generated by the order $2$ map $(x_{1}, x_{2}, x_{3}) \to (-x_{1}, -x_{2}, x_{3})$, the order $3$ cyclic permutations of the coordinates of $\rea^{3}$, and the rotation through an angle of $2\pi/5$ around the axis passing through any vertex, for example that given by the rotation with axis $(0, -1, \om)$:
\begin{align}
R = \tfrac{1}{2}\begin{pmatrix} -\om & 1 & \om - 1\\ -1 & 1-\om & -\om\\ 1-\om & -\om & 1 \end{pmatrix}.
\end{align}
The six vertices obtained by cyclically permuting $(1 + \om^{2})^{1/2}(1, \pm \om, 0)$ constitute an equiangular tight frame with $c = \tfrac{\om}{1 + \om^{2}}$. By Lemma \ref{gorbitlemma}, the $12$ vertices of the icosahedron constitute a centered tight frame with frame constant $4$. However, this frame is three-distance, not two-distance; with any vertex there are $5$ vertices with inner product $\tfrac{\om}{1 + \om^{2}}$, $5$ with inner product $-\tfrac{\om}{1 + \om^{2}}$, and $1$ vertex with inner product $-1$. 
\end{example}

For $v \in \alg$, the harmonic part of $\tfrac{1}{6}h(v, x)^{3}\in \pol^{3}(\alg)$ is 
\begin{align}\label{pvdefined}
P^{v}(x) = \tfrac{1}{6}\left(h(x, v)^{3} - \tfrac{3}{n+2}|x|^{2}|v|^{2}h(x, v)\right).
\end{align}
The polynomial $P^{v}$ vanishes on the hyperplane orthogonal to $v$, $P^{v}(x) = 0$ if $h(x, v) =0$, and satisfies $P^{v}(v) = \tfrac{n-1}{6(n+2)}|v|^{6}$, $P^{v}(u) = P^{u}(v)$ for all $u, v \in \alg$, and $P^{g\cdot v} = g\cdot P^{v}$ for all $g \in O(h)$. It follows from Corollary IV.2.13 of \cite{Stein-Weiss} that the restriction to the $h$-unit sphere of $\tfrac{(n+4)(n+2)n}{\om_{n}|v|^{3}}P^{v}$ is the zonal spherical harmonic of degree $3$ with pole $\tfrac{v}{|v|}$ (where $\om_{n}$ is the volume of the sphere).

\begin{definition}
The harmonic \emph{cubic polynomial} $P(x)$ associated with a frame $\frame$ in the $n$-dimensional Euclidean vector space $(\alg, h)$ is defined by
\begin{align}\label{framepolynomial}
\begin{split}
6P(x) &=\sum_{v\in \frame}P^{v}(x) = \sum_{v \in \frame}\left(h(x, v)^{3} - \tfrac{3}{n+2}|x|^{2}|v|^{2}h(x, v)\right)
\end{split}
\end{align}
\end{definition}
For a cardinality $m$ unit norm frame $\frame$ with centroid $\centroid$, \eqref{framepolynomial} takes the form
\begin{align}\label{unitframepolynomial}
\begin{split}
6P(x) &=\sum_{v\in \frame}P^{v}(x) = - \tfrac{3m}{n+2}|x|^{2}h(x, \centroid) + \sum_{v \in \frame}h(x, v)^{3}.
\end{split}
\end{align}
The automorphism group of a frame comprises those orthogonal transformations mapping the frame into itself. By construction the automorphism group of a frame acts as automorphisms of the cubic polynomial associated with the frame.

Differentiating \eqref{framepolynomial} shows that the $P$ associated with $\frame$ satisfies
\begin{align}\label{hesspv}
\begin{split}
2P(x)_{i} &  = 2\sum_{v \in \frame}P^{v}(x)_{i} = \sum_{v \in \frame}\left(\left(h(v, x)^{2} - \tfrac{1}{n+2}|v|^{2}|x|^{2}\right) v_{i} - \tfrac{2}{n+2}|v|^{2}h(v, x)x_{i} \right)\\
P(x)_{ij} & = \sum_{v \in \frame}P^{v}(x)_{ij} = \sum_{ v \in \frame}\left(h(x, v)v_{i}v_{j} - \tfrac{2}{n+2}|v|^{2}v_{(i}x_{j)} - \tfrac{1}{n+2}|v|^{2}h(x, v)h_{ij} \right).
\end{split}
\end{align}
For a cardinality $m$ centered unit norm frame $\frame$, \eqref{hesspv} becomes
\begin{align}\label{ctuf}
\begin{split}
2P(x)_{i} &  = \sum_{v \in \frame}\left(\left(h(v, x)^{2} - \tfrac{1}{n+2}|x|^{2}\right) v_{i} - \tfrac{2}{n+2}h(v, x)x_{i} \right)= \sum_{v \in \frame}h(v, x)^{2}v_{i},\\
P(x)_{ij} & = \sum_{ v \in \frame}\left(h(x, v)v_{i}v_{j} - \tfrac{2}{n+2}v_{(i}x_{j)} - \tfrac{1}{n+2}h(x, v)h_{ij} \right) =  \sum_{ v \in \frame}h(x, v)v_{i}v_{j} .
\end{split}
\end{align}
Hence, for a cardinality $m$ centered unit norm tight frame there holds
\begin{align}\label{cutfhessp}
\begin{split}
P^{ij}P_{ij}
& =\tfrac{m}{n}|x|^{2}+ \sum_{v \in \frame}\sum_{\frame \ni u \neq v}h(u, v)^{2}h(u, x)h(v, x) .
\end{split}
\end{align}

\begin{theorem}\label{twodistancetheorem}
In a Euclidean vector space $(\alg, h)$ of dimension $n \geq 2$, let $\frame$ be a cardinality $m$ centered two-distance tight frame such that $h(u, v) \in  \{c_{1}, c_{2}\}$ for all $u \neq v \in \frame$. The harmonic cubic polynomial $P(x)$ associated with $\frame$ as in \eqref{framepolynomial} solves
\begin{align}\label{ctdtfhessp}
\begin{split}
|\hess P|^{2} & = \tfrac{m}{n}\left(1 + (c_{1} + c_{2})(\tfrac{m}{n} - 1) + c_{1}c_{2}\right)|x|^{2}\\
& = \tfrac{m}{n}\left((c_{1} - 1)(c_{2} - 1) + \tfrac{m}{n}c_{1}c_{2}\right)|x|^{2}.
\end{split}
\end{align}
If $\frame$ is a centered equiangular tight frame with $c_{1} = c = -c_{2}$, then 
\begin{align}\label{cetfhessp}
\begin{split}
|\hess P|^{2} & = \tfrac{m(1-c^{2})}{n}|x|^{2} = \tfrac{m^{2}(n-1)}{n^{2}(m-1)}|x|^{2},
\end{split}
\end{align}
where $|h(u, v)| = c= \sqrt{\tfrac{m-n}{n(m-1)}}$ for $u \neq v \in \frame$. 
\end{theorem}

\begin{proof}
Suppose $\frame$ is a cardinality $m$ centered $k$-distance tight frame such that $h(u, v) \in \C = \{c_{1}, \dots, c_{k}\}$ for all $u, v \in \frame$, $u \neq v$. For $1 \leq i \leq k$, define $\G_{i} = \{(u, v) \in \frame \times \frame: h(u, v) = c_{i}\}$.
Note that for any quantity $\phi(u, v)$ there holds 
\begin{align}
\sum_{u \in \frame}\sum_{\frame \ni v \neq u}\phi(u, v) = \sum_{i = 1}^{k}\sum_{(u, v) \in \G_{i}} \phi(u, v).
\end{align}
Because $\frame$ is centered, unit norm, and tight there holds
\begin{align}\label{tdeq1}
\begin{split}
0 & = \left(\sum_{u \in \frame}h(u, x)\right)^{2} = \sum_{u \in \frame}h(u, x)^{2} + \sum_{i = 1}^{k}\sum_{(u, v) \in \G_{i}} h(u, x)h(v, x) \\
&= \tfrac{m}{n}|x|^{2} +  \sum_{i = 1}^{k}\sum_{(u, v) \in \G_{i}} h(u, x)h(v, x).
\end{split}
\end{align}
for any $x \in \alg$. On the other hand, pairing $\tfrac{m}{n}x = \sum_{u \in \frame}h(u, x)u$ with itself yields
\begin{align}\label{tdeq2}
\begin{split}
\tfrac{m^{2}}{n^{2}}|x|^{2} & = \sum_{u \in \frame}\sum_{v \in \frame}h(u, v)h(u, x)h(v, x) = \sum_{u \in \frame}h(u, x)^{2} +\sum_{i = 1}^{k} \sum_{(u, v)\in \G_{i}}h(u, v)h(u, x)h(v, x) \\
&= \tfrac{m}{n}|x|^{2} + \sum_{i = 1}^{k}c_{i}\sum_{(u, v)\in \G_{i}}h(u, x)h(v, x).
\end{split}
\end{align}
From \eqref{tdeq1} and \eqref{tdeq2} it follows that the $k$ quantities 
\begin{align}
\Sigma_{i} =\sum_{(u, v)\in \G_{i}}h(u, x)h(v, x).
\end{align}
solve the system of equations
\begin{align}\label{ksigmaeqs}
&\Sigma_{1} + \dots + \Sigma_{k} = -\tfrac{m}{n}|x|^{2},&& c_{1}\Sigma_{1} + \dots + c_{k}\Sigma_{k} = \tfrac{m}{n}\left(\tfrac{m}{n} - 1\right)|x|^{2}.
\end{align}
When $k = 2$ \eqref{ksigmaeqs} becomes the $2\times 2$ system
\begin{align}
&\Sigma_{1} + \Sigma_{2} = -\tfrac{m}{n}|x|^{2},&& c_{1}\Sigma_{1} + c_{2}\Sigma_{2} = \tfrac{m}{n}\left(\tfrac{m}{n} - 1\right)|x|^{2}.
\end{align}
Solving these equations yields
\begin{align}\label{tdsol}
\begin{split}
\sum_{(u, v)\in \G_{1}}h(u, x)h(v, x) & = \tfrac{1}{c_{1} - c_{2}}\tfrac{m}{n}\left(\tfrac{m}{n} - 1 + c_{2}\right)|x|^{2},\\
\sum_{(u, v)\in \G_{2}}h(u, x)h(v, x) & = -\tfrac{1}{c_{1} - c_{2}}\tfrac{m}{n}\left(\tfrac{m}{n} - 1 + c_{1}\right)|x|^{2}.
\end{split}
\end{align}
Substituting \eqref{tdsol} into \eqref{cutfhessp} and simplifying the result yields
\begin{align}\label{tdhessp}
\begin{split}
P^{ij}P_{ij}
& =\tfrac{m}{n}|x|^{2}+ \sum_{v \in \frame}\sum_{\frame \ni u \neq v}h(u, v)^{2}h(u, x)h(v, x) \\
& = \tfrac{m}{n}|x|^{2} + c_{1}^{2}\sum_{(u, v)\in \G_{1}}h(u, x)h(v, x) + c_{2}^{2}\sum_{(u, v)\in \G_{2}}h(u, x)h(v, x)\\
& = \tfrac{m}{n}\left(1 + (c_{1} + c_{2})(\tfrac{m}{n} - 1) + c_{1}c_{2}\right)|x|^{2},
\end{split}
\end{align}
proving \eqref{ctdtfhessp}.
\end{proof}

\begin{lemma}
In a Euclidean vector space $(\alg, h)$ of dimension $n \geq 2$, let $\frame$ be a cardinality $m$ centered equiangular tight frame such that $|h(u, v)| = c$ for $u \neq v \in \frame$. 
For $u \in \frame$, $[u] \in \critline(P) \setminus \zero(P)$ and the element $e = (1- c^{2})^{-1}u = \tfrac{n(m-1)}{m(n-1)}u$ satisfies $6P(e) = |e|^{2}_{h}$ and 
\begin{align}\label{frameleft}
(\hess P)(e)_{ij} = \sum_{v \in \frame}h(e, v)v_{i}v_{j} = (1-c^{2})e_{i}e_{j} + \sum_{\frame \ni v \neq u}h(e, v)v_{i}v_{j}.
\end{align}
\end{lemma}
\begin{proof}
By \eqref{ctuf}, for a centered equiangular tight frame $\frame$, if $u \in \frame$, then $2P(u)_{i} = (1-c^{2})u_{i}$, so $u$ is a critical point of the restriction of $P$ to the $h$-unit sphere. Because $\frame$ is centered,
\begin{align}
\begin{split}
6P(u) &= \sum_{v \in \frame}h(u, v)^{3} = 1 + c^{2}\sum_{\frame \ni v \neq u}h(u, v)  = 1- c^{2} =\tfrac{m(n-1)}{n(m-1)},
\end{split}
\end{align}
so that $6P(e) = \left(\tfrac{n(m-1)}{m(n-1)}\right)^{2} = |e|^{2}_{h}$.
The expression \eqref{frameleft} follows from \eqref{ctuf}.
\end{proof}

\begin{example}\label{simplicialframepolynomialexample}
This example justifies calling \emph{simplicial} the polynomial $P_{n}$ of Theorem \ref{ealgpolynomialtheorem}. It also illustrates that checking whether the polynomial associated with a frame as in Theorem \ref{twodistancetheorem} is equivalent to a given polynomial is nontrivial.
\begin{lemma}\label{pncharlemma}
The cubic polynomial associated with the simplicial frame described in Example \ref{simplicialframeexample} is equivalent to the simplicial polynomial $P_{n}$ defined in \eqref{simplicialpoly} of Theorem \ref{ealgpolynomialtheorem}.
\end{lemma}

\begin{proof}
Let $\frame$ be a cardinality $n+1$ unit frame in $\rea^{n}$ such that $h(u, v) = -\tfrac{1}{n}$ for all $u \neq v \in \frame$. Such a frame exists by Example \ref{simplicialframeexample} and any two such frames are orthogonally equivalent because they have the same Gram matrix. Let $P = P^{\frame}$ be the cubic polynomial associated with $\frame$. By Theorem \ref{twodistancetheorem}, $P$ solves \eqref{einsteinpolynomials} with constant $\ka = \tfrac{(n+1)^{2}(n-1)}{n^{3}}$. The claim follows from Theorems \ref{ealgpolynomialtheorem} and \ref{confassequivalencetheorem} once it is shown that $P$ is conformally associative. Let $L(x)_{i}\,^{j} = P(x)_{ip}h^{pj}$. By \eqref{ctuf},
\begin{align}
\begin{split}
L(x)_{p}\,^{j}L(y)_{i}\,^{p} & = \sum_{u \in \frame}h(x, u)\left(\sum_{u \neq v \in \frame}h(y, v)h(u, v)v_{i}\right)u^{j} + \sum_{u \in \frame}h(x, u)h(y, u)u_{i}u^{j}\\
& = -\tfrac{1}{n}\sum_{u \in \frame}h(x, u)\left(\sum_{u \neq v \in \frame}h(y, v)v_{i}\right)u^{j} + \sum_{u \in \frame}h(x, u)h(y, u)u_{i}u^{j}\\
&= -\tfrac{1}{n}\sum_{u \in \frame}h(x, u)\left(\tfrac{n+1}{n}y_{i} - h(y, u)u_{i}\right)u^{j} + \sum_{u \in \frame}h(x, u)h(y, u)u_{i}u^{j}\\
&= -\tfrac{1}{n}\sum_{u \in \frame}h(x, u)\left(\tfrac{n+1}{n}y_{i}\right)u^{j} + \tfrac{n+1}{n} \sum_{u \in \frame}h(x, u)h(y, u)u_{i}u^{j}\\
&= -\tfrac{(n+1)^{2}}{n^{3}}y_{i}x^{j} + \tfrac{n+1}{n} \sum_{u \in \frame}h(x, u)h(y, u)u_{i}u^{j},
\end{split}
\end{align}
where the second equality is special to the particular frame considered. It follows that
\begin{align}
\begin{split}
[L(x), L(y)]_{ij} & =  \tfrac{2(n+1)^{2}}{n^{3}}x_{[i}y_{j]} = \tfrac{2\ka}{n-1}x_{[i}y_{j]},
\end{split}
\end{align}
and by \eqref{casscommutator} this shows that $\cass(P) = 0$, so that $P$ is conformally associative.
\end{proof}

Together Lemmas \ref{zeroautomorphismlemma} and \ref{pncharlemma} imply that the autormorphism group of $P_{n}$ can be realized concretely as that generated by the reflections through the hyperplanes orthogonal to the pairwise differences of the elements of the equiangular simplicial frame. Let $P$ be the cubic polynomial associated with the simplicial $\frame$ as in the proof of Lemma \ref{pncharlemma}. Then $6P(v) = \sum_{u \in \frame}h(u, v)^{3} = 1 - n^{-2}$ for all $v \in \frame$, so, if $v \neq w \in \frame$, then 
\begin{align}
\begin{split}
P(v-w) &= \sum_{u \in \frame}h(v-w, u)^{3} \\
&= P(v) - P(w) + 3\sum_{u \in \frame\setminus\{v, w\}}h(u, v)h(u, w)\left(h(v, u) - h(w, u)\right) = 0,
\end{split}
\end{align}
so $v - w \in \zero(P)$. Similarly,
\begin{align}
\begin{split}
\lb dP(x), v - w\ra &= \sum_{u \in \frame}h(x, u)^{2}h(v-w, u) = \sum_{u \in\setminus\{v, w\}}h(x, u)^{2}h(v-w, u) = 0,
\end{split}
\end{align}
for all $x \in \ste$. By Lemma \ref{zeroautomorphismlemma} the reflection through the hyperplane orthogonal to $v - w$ is an automorphism of $P$. Since these reflections act on $\frame$ as transpositions, they generate a group isomorphic to the symmetric group $S_{n+1}$. Since the automorphisms of $P$ permute $\frame$, this is the full automorphism group of $P$.
\end{example}

\begin{example}
The cardinality $6$ equiangular frame described in Example \ref{icosahedralexample} is not centered, so Theorem \ref{twodistancetheorem} does not apply to this example.
\end{example}

\begin{example}\label{d2poly2frameexample}
This example shows that different centered two-distance tight frames can give rise to the same cubic polynomial. In the construction here the different frames arise as generators of different subsets of the set $\critline(P)$ of critical lines of $P$.

The $8$ columns of
\begin{align}
\frac{1}{\sqrt{6}}
\begin{pmatrix*}[r]
1   &  1 &   -1&    -1 &    1 &    1 &   -1 &    -1\\ 
     1    &-1&     1 &   -1&     1&    -1&     1&    -1\\
     1    &-1  &  -1    & 1   &  1  &  -1  &  -1  &   1\\
     1    &-1   & -1     &1   & -1   &  1   &  1   & -1\\
     1    &-1 &    1  &  -1  &  -1 &    1 &   -1 &    1\\
   1   &  1   & -1   & -1   & -1   & -1    & 1   &  1\\
\end{pmatrix*}
\end{align}
are a centered two-distance tight frame in $\rea^{6}$ with $c_{1} = 0$ and $c_{2} = -1/3$. 
The associated cubic polynomial is $\tfrac{8}{\sqrt{6}}(x_{1}x_{2}x_{3} + x_{1}x_{4}x_{5} + x_{2}x_{4}x_{6} + x_{3}x_{5}x_{6})$, a constant multiple of \eqref{d2poly2}.

The $16$ columns of 
\begin{align}\label{16colexample}
\setcounter{MaxMatrixCols}{20}
\frac{1}{\sqrt{3}}
\begin{pmatrix*}[r]
1 & -1 & 1 & -1 & 1 & -1 & 1 & -1 & 0 & 0 & 0 & 0 & 0 & 0 & 0 & 0\\
1 & 1 &-1 &-1 & 0 & 0 & 0 & 0 & 1 & 1 &-1 &-1& 0 & 0 & 0 & 0\\
1 &-1&-1 & 1 &  0 & 0 & 0 & 0 & 0 & 0 & 0 & 0 & 1 &-1 &-1 &1 \\
0 & 0 & 0 & 0 & 1 & -1 &-1 &1 &1 & -1&-1& 1   & 0 & 0 & 0 & 0 \\
0 & 0 & 0 & 0 & 1 & 1 &-1 &-1 & 0 & 0 & 0 & 0 & 1 & 1 &-1 &-1\\
0 & 0 & 0 & 0 & 0 & 0 & 0 & 0 & 1 & -1 & 1 & -1 & 1 & -1 & 1 & -1\\
\end{pmatrix*}
\end{align}
are a centered equiangular tight frame in $\rea^{6}$ with $c = 1/3$. (Example \eqref{16colexample} is not new; it was given at the end of section $3$ of \cite{Fickus-Jasper-Mixon-Peterson-Watson}, where it was obtained from purely combinatorial considerations.) The associated cubic polynomial is
\begin{align}
\begin{split}
18\sqrt{3}P(x) & = (x_{1} + x_{2} + x_{3})^{3} + (-x_{1} + x_{2} - x_{3})^{3}+ (x_{1} - x_{2} - x_{3})^{3}+ (-x_{1} - x_{2} + x_{3})^{3}\\
&\quad +  (x_{1} + x_{4} + x_{5})^{3} + (-x_{1} + x_{4} - x_{5})^{3}+ (x_{1} - x_{4} - x_{5})^{3}+ (-x_{1} - x_{4} + x_{5})^{3}\\
&\quad +  (x_{2} + x_{4} + x_{6})^{3} + (-x_{2} + x_{4} - x_{6})^{3}+ (x_{2} - x_{4} - x_{6})^{3}+ (-x_{2} - x_{4} + x_{6})^{3}\\
&\quad +  (x_{3} + x_{5} + x_{6})^{3} + (-x_{3} + x_{5} - x_{6})^{3}+ (x_{3} - x_{5} - x_{6})^{3}+ (-x_{3} - x_{5} + x_{6})^{3}\\
& = 24(x_{1}x_{6}x_{2} + x_{1}x_{5}x_{3} + x_{6}x_{5}x_{4} + x_{2}x_{3}x_{4}),
\end{split}
\end{align}
and $P$ solves \eqref{einsteinpolynomials} with $\ka = 64/27$. 
\end{example}

\begin{example}\label{728example}
The $28$ possible permutations of 
\begin{align}
\tfrac{1}{\sqrt{24}}\begin{pmatrix*}[c] -3 & -3 & 1 & 1 & 1 & 1 & 1 & 1\end{pmatrix*}^{t}
\end{align}
span the $7$-dimensional subspace $\stw = \{x \in \rea^{8}: h(x, e) = 0\}$ orthogonal to $e = e_{1} + \dots + e_{8}$, and it can be checked that they constitute a centered equiangular tight frame $\frame$ in $\stw$ having frame constant $4$ and equiangularity constant $1/3$.  Alternatively, these claims follow from Lemma \ref{gorbitlemma}, as this frame is the orbit of a unit vector in the $7$-dimensional irreducible representation of $S_{8}$ having stabilizer $S_{6}\times S_{2}$. The associated $7$-variable harmonic cubic polynomial $P$ solves \eqref{einsteinpolynomials} with $\ka = 32/9$. The elements of $\frame$ are the unit vectors in the directions of the pairwise sums of distinct elements of the $8$ element equiangular simplicial frame in $\stw$, and, as in Example \ref{d2poly2frameexample}, $\frame$ and the simplicial frame determine multiples of the same cubic polynomial (in this case it is the simplicial polynomial $P_{7}$).
\end{example}

\bibliographystyle{amsplain}
\def\polhk#1{\setbox0=\hbox{#1}{\ooalign{\hidewidth
  \lower1.5ex\hbox{`}\hidewidth\crcr\unhbox0}}} \def\cprime{$'$}
  \def\cprime{$'$} \def\cprime{$'$}
  \def\polhk#1{\setbox0=\hbox{#1}{\ooalign{\hidewidth
  \lower1.5ex\hbox{`}\hidewidth\crcr\unhbox0}}} \def\cprime{$'$}
  \def\cprime{$'$} \def\cprime{$'$} \def\cprime{$'$} \def\cprime{$'$}
  \def\polhk#1{\setbox0=\hbox{#1}{\ooalign{\hidewidth
  \lower1.5ex\hbox{`}\hidewidth\crcr\unhbox0}}} \def\cprime{$'$}
  \def\Dbar{\leavevmode\lower.6ex\hbox to 0pt{\hskip-.23ex \accent"16\hss}D}
  \def\cprime{$'$} \def\cprime{$'$} \def\cprime{$'$} \def\cprime{$'$}
  \def\cprime{$'$} \def\cprime{$'$} \def\cprime{$'$} \def\cprime{$'$}
  \def\cprime{$'$} \def\cprime{$'$} \def\cprime{$'$} \def\dbar{\leavevmode\hbox
  to 0pt{\hskip.2ex \accent"16\hss}d} \def\cprime{$'$} \def\cprime{$'$}
  \def\cprime{$'$} \def\cprime{$'$} \def\cprime{$'$} \def\cprime{$'$}
  \def\cprime{$'$} \def\cprime{$'$} \def\cprime{$'$} \def\cprime{$'$}
  \def\cprime{$'$} \def\cprime{$'$} \def\cprime{$'$} \def\cprime{$'$}
  \def\cprime{$'$} \def\cprime{$'$} \def\cprime{$'$} \def\cprime{$'$}
  \def\cprime{$'$} \def\cprime{$'$} \def\cprime{$'$} \def\cprime{$'$}
  \def\cprime{$'$} \def\cprime{$'$} \def\cprime{$'$} \def\cprime{$'$}
  \def\cprime{$'$} \def\cprime{$'$} \def\cprime{$'$} \def\cprime{$'$}
  \def\cprime{$'$} \def\cprime{$'$} \def\cprime{$'$}
\providecommand{\bysame}{\leavevmode\hbox to3em{\hrulefill}\thinspace}
\providecommand{\MR}{\relax\ifhmode\unskip\space\fi MR }
\providecommand{\MRhref}[2]{%
  \href{http://www.ams.org/mathscinet-getitem?mr=#1}{#2}
}
\providecommand{\href}[2]{#2}

\end{document}